\documentclass[a4paper,10pt]{article}
\usepackage[utf8]{inputenc}
\usepackage[left=2.5cm, right=2.5cm, top=2cm]{geometry}
\usepackage{subfigure}
\usepackage{float}
\usepackage{amsmath}
\usepackage{mathtools}

\usepackage[]{algorithm2e}
\usepackage{bbm} 
\usepackage{dsfont} 

\usepackage{commath} 

\usepackage{amsmath, amsthm, amssymb}
\usepackage{mathabx} 
\usepackage{mathrsfs}
\usepackage{url}
\usepackage{amssymb,amsfonts}
\usepackage{color}
\usepackage{shuffle}
\usepackage[usenames,dvipsnames,table]{xcolor}

\usepackage[linecolor=white,backgroundcolor=white,bordercolor=white,textsize=tiny]{todonotes}
\usepackage{breqn}

\usepackage{silence}
\usepackage[style=alphabetic,sorting=anyt,natbib,maxbibnames=99]{biblatex}
\usepackage{hyperref}
\hypersetup{
     colorlinks=true,
     linkcolor=blue,
     filecolor=blue,
     citecolor=cyan,      
     urlcolor=cyan,
     }
\usepackage[capitalize]{cleveref}
\usepackage{graphicx}
\usepackage[font={small,it}]{caption}
\usepackage[flushleft]{threeparttable}
\usepackage{longtable}
\usepackage[shortlabels]{enumitem}

\let\underbar\underline

\newtheorem{theorem}{Theorem}[section]
\theoremstyle{plain}

\newtheorem{corollary}[theorem]{Corollary}

\newtheorem{example}[theorem]{Example}

\newtheorem{lemma}[theorem]{Lemma}

\newtheorem{proposition}[theorem]{Proposition}
\newtheorem{remark}[theorem]{Remark}

\theoremstyle{definition} 

\newtheorem{definition}[theorem]{Definition}
\newtheorem*{tata}{Generalization}
  {\begin{mdframed}[backgroundcolor=lightgray]\begin{tata}}%
  {\end{tata}\end{mdframed}}

\newcommand{\N}{\mathbb{N}}

\newcommand{\Q}{\mathbb{Q}}

\newcommand{\im}{\operatorname{im}}

\newcommand{\Hom}{\operatorname{Hom}}

\newcommand{\id}{\operatorname{id}}

\newcommand{\mute}[1]{}
\let\todon\todo
\renewcommand{\todo}[1]{\todon{\color{red}#1}}

\newcommand{\evaluatedAt}[1]{\,\raisebox{-.5em}{$\vert_{#1}$}}

\newcommand{\lrb}[1]{{\left \lbrace { {#1} } \right \rbrace }}
\newcommand{\lrp}[1]{{\left ( { {#1} } \right ) }}
\newcommand{\lrv}[1]{{\left | { {#1} } \right | }}
\newcommand{\lra}[1]{{\left \langle { {#1} } \right \rangle }}

\usepackage{float}
\usepackage{scalerel}
\usepackage{tikz-cd}
\tikzcdset{scale cd/.style={every label/.append style={scale=#1},
    cells={nodes={scale=#1}}}}
\usetikzlibrary{matrix,arrows,backgrounds}
\newif\ifshowgraphs
\showgraphsfalse
\showgraphstrue

\ifshowgraphs

\newcommand{\DRAWgraphVERTEX}{
\raisebox{+.1em}{
\begin{tikzpicture}[scale=.75]
  \fill
    (0em,-1em) circle [radius=.2em];
\end{tikzpicture}}
}
\newbox{\graphVERTEX}
\sbox{\graphVERTEX}{\hspace{-.4em}\DRAWgraphVERTEX\hspace{-.2em}}
\newcommand{\vertex}{{\usebox{\graphVERTEX}}}

\newcommand{\DRAWgraphEDGE}{{
\raisebox{-.2em}{
\begin{tikzpicture}[scale=.75]
    \path[very thick]
    (0em,-1em) edge (0,-.334em);
    \fill
    (0em,-1em) circle [radius=.2em]
    (0em,-.334em) circle [radius=.2em];
\end{tikzpicture}
}
}
}
\newbox{\graphEDGE}
\sbox{\graphEDGE}{{\hspace{-.2em}\DRAWgraphEDGE\hspace{-1em}}}
\newcommand{\edge}{{\usebox{\graphEDGE}}}

\newcommand{\DRAWgraphCHERRY}{
\raisebox{-.2em}{
\begin{tikzpicture}[scale=.75]
  \path[very thick]
    (0em,-1em) edge (0.5em,-.134em)
    (0.5em,-.134em) edge (1em,-1em);
  \fill
    (0em,-1em) circle [radius=.2em]
    (0.5em,-.134em) circle [radius=.2em]
    (1em,-1em) circle [radius=.2em];
\end{tikzpicture}}
}
\newbox{\graphCHERRY}
\sbox{\graphCHERRY}{{\hspace{-.4em}\DRAWgraphCHERRY\hspace{-.2em}}}
\newcommand{\cherry}{{\usebox{\graphCHERRY}}}

\newcommand{\DRAWgraphTRIANGLE}{
\raisebox{-.2em}{
\begin{tikzpicture}[scale=.75]
  \path[very thick]
    (0em,-1em) edge (0.5em,-.134em)
    (0.5em,-.134em) edge (1em,-1em)
    (1em,-1em) edge (0em,-1em);
  \fill
    (0em,-1em) circle [radius=.2em]
    (0.5em,-.134em) circle [radius=.2em]
    (1em,-1em) circle [radius=.2em];
\end{tikzpicture} }
}
\newbox{\graphTRIANGLE}
\sbox{\graphTRIANGLE}{{\hspace{-.6em}\DRAWgraphTRIANGLE\hspace{-.2em}}}
\newcommand{\tria}{{\usebox{\graphTRIANGLE}}}

\newcommand{\DRAWgraphTAILEDtriangle}{
\raisebox{-.2em}{
\begin{tikzpicture}[scale=.9]
  \draw[very thick]
	(0.0em,0em) -- (0.0em,-1em)
	(0.0em,-1em) -- (1.0em,-1em)
	(1.0em,-1em) -- (1.0em,0em)
	(1.0em,0em) -- (0.0em,-1em);
\fill (0.0em,0em) circle [radius=.2em];
\fill (0.0em,-1em) circle [radius=.2em];
\fill (1.0em,-1em) circle [radius=.2em];
\fill (1.0em,0em) circle [radius=.2em];
\end{tikzpicture}
}
}
\newbox{\graphTAILEDtriangle}
\sbox{\graphTAILEDtriangle}{{\hspace{-.4em}\DRAWgraphTAILEDtriangle\hspace{-.4em}}}
\newcommand{\tailedTriangle}{{\usebox{\graphTAILEDtriangle}}}

\newcommand{\DRAWgraphSWING}{
\raisebox{-.2em}{
\begin{tikzpicture}[scale=.9]
  \draw[very thick]
	(0.0em,0em) -- (0.0em,-1em)
	(0.0em,-1em) -- (1.0em,-1em)
	(1.0em,-1em) -- (1.0em,0em);
\fill (0.0em,0em) circle [radius=.2em];
\fill (0.0em,-1em) circle [radius=.2em];
\fill (1.0em,-1em) circle [radius=.2em];
\fill (1.0em,0em) circle [radius=.2em];
\end{tikzpicture}
}
}
\newbox{\graphSWING}
\sbox{\graphSWING}{{\hspace{-.4em}\DRAWgraphSWING}}
\newcommand{\threeLadder}{{\usebox{\graphSWING}}}

\newcommand{\DRAWgraphCLAW}{
\raisebox{-.2em}{
\begin{tikzpicture}[scale=.9]
  \draw[very thick]
	(0.0em,0em) -- (0.0em,-1em)
	(0.0em,-1em) -- (1.0em,-1em)
	(1.0em,0em) -- (0.0em,-1em);
\fill (0.0em,0em) circle [radius=.2em];
\fill (0.0em,-1em) circle [radius=.2em];
\fill (1.0em,-1em) circle [radius=.2em];
\fill (1.0em,0em) circle [radius=.2em];
\end{tikzpicture}
}
}
\newbox{\graphCLAW}
\sbox{\graphCLAW}{{\hspace{-.4em}\DRAWgraphCLAW\hspace{-.6em}}}
\newcommand{\threeStar}{{\usebox{\graphCLAW}}}

\newcommand{\DRAWgraphSQUARE}{
\raisebox{-.2em}{
\begin{tikzpicture}[scale=.9]
  \draw[very thick]
	(0.0em,0em) -- (0.0em,-1em)
	(0.0em,-1em) -- (1.0em,-1em)
	(1.0em,-1em) -- (1.0em,0em)
	(1.0em,0em) -- (0.0em,0em);
\fill (0.0em,0em) circle [radius=.2em];
\fill (0.0em,-1em) circle [radius=.2em];
\fill (1.0em,-1em) circle [radius=.2em];
\fill (1.0em,0em) circle [radius=.2em];
\end{tikzpicture}
}
}
\newbox{\graphSQUARE}
\sbox{\graphSQUARE}{{\hspace{-.6em}\DRAWgraphSQUARE\hspace{-.6em}}}
\newcommand{\cyclefour}{{\usebox{\graphSQUARE}}}

\newcommand{\diam}{
\raisebox{-.2em}{
\begin{tikzpicture}[scale=.9]
  \draw[very thick]
	(0.0em,0em) -- (0.0em,-1em)
	(0.0em,0em) -- (1.0em,-1em)
	(0.0em,-1em) -- (1.0em,-1em)
	(1.0em,-1em) -- (1.0em,0em)
	(1.0em,0em) -- (0.0em,-1em);
\fill (0.0em,0em) circle [radius=.2em];
\fill (0.0em,-1em) circle [radius=.2em];
\fill (1.0em,-1em) circle [radius=.2em];
\fill (1.0em,0em) circle [radius=.2em];
\end{tikzpicture}
}
}

\newcommand{\Kfour}{
\raisebox{-.2em}{
\begin{tikzpicture}[scale=.9]
  \draw[very thick]
	(0.0em,0em) -- (0.0em,-1em)
	(0.0em,0em) -- (1.0em,-1em)
	(0.0em,-1em) -- (1.0em,-1em)
	(1.0em,-1em) -- (1.0em,0em)
	(0.0em,0em) -- (1.0em,0em)
	(1.0em,0em) -- (0.0em,-1em);
\fill (0.0em,0em) circle [radius=.2em];
\fill (0.0em,-1em) circle [radius=.2em];
\fill (1.0em,-1em) circle [radius=.2em];
\fill (1.0em,0em) circle [radius=.2em];
\end{tikzpicture}
}
}

\else
\include{smallgraphsdummy}
\fi
\usepackage{accents}

\usepackage{soul}

\usepackage{cleveref}
\usepackage{enumitem}
\setlist[enumerate]{label=\roman*.}
\setlist[enumerate,2]{label=\arabic*.}

\usepackage{scalerel,stackengine}
\stackMath
\newcommand\reallywidehat[1]{%
\savestack{\tmpbox}{\stretchto{%
  \scaleto{%
    \scalerel*[\widthof{\ensuremath{#1}}]{\kern-.6pt\bigwedge\kern-.6pt}%
    {\rule[-\textheight/2]{1ex}{\textheight}}
  }{\textheight}%
}{0.5ex}}%
\stackon[1pt]{#1}{\tmpbox}%
}

\newcommand\Cat{{\mathcal{C}}}
\newcommand\FinGraph{{\mathcal{G}}}

\newcommand\LFinGraph{{\langle \FinGraph \rangle_\Q}}
\newcommand\FinGraphNo{\widetilde{\mathcal{G}}}
\newcommand\LFinGraphNo{{\langle \FinGraphNo \rangle_\Q}}

\newcommand\LA{\Big\langle}
\newcommand\RA{\Big\rangle}
\newcommand{\curlyH}{{\mathcal{H}}}

\definecolor{darkgray}{HTML}{636363}

\newcommand{\ka}{\textbf{k}}
\newcommand\e{\mathsf{e}} 
\newcommand\disj{\dot\cup}    
 
\newcommand\mor{\operatorname{mor}}
\DeclareMathSymbol{\mlqq}{\mathrel}{operators}{"5C}
\DeclareMathSymbol{\mrqq}{\mathrel}{operators}{`"}

\newcommand\arr[1]{\xrightarrow{\ #1\ }}
\newcommand\qs{\stackrel{\mathsf{qs}}\shuffle}

\newcommand\Aut{\operatorname{Aut}}
\newcommand\Iso{\operatorname{Iso}}
\newcommand\Epi{\operatorname{Epi}}
\newcommand\RegEpi{\operatorname{RegEpi}}
\newcommand\Mono{\operatorname{Mono}}
\newcommand\RegMono{\operatorname{RegMono}}
\newcommand\mono{{\mathsf{mono}}} %
\newcommand\regmono{{\mathsf{regmono}}} %
\renewcommand\hom{{\mathsf{hom}}}
\newcommand\GCmonoprime{{\mathtt{GC}^{\mathsf{\mono^{\prime}}}}}
\newcommand\GCregmonoprime{{\mathtt{GC}^{\regmono^{\prime}}}}
\newcommand\GChom{{\mathtt{GC}^\hom}}
\newcommand\GChomprime{{\mathtt{GC}^{\mathsf{hom^{\prime}}}}}
\newcommand\GCmono{{\mathtt{GC}^{\mathsf{mono}}}}
\newcommand\GCregmono{{\mathtt{GC}^{\mathsf{regmono}}}}

\title{Signatures of graphs for bicommutative Hopf algebras}

\author{Diego Caudillo$^1$, Joscha Diehl$^2$, Kurusch Ebrahimi-Fard$^1$, Emanuele Verri$^2$}
\date{ \small
    $^1$\textit{Department of Mathematical Sciences, Norwegian University of Science and Technology}\\
    [1ex]
    $^2$\textit{Institute of Mathematics and Computer Science, University of Greifswald}\\[2ex]%
    
}

\begin{document}

\maketitle

\begin{abstract}

This article approaches the counting
of subgraphs,
in terms of signature-type functionals defined over
combinatorial Hopf algebras of graphs.
Well-known algebraic identities that arise in the context of counting subgraphs
are then captured by their character property and a type of ``Chen's identity''.
While different notions of subgraphs (and homomorphisms) correspond
to different combinatorial Hopf algebras on graphs, we will show that they are
all isomorphic to a polynomial Hopf algebra.
In addition, the isomorphy between the Hopf algebras
can be realized by maps that respect the counting operations.

\smallskip

\end{abstract}

\tableofcontents

\newpage
\section*{Introduction}

We propose to consider the number of times $c_\sigma(\Lambda)$
a
graph $\sigma$ appears as a substructure of a graph $\Lambda$ as a \emph{feature} of the latter.
We store these features, i.e., the numbers $c_\sigma(\Lambda)$, in the so-called
\textit{graph signature} $\mathtt{GC}$ of the graph $\Lambda$ 
\begin{align*}
    \mathtt{GC}(\Lambda):= \sum_{\gamma}c_{\gamma}(\Lambda)\gamma^*.
\end{align*}
It is a linear functional defined on a certain (Hopf) algebra of graphs.
A basis of this Hopf algebra is given by graphs $\gamma$,
and $\gamma^*$ denotes elements of the dual basis.

Here, we were purposefully vague when speaking of ``substructures''.
In fact there are several relevant ways of considering $\sigma$ a substructure of $\Lambda$.
\begin{itemize}
  \item There exists a morphism $\sigma \to \Lambda$.
  \item There exists an injective morphism $\sigma \to \Lambda$.
  \item Restricted to a subset $A$ of the edges, $\Lambda\evaluatedAt{A}$ is
    isomorphic to $\sigma$.
  \item The graph induced by a subset $B$ of the vertices, $\Lambda_B$ is
    isomorphic to $\sigma$.
\end{itemize}
Moreover one has to decide whether one wants to count multiplicities
of the automorphism group of $\sigma$
(as is done, when considering injective morphisms)
or not (as is done, when counting isomorphic copies obtained by edge-restriction).

Several of these incarnations of counting operations
have appeared in the literature, as is summarized below.
We bring additional structure to this assortment of feature maps
by establishing them as characters on certain Hopf algebras.

In particular,
\begin{itemize}

  \item
    We realize several counting operations (some of the well-known, some of them less-known)
    as characters on bicommutative Hopf algebras.

  \item
    We establish that all these Hopf algebras are isomorphic
    (to a polynomial Hopf algebra).

  \item
    The translation between the different counting characters
    are realized as Hopf isomorphisms.

\end{itemize}

\textbf{Related work}
The procedure of storing features of some data structure
as a character on some Hopf algebra originates
in combinatorics \cite{aguiar2006combinatorial}
and rough path theory
\cite{lyons1998differential,biau2020learning,chevyrev2018signature,diehl2020iterated,diehl2020generalized,kiraly2019kernels}.

The idea of associating a graph $\Lambda$ to its sequence of numbers of
homomorphisms is well-established in the literature, see for instance
\cite[Thm.~3.6]{lovasz1967operations}, \cite[Thm.~5.29]{lovasz2012large} and
the introduction of \cite{borgs2006counting}, where it is called the
\textit{profile} of a graph
(see also \cite{penaguiao2020pattern}).

\subsection*{Outline}

The paper is structured as follows. In \cref{sec:HA_graphs}, we recall basic
notions about graphs, in particular the types of subgraphs and graph
homomorphisms which will be used throughout the text.  We then endow graphs
with graded linear structures and introduce a variety of products: $m_{\hom}$,
$m_{\hom^{\prime}}$, $m_{\mono}$, $m_{\mono^{\prime}}$, $m_{\regmono}$,
$m_{\regmono^{\prime}}$ and their respective dual coproducts.

Some of the products and coproducts form filtered or (graded) Hopf algebras, as
we show in \cref{ssec:OurBialgebrasOnGraphs}.

In \cref{sec:Isom_poly_HA}, we show that these Hopf algebras are isomorphic to
the polynomial Hopf algebra on connected graphs. This is an application of
Samuelson--Leray's theorem.
As a consequence, we also obtain that the
underlying algebras are free commutative over connected graphs.

In \cref{sec:Signatures_translation}, we introduce counting-signatures defined
over these Hopf algebras.  First, in~\cref{ssec:CountingDefinitions}, we
introduce three different homomorpism numbers, which are respectively
based on: morphism, monomorphisms and regular monomorphism. The terminology
stems from category theory, as explained in \cref{sssec:
categoryfinitesimplegraphs}. We denote them $c_{\sigma}^{\hom}(\Lambda)$,
$c_{\sigma}^{\mono}(\Lambda)$, and $c_{\sigma}^{\regmono}(\Lambda)$.
Three other counting operations are defined analogously except that we divide them by
the cardinality of $\Aut(\sigma)$. We denote them as
$c_{\sigma}^{\hom^{\prime}}(\Lambda)$, $c_{\sigma}^{\mono^{\prime}}(\Lambda)$,
and $c_{\sigma}^{\regmono^{\prime}}(\Lambda)$.
They allow us to count the
occurrences of edge-restricted subgraphs and vertex-induced subgraphs as well.
In
\cref{ssec:Signatures} we introduce three signature-type objects denoted
$\GChom(\Lambda)$, $\GCmono(\Lambda)$, $\GCregmono(\Lambda)$, and also
$\GChomprime(\Lambda)$, $\GCmonoprime(\Lambda)$, and
$\GCregmonoprime(\Lambda)$. They encode the counting operations as linear
functionals over graphs and satisfy both a character property as well as 
a version of Chen's identity.

Finally, in \cref{ssec:translating}, translations between counting characters
are formulated in terms of Hopf algebra isomorphisms, including those already
known in the computer science literature.
This also gives an alternative proof of the isomorphy between these Hopf algebras.

\section*{Acknowledgements}

The authors would like to thank Harald Oberhauser for pointing out the paper
\cite{bib:BNH2015} and for many fruitful discussions on signature-type objects.

All authors benefited from the mobility fund of the project Pure Mathematics in
Norway, funded by the Trond Mohn Foundation and Tromsø Research Foundation. The
third author is supported by the Research Council of Norway through project
302831 Computational Dynamics and Stochastics on Manifolds (CODYSMA). Many
discussions took place in the installations of the Alfried Krupp
Wissenschaftskolleg in Greifswald, Germany, and The Erwin Schrödinger
International Institute for Mathematics and Physics in Vienna, Austria.

\section{Hopf algebras of graphs}
\label{sec:HA_graphs}

In this section, we introduce several algebras and coalgebras on graphs and determine which combination of them form Hopf algebras.
We start by introducing graphs, subgraphs and the free vector space on graphs. For a brief introduction on Hopf algebras and related concepts used throughout the text, we refer to \cref{ssec:back_Hopf} and the references cited there.

\subsection{Simple graphs}
\label{ssec:graphs}

The combinatorial objects under consideration in this work are \textbf{simple graphs}. Recall that a graph $\tau$ consists of a pair $(V,E)$ where $V=V(\tau)$ is the set of vertices and $E=E(\tau) \subseteq [V]^{2}$ is the set of edges; $[V]^{2}$ denotes the set of all 2-element subsets of $V$\footnote{Note that this definition allows for isolated vertices,
but no parallel edges and no self-loops.}. We denote the \textbf{empty graph} $\e:=(\emptyset,\emptyset)$. The graphs considered in this paper are finite, i.e., $|V| < \infty$.
We now define subgraphs. 

\begin{definition}
Given a subset of vertices $U\subseteq V(\tau)$ and a subset of edges $A \subseteq E(\tau)$, such that $\cup A \subseteq U$, we define the \textbf{subgraph}
\begin{align*}
  \tau\evaluatedAt{U,A} &:= \lrp{ U , A }.
\end{align*}
Two important special cases are as follows.

First, given a collection of vertices $U\subseteq V(\tau)$, we denote the \textbf{vertex-induced subgraph} by $\tau_U$, where $V(\tau_U) = U$ and
\begin{align*}
    E(\tau_{U}) &:= 
    \lrb{ e \in E(\tau) \,|   e \subset U}.
\end{align*}

Second, given $A\subset E(\tau)$, we denote the \textbf{edge-restricted} subgraph with $\tau\evaluatedAt{A}$, where 
\begin{align*}
    E\left(\tau\evaluatedAt{A}\right)=A
    \quad \text{and} \qquad
    V\Big(\tau\evaluatedAt{A}\Big) &:= \bigcup A. 
\end{align*}

\end{definition}

Given two sets $S_{1}$ and $S_{2}$, we write $S_{1}\;\disj\;S_{2}$ for their (external) disjoint union. We define the \textbf{disjoint union of graphs}, $\sigma$ and $\tau$, as follows:
\begin{align}
\label{disj_prod}
    \sigma \sqcup \tau := 
    \left( V(\sigma) \disj V(\tau) \>,\> E(\sigma) \disj E(\tau) \right).
\end{align}

\subsection{Graph homomorphisms}
\label{ssec:grhom}
\begin{definition}
Let $\sigma,\tau$ be graphs, a map $f: V(\sigma)\to V(\tau)$
is a \textbf{graph homomorphism} if for all $i,j\in V(\sigma)$
\begin{align*}
  \{i,j\} \in E(\sigma) \Rightarrow \{f(i),f(j)\} \in E(\tau).
\end{align*}
For a graph homomorphism $f$ we also write $f: \sigma \to \tau$ instead of $f: V(\sigma)\to V(\tau)$.
Accordingly $\im_V f := f(V(\sigma))$ is the image of the vertex set.
By slight abuse of notation, we also denote by $f$ the induced map on edges, in particular, 
\begin{align*}
  f( E(\sigma) ) = \{ \{f(i),f(j)\} \mid \{i,j\} \in E(\sigma) \} \subset E(\tau).
\end{align*}

We call $f$ a \textbf{graph isomorphism} if it is a bijection
and for all $i,j\in V(\sigma)$
\begin{align*}
  \{i,j\} \in E(\sigma) \Leftrightarrow \{f(i),f(j)\} \in E(\tau).
\end{align*}
If there exists a graph isomorphism $f:\sigma\to\tau$, then we write $\sigma\cong\tau$. 
\end{definition}
The notion of isomorphism allows us to consider the set of equivalence classes of graphs, denoted with $\FinGraph$. We will refer to graphs when considering both graphs or
equivalence classes of graphs. The meaning will hopefully be clear from the context. The following families of graph homomorphisms will be important for the rest of the text.
\begin{itemize}
\item 
$\Hom(\sigma,\tau)$ is the set of all graph homomorphisms.
\item
$\Mono(\sigma,\tau)$ is the set of all injective (on vertices) graph homomorphisms.
\item
$\RegMono(\sigma,\tau)$ is the set of all graph embeddings, i.e.~$f \in\Hom(\sigma,\tau)$ and $f: \sigma \to \tau_{f(V(\sigma))}$ is an isomorphism.

\item
$\Epi(\sigma,\tau)$ the set of all surjective (on vertices) graph homomorphisms.
\item
$\RegEpi(\sigma,\tau)$ the set of all surjective (on vertices) graph homomorphisms
such that the induced function on edges is also surjective.
\item
$\Iso(\sigma,\tau)$ the set of all graph isomorphisms  between $\sigma$ and $\tau$.
We also denote by $\Aut(\sigma) = \Iso(\sigma,\sigma)$ the group of automorphisms of $\sigma$.
\end{itemize}
The given nomenclature originates in category theory, see \Cref{ssec:category}. 

    \begin{figure}[H]       
        \small \it
        \begin{minipage}{\linewidth}
            \begin{minipage}{0.3\linewidth} \centering
            \includegraphics[height=3cm]{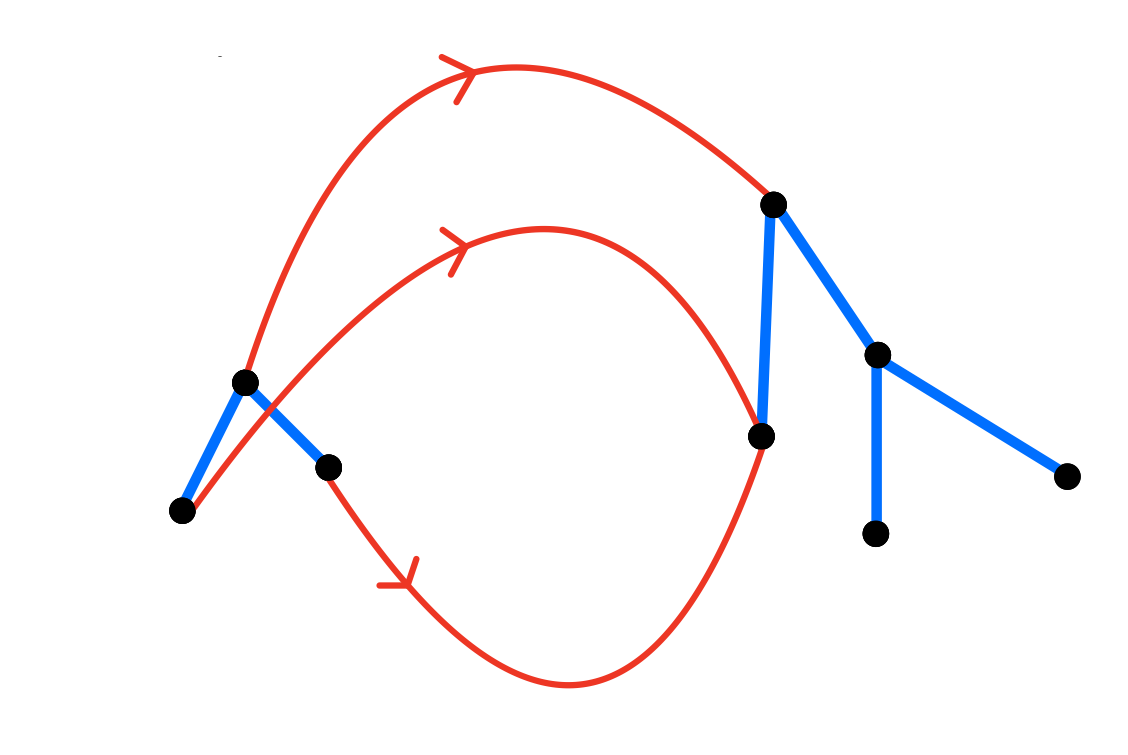}
            \end{minipage}
            \begin{minipage}{0.3\linewidth} \centering
            \includegraphics[height=3cm]{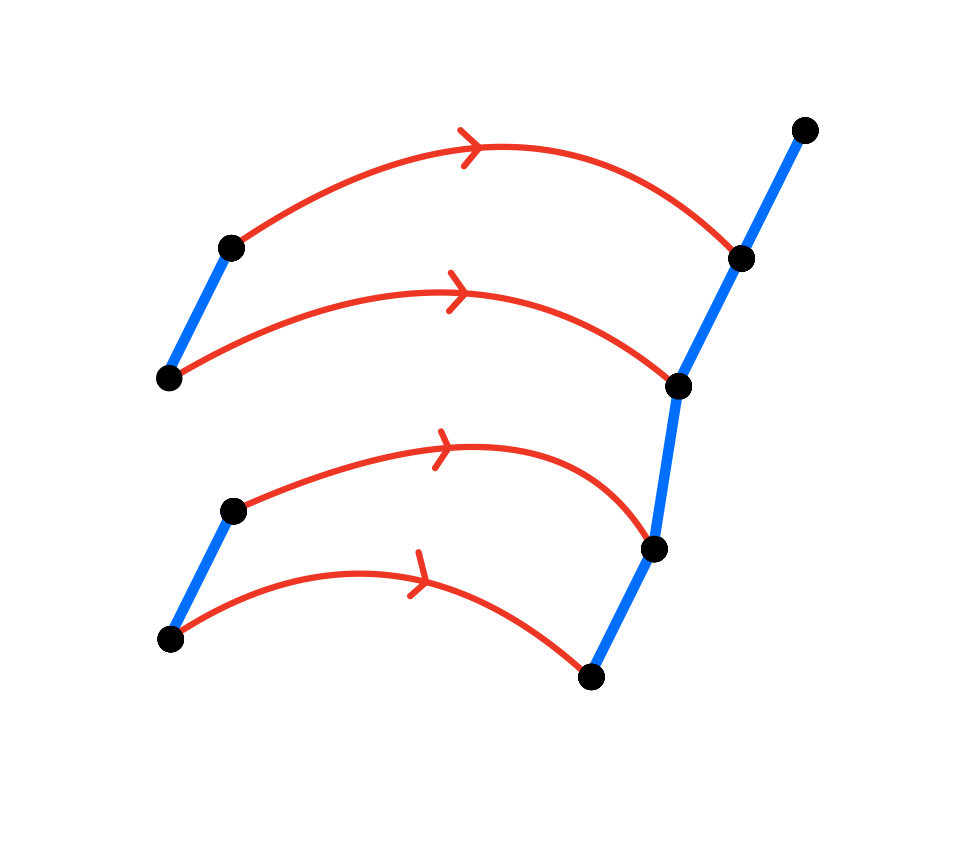}
            \end{minipage}
            \begin{minipage}{0.3\linewidth} \centering
            \includegraphics[height=3cm]{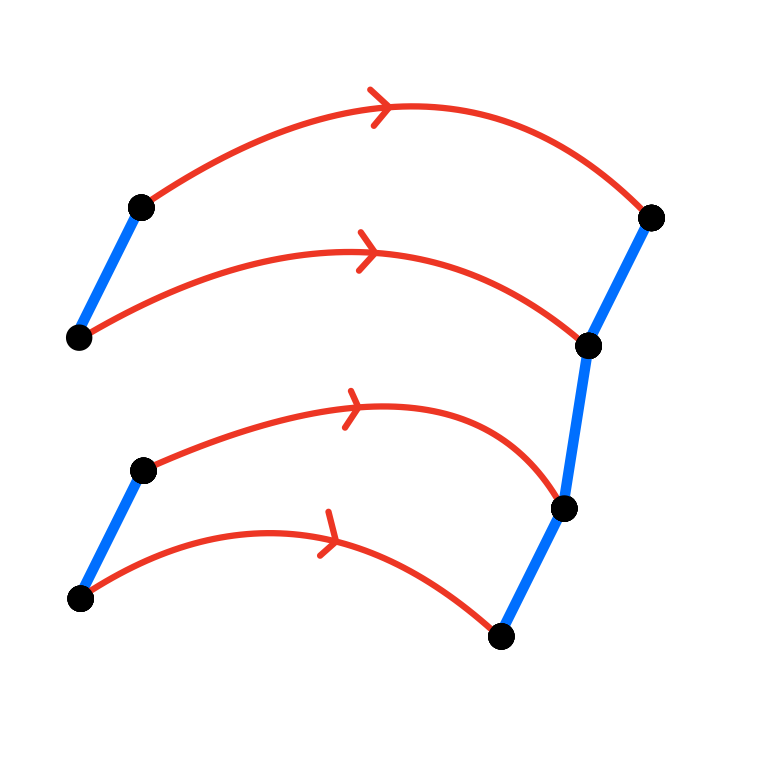}
            \end{minipage}
        \end{minipage}  
        \begin{minipage}{\linewidth}
            \begin{minipage}{0.3\linewidth} \centering
               graph homomorphism
            \end{minipage}
            \begin{minipage}{0.3\linewidth} \centering
                monomorphism which is not a regular monomorphism
            \end{minipage}
            \begin{minipage}{0.3\linewidth} \centering
            epimorphism which is not a regular epimorphism
            \end{minipage}
        \end{minipage}  
        
        \begin{minipage}{\linewidth}
            \begin{minipage}{0.3\linewidth} \centering
            \includegraphics[height=3cm]{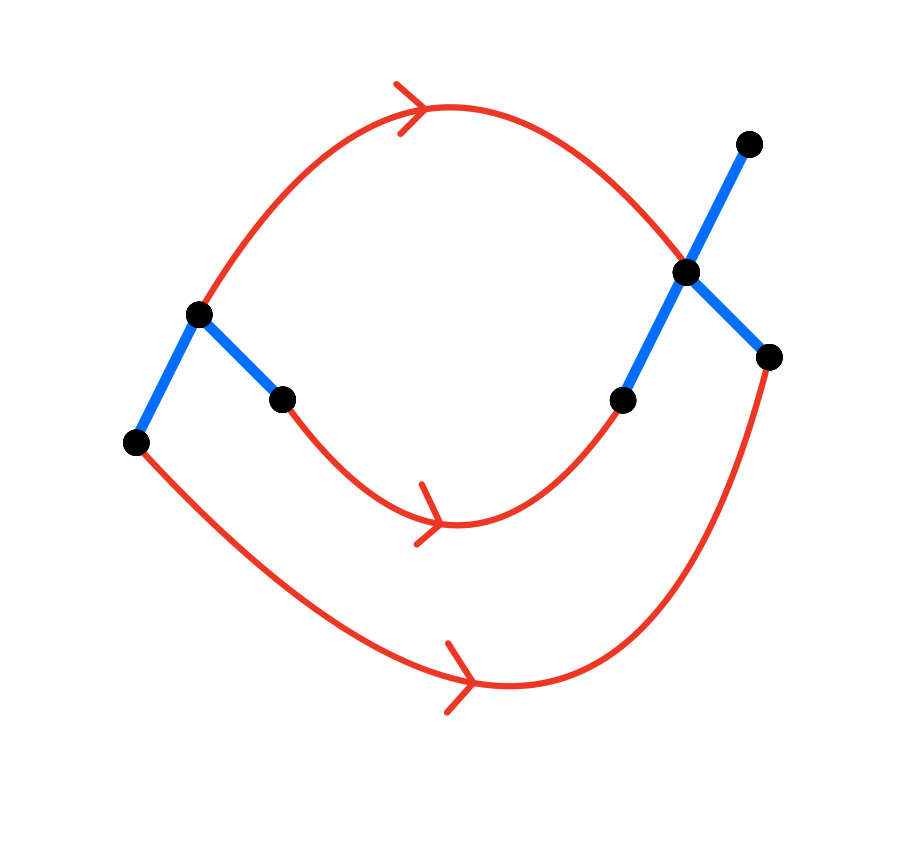}
            \end{minipage}
            \begin{minipage}{0.3\linewidth} \centering
            \includegraphics[height=3cm]{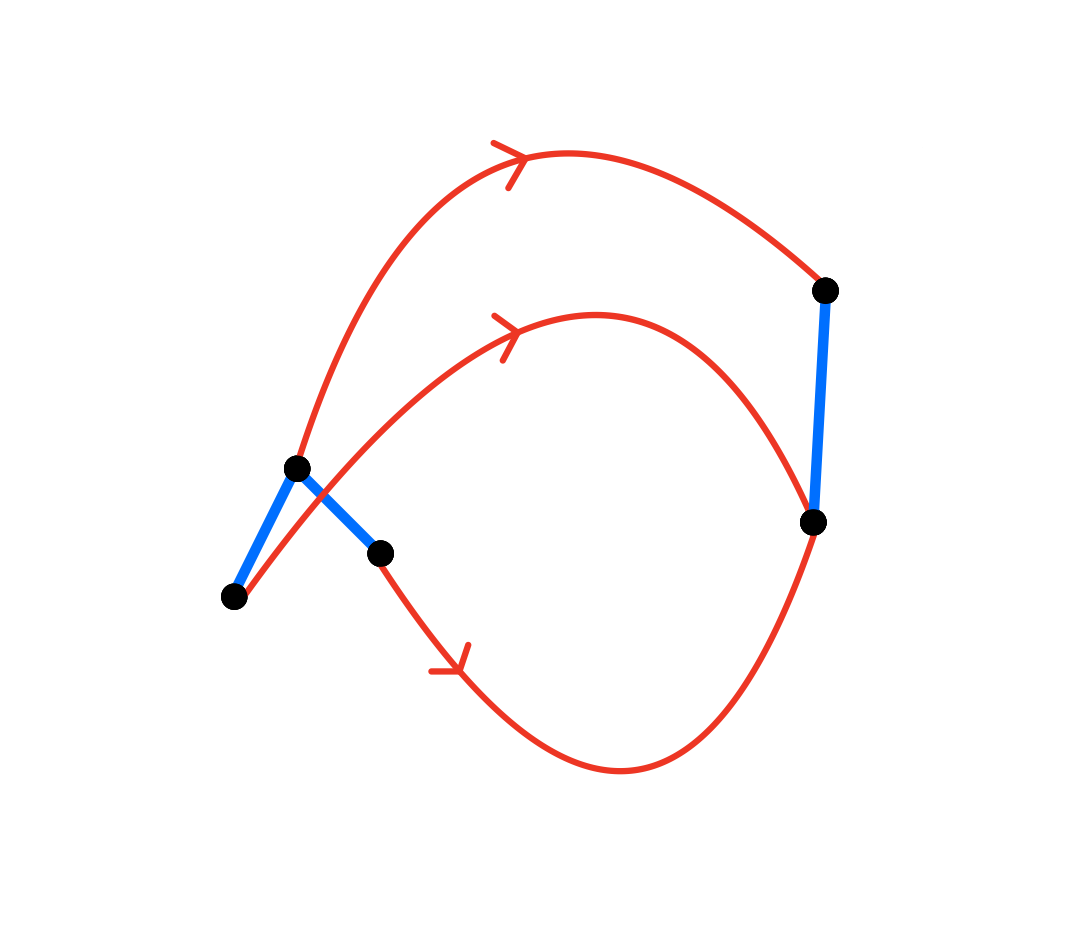}
            \end{minipage}
            \begin{minipage}{0.3\linewidth} \centering
            \includegraphics[height=3cm]{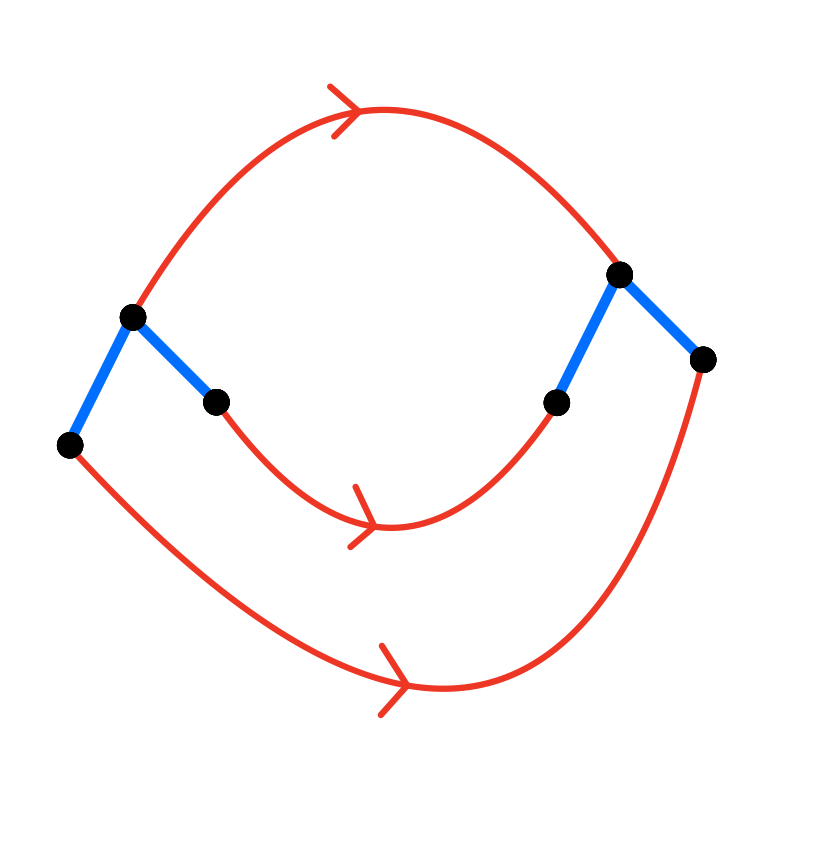}
            \end{minipage}
        \end{minipage}  
        \begin{minipage}{\linewidth}
            \begin{minipage}{0.3\linewidth} \centering
            regular monomorphism
            \end{minipage}
            \begin{minipage}{0.3\linewidth} \centering
             regular epimorphism
            \end{minipage}
            \begin{minipage}{0.3\linewidth} \centering
            isomorphism
            \end{minipage}
        \end{minipage}
         \caption{Examples of the classes of graph homomorphisms.}
    \end{figure}

\subsection{Free vector space on graphs}
\label{ssec:free_vector}
Let $\FinGraph$ be the set of equivalence classes of isomorphic graphs. We denote by $\LFinGraph$ the \textbf{free vector space} spanned by the elements in $\FinGraph$, i.e., formal \emph{finite} linear combinations of graphs with coefficients in $\Q$, the rational numbers. Let $\FinGraph_{n}$ be the set of equivalence classes of graphs with $n$ vertices. We can define a grading on $\LFinGraph$ by

\begin{align}
\label{eq:vertexCountGrading}
    \LFinGraph = \bigoplus_{n \ge 0}\langle\FinGraph_{n}\rangle_\Q.
\end{align}

This vector space is connected since the zero level is spanned by the singleton $\FinGraph_{0}= \lrb{\e}$. Moreover, it is \textit{locally finite}, i.e., each vector space appearing in the direct sum is finite-dimensional. 

In the following we will also consider the dual space $\LFinGraph^{*} = (\bigoplus_{n \ge 0}\langle\FinGraph_{n}\rangle_\Q)^{*}$.
Note that $\LFinGraph^{*}$
can be identified with the space of formal, \emph{infinite} linear combinations
$\sum_{\tau}c_{\tau}\tau^* \in \LFinGraph^{*}$,
whose pairing
with
an element
$\sum_{\sigma}d_{\sigma}\sigma \in \LFinGraph$
is defined as
\begin{align}
\label{pairing}
    \lra{\sum_{\tau}c_{\tau}\tau^*,\sum_{\sigma}d_{\sigma}\sigma} := \sum_{\gamma}c_{\gamma}d_{\gamma}.
\end{align}
The graded dual $\LFinGraph^{\circ} := \bigoplus_{n \ge 0}(\langle\FinGraph_{n}\rangle_\Q)^{*} \subset \LFinGraph^{*}$ will be used too. 

While we sometimes denote elements in $\LFinGraph^{*}$ with $\tau^*$ or $\delta_\tau$, when it is clear from the context, we will use $\tau$.

\subsection{Algebras and coalgebras on graphs}
\label{ssec:CatalogOfAlgebras}

We introduce several coproducts (or products),
that can be broadly classified into three groups:
coproducts that only consider the connected components
of a graph, coproducts that relate to edge-restricted subgraphs
and coproducts that relate to vertex-induced subgraphs.

\subsubsection{Coproducts related to the graph's connected components and polynomial algebras of graphs}
\label{sssec:symgraph}

Consider the operation in \eqref{disj_prod}, and extend it linearly to a product
\begin{align*}
            &\cdot_{\hom}: \LFinGraph \otimes \LFinGraph   \to \LFinGraph.
\end{align*}
The notation $\cdot_{\hom}$ reflects the compatibility with homomorphism counting (see \Cref{thm:character property} below).

    \begin{example}
    \begin{align*}
            \edge \cdot_{\hom} \edge \cherry &= \edge \edge \cherry
    \end{align*}
\end{example}
Since $\cdot_{\hom}$ is a graded map, we can use \eqref{eq:dual_alg} in Appendix \ref{sec.DualBialgebras} to define a corresponding coproduct:

\begin{definition}[\textbf{co-disjoint union}]
    \begin{align*}
            \Delta_\hom: \LFinGraph &\to \LFinGraph \otimes \LFinGraph\\
            \Delta_\hom(\tau) &:=\sum_{
            (\tau_{1},\tau_{2}) \in \FinGraph \times  \FinGraph} \langle \tau , \tau_{1} \cdot_{\hom} \tau_{2} \rangle\ \tau_{1} \otimes \tau_{2}.
        \end{align*}

        \begin{example}
          \begin{align*}
            \Delta_\hom(\! \edge )      &= \e \otimes\! \edge + \edge \otimes \e \\
            \Delta_\hom(\! \edge\edge ) &= \e \otimes\! \edge + \edge \otimes\! \edge + \edge \otimes \e\\
            \Delta_\hom(\! \edge\cherry ) &= \e \otimes\! \edge\cherry + \edge \otimes\! \cherry + \cherry\otimes\! \edge  
            + \edge\cherry \otimes \e.
          \end{align*}
        \end{example}
\end{definition}

\begin{remark}
Note that $\Delta_\hom$ only ``sees'' connected components. By a slight abuse of notation we can write a graph $\tau = \tau_{1}^{q_{1}} \cdots \tau_{m}^{q_{m}} \in \FinGraph$, where the $q_{i}$'s tell us the multiplicity of the respective connected graph. Then, we can write
\begin{align}
\label{eq:expl_disj_union}
    \Delta_\hom(\tau) &= \sum_{\substack{0 \le p_{1} \le q_{1} \\ \vdots \\ 0 \le p_{m} \le q_{m}\\ }}\tau_{1}^{p_{1}} \cdots  \tau_{m}^{p_{m}} \otimes \tau_{1}^{(q_{1} - p_{1})} \cdots   \tau_{m}^{(q_{m} - p_{m})}.
\end{align}
where $\tau^{0} := \e $.
 \end{remark}  
 
We now introduce a coproduct that turns the polynomial algebra into a bialgebra.

\begin{definition}
        On $\LFinGraph$, define for connected graphs $\tau$
        \begin{align*}
            \Delta_{\hom^{\prime}}(\tau) &:=\tau \otimes \e + \e \otimes \tau,
        \end{align*}
        and extend it as an algebra morphism for the $\cdot_{\hom}$ product. %
        \begin{example}
        \begin{align*}
            \Delta_{\hom^{\prime}}(\! \edge )      
            &= \e \otimes\! \edge + \edge \otimes \e \\
            \Delta_{\hom^{\prime}}(\! \edge \cdot_{\hom}\edge ) 
            &=  \Delta_{\hom^{\prime}}(\! \edge) \cdot_{\hom} \Delta_{\hom^{\prime}}(\! \edge)
            \\&=  (\e \otimes\! \edge + \edge \otimes \e) \cdot_{\hom} (\e \otimes\! \edge + \edge \otimes \e) 
            =  \e \otimes\! \edge\edge + 2 \edge \otimes\! \edge + \edge\edge \otimes \e.
          \end{align*}
          \end{example}
        \end{definition}
        
\begin{remark}
   Like $\Delta_{\hom}$, the coproduct $\Delta_{\hom^{\prime}}$ only ``sees'' connected components. Again we write $\tau = \tau_{1}^{q_{1}} \cdots \tau_{m}^{q_{m}} \in \FinGraph$, where the $q_{i}$'s tell us the multiplicity of the respective connected graph. It follows then
\begin{align*}
    \Delta_{\hom^{\prime}}(\tau_{1}^{ q_{1}}  \cdots  \tau_{m}^{q_{m}}) &=  \Delta_{\hom^{\prime}}(\tau_{1}^{ \cdot_{\hom} q_{1}} \cdot_{\hom} \cdots  \cdot_{\hom} \tau_{m}^{\cdot_{\hom} q_{m}})\\
    &= (\e \otimes \tau_{1}+\tau_{1}\otimes \e)^{\cdot_{\hom} q_{1}} \cdot_{\hom} \cdots \cdot_{\hom} (\e \otimes \tau_{m}+ \tau_{m}\otimes \e)^{\cdot_{\hom} q_{m}}.
\end{align*} 
which yields:
\begin{align*}
    \Delta_{\hom^{\prime}}(\tau) &:= \sum_{\substack{0 \le p_{1} \le q_{1} \\ \vdots \\ 0 \le p_{m} \le q_{m}\\ }}\binom{q_{1}}{p_{1}}\cdots\binom{q_{m}}{p_{m}}\tau_{1}^{p_{1}} \cdots  \tau_{m}^{p_{m}} \otimes \tau_{1}^{(q_{1} - p_{1})} \cdots  \tau_{m}^{(q_{m} - p_{m})}.
\end{align*}

Notice that \begin{align*}\frac{|\Aut(\tau_{1}^{q_{1}} \cdots   \tau_{m}^{ q_{m}})|}{|\Aut(\tau_{1}^{ p_{1}}  \cdots  \tau_{m}^{p_{m}})||\Aut(\tau_{1}^{(q_{1} - p_{1})}  \cdots  \tau_{m}^{(q_{m} - p_{m})})|} = \binom{q_{1}}{p_{1}}\cdots\binom{q_{m}}{p_{m}}.\end{align*}
\end{remark}
We can define a corresponding product $\cdot_{\hom^{\prime}}: \LFinGraph \otimes \LFinGraph \to \LFinGraph$
\begin{align*}
            \tau_{1} \cdot_{\hom^{\prime}} \tau_{2} &:=\sum_{
            \gamma \in \FinGraph} \langle \tau_{1} \otimes \tau_{2} , \Delta_{\hom^{\prime}}(\gamma) \rangle\, \gamma.
\end{align*}

The following remark shows the explicit relationship between $\hom$ and $\hom^\prime$.

\begin{remark}
\label{remark:defacto3}
With the linear map $\phi:\LFinGraph \to \LFinGraph$ satisfying $\phi(\tau) = |\Aut(\tau)|\tau$, it is immediate to see that 
\begin{align*}
    \cdot_{\hom^{\prime}} &= \phi \circ \cdot_{\hom} \circ (\phi^{-1} \otimes \phi^{-1}),
    \end{align*}
    which can be seen to be equivalent to
    \begin{align*}
    \Delta_{\hom^{\prime}} &=  (\phi^{-1} \otimes \phi^{-1}) \circ \Delta_{\hom} \circ  \phi.
    \end{align*}
\end{remark}

\begin{example}
 \begin{align*}
    \cherry \edge \cdot_{\cdot_{\hom^{\prime}}} \cherry \edge \edge &= \frac{|\Aut(\cherry \cherry \edge \edge \edge)|}{|\Aut(\cherry \edge )||\Aut( \cherry \edge \edge)|} \cherry \cherry \edge \edge \edge\\
    &=\frac{|\Aut(\cherry \cherry)||\Aut(\edge \edge \edge)|}{|\Aut(\cherry)||\Aut( \edge )||\Aut( \cherry)||\Aut(\edge \edge)|} \cherry \cherry \edge \edge \edge \\
    &=\binom{2}{1}\frac{|\Aut(\cherry)|^{2}}{|\Aut(\cherry)||\Aut( \cherry)|} \binom{3}{1}\frac{|\Aut(\edge)|^{3}}{|\Aut( \edge )||\Aut(\edge)|^{2}}\cherry \cherry \edge \edge \edge\\
    &= 6 \cherry \cherry \edge \edge \edge.
  \end{align*}
\end{example}

 \subsubsection{Coproducts related to subgraphs}
\label{sssec:coalgebra_sub}
\begin{definition}
    \begin{align*}
     \Delta_\mono(\tau) &:=\sum_{(\sigma_{1},\sigma_{2}) \in \FinGraph \times \FinGraph} \frac{|\{f \in \RegEpi(\sigma_{1}\cdot_{\hom}\sigma_{2},\tau)\ |\ f\evaluatedAt{\sigma_{1}},f\evaluatedAt{\sigma_{2}}\text{are monomorphisms}\}|}{|\Aut(\tau)|}\;\sigma_{1} \otimes \sigma_{2}. 
      \end{align*}
 \end{definition}
where $f\evaluatedAt{\sigma_{1}},f\evaluatedAt{\sigma_{2}}$ are restrictions of the map $f$ to the vertices of $\sigma_{1}$ and $\sigma_{2}$ respectively.

\begin{example}
 \begin{align*}
 \Delta_{\mono}(\vertex) &= \e \otimes \vertex + \vertex \otimes \e + \vertex \otimes \vertex \\
  \Delta_{\mono}(\edge\;) &= \e \otimes \edge + 2\vertex \otimes \edge + 2\vertex \vertex \otimes \edge + \edge \otimes \e + 2\edge \otimes \vertex + 2\edge \otimes \vertex \vertex + 2\edge \otimes \edge\\
  \Delta_{\mono}(\edge \vertex\;) &= \e \otimes \edge \vertex + \vertex \otimes \edge + 3\vertex \otimes \edge \vertex + 4\vertex \vertex \otimes \edge +6 \vertex \vertex \otimes \edge \vertex+ \edge \otimes \vertex + 4\edge \otimes \vertex \vertex + 2\edge \otimes \edge \vertex \\&+ \edge \vertex \otimes \e + 3\edge \vertex \otimes \vertex + 6\edge \vertex \otimes \vertex \vertex + 2\edge \vertex \otimes \edge + 2\edge \vertex \otimes \edge \vertex\\
   \Delta_{\mono}(\cherry\:) &= \e \otimes \cherry + 3\vertex \otimes \cherry + 6\vertex \vertex \otimes \cherry + 4\edge \otimes \edge + 4\edge \otimes \cherry \\&+ 4\edge \vertex \otimes \edge + 12\edge \vertex \otimes \edge \vertex + 4\edge \vertex \otimes \cherry + 4\cherry \otimes \edge \\&+ 6\cherry \otimes \vertex \vertex + 4\cherry \otimes \edge + 4\cherry \otimes \edge \vertex + 2\cherry \otimes \cherry .
 \end{align*}
\end{example}

We can use \cref{lemma:dualnongraded} to define a product 
$ \cdot_{\mono}\ : \LFinGraph \otimes \LFinGraph \to \LFinGraph,$ given by
\begin{align*}
            \tau_{1} \cdot_{\mono} \tau_{2} &:=\sum_{
            \gamma \in \FinGraph} \langle \tau_{1} \otimes \tau_{2} , \Delta_{\mono}(\gamma) \rangle\, \gamma.
\end{align*}

\begin{example}
  \begin{align*}
   \vertex\:\cdot_{\mono}\vertex
    &=
    \vertex \vertex + \vertex
\\
   \edge\:\cdot_{\mono}\vertex
    &=
    \edge \vertex
    + 2\edge\\
    \edge\:\cdot_{\mono}\edge
    &=
    \edge\edge 
    + \cherry
    + \frac{1}{2}\edge\\
    \edge\:\cdot_{\mono} \cherry
    &= \edge \cherry 
    + 3\tria +3\threeStar 
    + 2\threeLadder
    + 2 \cherry 
\end{align*}
 \end{example}

We now define a coproduct that is obtained from $\Delta_{\mono}$, in the same way that $\Delta_{\hom^{\prime}}$ was obtained from $\Delta_{\hom}$, see \cref{remark:defacto3}.

\begin{definition}
    \begin{align*}
     \Delta_{\mono^{\prime}}(\tau) &:=\sum_{(\sigma_{1},\sigma_{2}) \in \FinGraph \times \FinGraph} \frac{|\{f \in \RegEpi(\sigma_{1}\cdot_{\hom}\sigma_{2},\tau)|f\evaluatedAt{\sigma_{1}},f\evaluatedAt{\sigma_{2}}\text{are monomorphisms}\}|}{|\Aut(\sigma_{1})||\Aut(\sigma_{2})|}\;\sigma_{1} \otimes \sigma_{2}. 
      \end{align*}
 \end{definition}
Alternatively, it can be defined directly on the vertex set and the edge set as follows
\begin{align*}
     \Delta_{\mono^{\prime}}(\tau) &:=\sum_{\substack{A\cup B = V(\tau) \\ C\cup D = E(\tau)\\\cup C\subseteq A,\;\cup D \subseteq B}} \tau\evaluatedAt{A,C} \otimes \tau\evaluatedAt{B,D}.
\end{align*}

Again, from~\cref{lemma:dualnongraded}, we can define a product 
$ \cdot_{\mono^{\prime}}\ : \LFinGraph \otimes \LFinGraph \to \LFinGraph,$ given by
\begin{align*}
            \tau_{1} \cdot_{\mono^{\prime}} \tau_{2} &:=\sum_{
            \gamma \in \FinGraph} \langle \tau_{1} \otimes \tau_{2} , \Delta_{\mono^{\prime}}(\gamma) \rangle\, \gamma.
\end{align*}

\begin{example}
\begin{align*}
    \edge \cdot_{\mono^{\prime}} \edge &= \edge + 2 \edge \edge + 2 \cherry
\end{align*}
\end{example}

\begin{remark}~
  \label{rem:quasishuffle}
  \begin{enumerate}

  \item 
  This product also appears in \cite[Def.~5]{maugis2020testing}.

  \item
  \label{rem:quasishuffle:borie}
In \cite{bib:BNH2015}, the author introduces the restriction of $\cdot_{\mono^{\prime}}$ to $\FinGraphNo$, the set of graphs without isolated vertices. This restriction corresponds to a quasi-shuffle product, $\qs: \LFinGraphNo \otimes \LFinGraphNo \to \LFinGraphNo$. We name it quasi-shuffle since we can write
\begin{align*}
            \tau_{1} \qs \tau_{2} &=\sum_{
            \gamma \in \FinGraphNo}|\{A,B \subset E(\gamma)|A \cup B = E(\gamma),\gamma\evaluatedAt{A} \cong \tau_{1}, \gamma\evaluatedAt{B} \cong \tau_{2}\} | \gamma.
\end{align*}

Here, one can define a dual coproduct, as a map $\LFinGraphNo \to \LFinGraphNo \otimes \LFinGraphNo$
\begin{align*}
            \Delta_{\qs}(\tau) &:=\sum_{A \cup B = E(\tau)}\tau\evaluatedAt{A} \otimes \tau\evaluatedAt{B}.
\end{align*}
For example,
\begin{align*}
   \Delta_{\qs}(\edge\;) = \e \otimes \edge + \edge + \edge \otimes \e + \edge \otimes \edge.
\end{align*}

  \end{enumerate}

\end{remark}

\subsubsection{Coproducts related to induced subgraphs}
\label{sssec:coalgebra induced}

\begin{definition}
    \begin{align*}
        \Delta_{\regmono}(\tau) 
        &:=\sum_{(\sigma_{1},\sigma_{2}) \in \FinGraph \times \FinGraph} \frac{|\{f \in \Epi(\sigma_{1}\cdot_{\hom}\sigma_{2},\tau)\mid f\evaluatedAt{\sigma_{1}},f\evaluatedAt{\sigma_{2}}\text{are  regular monomorphisms}\}|}{|\Aut(\tau)|}\;\sigma_{1} \otimes \sigma_{2}. 
    \end{align*}
    The corresponding product is given by
    \begin{align*}
        \tau_{1} \cdot_{\regmono} \tau_{2} &:=\sum_{\gamma \in \FinGraph} \langle \tau_{1} \otimes \tau_{2} , \Delta_{\regmono}(\gamma) \rangle\, \gamma.
    \end{align*}
\end{definition}

As in the previous section, we define a rescaled version of the coproduct, see \cref{remark:defacto3}.

\begin{definition}

\begin{align*}
        \Delta_{\regmono^{\prime}}(\tau) 
        &:= \sum_{(\sigma_{1},\sigma_{2}) \in \FinGraph \times \FinGraph} \frac{|\{f \in \Epi(\sigma_{1}\cdot_{\hom}\sigma_{2},\tau)|f\evaluatedAt{\sigma_{1}},f\evaluatedAt{\sigma_{2}}\text{are  regular monomorphisms}\}|}{|\Aut(\sigma_{1})||\Aut(\sigma_{2})|}\;\sigma_{1} \otimes \sigma_{2},
    \end{align*}
or, directly on the set of vertices as
    \begin{align*}
        \Delta_{\regmono^{\prime}}(\tau) 
        &:= \sum_{I \cup J = V(\tau)} \tau_I \otimes \tau_J.
    \end{align*}
    The corresponding product is given by
    \begin{align*}
        \tau_{1} \cdot_{\regmono^{\prime}} \tau_{2} &:=\sum_{\gamma \in \FinGraph} \langle \tau_{1} \otimes \tau_{2} , \Delta_{\regmono^{\prime}}(\gamma) \rangle\, \gamma.
    \end{align*}
    Note that we can write
    \begin{align*}
\langle \tau_{1} \otimes \tau_{2} , \Delta_{\regmono^{\prime}}(\gamma) \rangle = |\{A,B \subseteq V(\gamma)|A \cup B = V(\tau),\tau_{A} \cong \tau_{1}, \tau_{B} \cong \tau_{2}\}|,
    \end{align*}
    which is the way it appears in \cite{penaguiao2020pattern}.
    
\begin{remark}
  $\Delta_{\regmono^{\prime}}$ appears in \cite{schmitt1995hopf} and was probably rediscovered in \cite{penaguiao2020pattern}. Schmitt constructed this coalgebra to provide an example of a Hopf algebra that is \emph{not} of incidence type. 
\end{remark}
    
    \end{definition}
\begin{example}
    \begin{align*}
        \Delta_{\regmono^{\prime}}(\vertex\:) 
        &= \vertex\otimes\e 
        + \e\otimes\! \vertex 
        + \vertex \otimes\! \vertex.\\
        \Delta_{\regmono^{\prime}}(\edge\:) 
        &= \edge \otimes \e
        + \e\otimes\! \edge 
        + 2\vertex \otimes\! \vertex
        + \edge \otimes\! \edge 
        + 2 \vertex\otimes\!\edge 
        + 2 \edge \otimes\! \vertex.
    \end{align*}
\end{example}
\begin{remark}
\label{remark:schmittinc}
Observe that
\begin{align*}
        \Delta_{\regmono^{\prime}}(\tau) 
        =\sum_{\substack{I \cup J = V(\tau)\\ I \cap J = \emptyset}} \tau_I \otimes \tau_J + \sum_{\substack{I \cup J = V(\tau)\\ I \cap J \neq \emptyset}} \tau_I \otimes \tau_J.
\end{align*}
In \cite[Section 12]{bib:schmitt1994incidence},  Schmitt 
introduces
a variant of this coproduct,
which ``ignores'' overlapping sets of vertices, i.e.
\begin{align*}
         \Delta_{\mathsf{Schmitt}}(\tau) 
        &:=\sum_{\substack{I \cup J = V(\tau)\\ I \cap J = \emptyset}} \tau_I \otimes \tau_J .
\end{align*}
 As an example,
    \begin{align*}
        \Delta_{\regmono^{\prime}}(\vertex\:) -\Delta_{\mathsf{Schmitt}}(\vertex\:)
        &= \vertex \otimes\! \vertex.\\
        \Delta_{\regmono^{\prime}}(\edge\:) -\Delta_{\mathsf{Schmitt}}(\edge\:) 
        &= \edge \otimes\! \edge 
        + 2 \vertex\otimes\!\edge 
        + 2 \edge \otimes\! \vertex.
    \end{align*}

\end{remark}

\begin{example}
\begin{align*}
    \vertex\:\cdot_{\regmono^{\prime}} \vertex 
    &= 2\vertex\vertex 
    + 2 \edge
    + \vertex\\
    \edge\:\cdot_{\regmono^{\prime}} \edge 
    &= 2 \edge \edge 
    + 2\threeLadder 
    + 2 \tailedTriangle 
    + 4 \cyclefour 
    + 4\diam
    + 6\Kfour
    + \edge 
    + 2\cherry 
    + 6\tria
\end{align*}
\end{example}

In Appendix \ref{ssec:Associativity.Proof} it is shown that the (co)products defined above for graphs are (co)commutative and (co)associative. Moreover, all of them share the same (co)unit. 

\begin{definition}
    The unit map is given by $u: \Q \to \LFinGraph$ with $u(1_{\Q})=\e$. 
    The counit map is defined by
    \begin{align*}
        \varepsilon(\tau) = 
    \begin{cases}
1, \;\text{if}\; \tau =  \e\\
0, \;\text{else}.
\end{cases}
    \end{align*}
    \end{definition}
  
The proof of the following result is given in \cref{ssec:Associativity.Proof}.

\begin{theorem}
  \label{thm:coassociative}
  The coproducts in \cref{sssec:symgraph}, \cref{sssec:coalgebra_sub}, and  \cref{sssec:coalgebra induced} are counital, coassociative and cocommutative. Dually, the corresponding products are unital, associative and commutative.
\end{theorem}

   The following lemma shows that the (co)products respect the grading defined in \cref{ssec:free_vector}. The proof is immediate.

\begin{lemma}
    Given the grading by the number of vertices in \cref{eq:vertexCountGrading}, the products and coproducts behave as follows.
~
\begin{enumerate}
\item $\cdot_{\hom}, \Delta_{\hom},\cdot_{\hom^{\prime}}, \Delta_{\hom^{\prime}}$ are graded;
\item $\cdot_{\mono},\cdot_{\mono^{\prime}},\cdot_{\regmono},\cdot_{\regmono^{\prime}}$ are filtered;
\end{enumerate}

\begin{remark}
On the contrary, $\Delta_{\mono},\Delta_{\mono^{\prime}},\Delta_{\regmono},\Delta_{\regmono^{\prime}}$ are \emph{not} filtered, according to the filtration induced by \cref{eq:vertexCountGrading}.
\end{remark}
\end{lemma}

\subsection{Bialgebras and Hopf algebras}
\label{ssec:OurBialgebrasOnGraphs}

We now determine which operations are compatible in the sense of bialgebra. 
\begin{lemma}[Bialgebra property]
\label{lemma:bialg_prop}
Considering all possible product-coproduct combinations:
\begin{table}[H]
\centering
\begin{tabular}{l|l|l|l|l|l|l|}
\cline{2-7}
     Bialgebra                  & $\cdot_{\hom} $  & $\cdot_{\hom^{\prime}}$ &  $\cdot_{\mono}$& $\cdot_{\mono^{\prime}}$ & $\cdot_{\regmono}$  & $\cdot_{\regmono^{\prime}}$  \\ \hline
\multicolumn{1}{|l|}{$\Delta_{\hom} $}  & \textcolor{brown}{no} & \textcolor{blue}{yes} & \textcolor{brown}{no} & \textcolor{blue}{yes} & \textcolor{brown}{no} &  \textcolor{blue}{yes}\\ \hline
\multicolumn{1}{|l|}{$\Delta_{\hom^{\prime}} $} & \textcolor{blue}{yes} & \textcolor{brown}{no} & \textcolor{blue}{yes} & \textcolor{brown}{no} & \textcolor{blue}{yes} &  \textcolor{brown}{no} \\ \hline
\multicolumn{1}{|l|}{$\Delta_{\mono} $} & \textcolor{brown}{no} & \textcolor{blue}{yes} & \textcolor{brown}{no} & \textcolor{brown}{no} & \textcolor{brown}{no} &  \textcolor{brown}{no}  \\ \hline
\multicolumn{1}{|l|}{$\Delta_{\mono^{\prime}} $} & \textcolor{blue}{yes} & \textcolor{brown}{no} & \textcolor{brown}{no} & \textcolor{brown}{no} & \textcolor{brown}{no} &  \textcolor{brown}{no} \\ \hline
\multicolumn{1}{|l|}{$\Delta_{\regmono} $} & \textcolor{brown}{no} & \textcolor{blue}{yes} & \textcolor{brown}{no} & \textcolor{brown}{no} & \textcolor{brown}{no} &  \textcolor{brown}{no} \\ \hline
\multicolumn{1}{|l|}{$\Delta_{\regmono^{\prime}} $} & \textcolor{blue}{yes} & \textcolor{brown}{no} & \textcolor{brown}{no} & \textcolor{brown}{no} & \textcolor{brown}{no} &  \textcolor{brown}{no}  \\ \hline
\end{tabular}
\end{table}
\end{lemma}

\begin{proof}
 Since $(\LFinGraph, \cdot_{\hom}, \Delta_{\hom^{\prime}})$ is a bialgebra by construction, it follows that its graded dual  $(\LFinGraph^*, \cdot_{\hom^{\prime}}, \Delta_{\hom})$ is also a bialgebra. 
 
 We now prove that $(\LFinGraph, \cdot_{\hom}, \Delta_{\mono^{\prime}})$ is a bialgebra. Let $\sigma,\tau \in \FinGraph$. Then
    \begin{align*}
     \Delta_{\mono^{\prime}}(\sigma \cdot_{\hom} \tau) &:=\sum_{\substack{A\cup B = V(\sigma \cdot_{\hom} \tau) \\ C\cup D = E(\sigma \cdot_{\hom} \tau)\\\cup C\subseteq A,\;\cup D \subseteq B}} (\sigma \cdot_{\hom} \tau)\evaluatedAt{A,C} \otimes (\sigma \cdot_{\hom} \tau)\evaluatedAt{B,D}\\
     &= \sum_{\substack{A\cup B = V(\sigma) \disj V(\tau) \\ C\cup D = E(\sigma) \disj E(\tau)\\\cup C\subseteq A,\;\cup D \subseteq B}} \sigma\evaluatedAt{A\cap V(\sigma),C\cap E(\sigma)} \cdot_{\hom} \tau\evaluatedAt{A \cap V(\tau),C \cap E(\tau)} \otimes \sigma\evaluatedAt{B\cap V(\sigma),D\cap E(\sigma)} \cdot_{\hom} \tau\evaluatedAt{B \cap V(\tau),D \cap 
 E(\tau)}\\
     &= \sum_{\substack{A\cup B = V(\sigma) \disj V(\tau) \\ C\cup D = E(\sigma) \disj E(\tau)\\\cup C\subseteq A,\;\cup D \subseteq B}} \sigma\evaluatedAt{A\cap V(\sigma),C\cap E(\sigma)} \otimes  \sigma\evaluatedAt{B\cap V(\sigma),D\cap E(\sigma)} \cdot_{\hom} \tau\evaluatedAt{A \cap V(\tau),C \cap E(\tau)} \otimes  \tau\evaluatedAt{B \cap V(\tau),D \cap 
 E(\tau)}\\
     &= \sum_{\substack{A\cup B = V(\sigma)  \\ C\cup D = E(\sigma) \\ \cup C\subseteq A,\;\cup D \subseteq B}} \sigma\evaluatedAt{A,C} \otimes \sigma\evaluatedAt{B ,D} \cdot_{\hom} \sum_{\substack{A\cup B = V(\tau)\\ C\cup D = E(\tau)\\ \cup C\subseteq A,\;\cup D \subseteq B}} \tau\evaluatedAt{A, C} \otimes \tau\evaluatedAt{B, D}\\
     &= \Delta_{\mono^{\prime}}(\sigma)\cdot_{\hom} \Delta_{\mono^{\prime}}(\tau),
      \end{align*}
where the second to last equality holds because the following two sets are equal
\begin{align*}
S_{1} &:= \{(X_{1},Y_{1},X_{2},Y_{2},W_{1},Z_{1},W_{2},Z_{2})| X_{1} \cup X_{2} = V(\sigma), Y_{1} \cup Y_{2} = E(\sigma), W_{1} \cup W_{2} = V(\tau), Z_{1} \cup Z_{2} = E(\tau)\\& \cup Y_{1} \subseteq X_{1},\cup Y_{2} \subseteq X_{2},\cup W_{1} \subseteq Z_{1},\cup W_{2} \subseteq Z_{2}    \},\\
S_{2} &:= \{(A \cap V(\sigma),C \cap E(\sigma),B \cap V(\sigma),D\cap E(\sigma),A \cap V(\tau),C \cap E(\tau),B \cap V(\tau),D\cap E(\tau))|\\& A \cup B = V(\sigma) \disj V(\tau), C \cup D = E(\sigma) \disj E(\tau), \cup C \subseteq A, \cup D \subseteq B  \}.
\end{align*}
Indeed, let $a_{1} \in S_{1}$. Now let $A = X_{1} \disj W_{1}$, $B = X_{2} \disj W_{2}$, $C = Y_{1} \disj Y_{2}$ and $D = Z_{1} \disj Z_{1}$. Then $a_{1} \in S_{2}$. Notice that if $\cup C \subseteq A$, then $\cup (C \cap E(\sigma))\subseteq A \cap V(\sigma)$. If $a_{2} \in S_{2}$, then it is immediate to see that it is also in $S_{1}$.

The proof of $(\LFinGraph, \cdot_{\hom}, \Delta_{\regmono^{\prime}})$ being a bialgebra is very similar. It follows dually that $(\LFinGraph, \cdot_{\mono^{\prime}}, \Delta_{\hom})$ and $(\LFinGraph, \cdot_{\regmono^{\prime}}, \Delta_{\hom})$ are bialgebras as well.

For $(\LFinGraph, \cdot_{\mono}, \Delta_{\hom^{\prime}})$ and $(\LFinGraph, \cdot_{\regmono^{\prime}}, \Delta_{\hom^{\prime}})$, we can use that  $\phi:(\LFinGraph,\cdot_{\hom},\Delta_{\hom^{\prime}}) \to (\LFinGraph,\cdot_{\hom^{\prime}},\Delta_{\hom})$, where $\phi(\tau):=|\Aut(\tau)|\tau$, is a bialgebra isomorphism. We then use \cref{remark:defacto3} to see that:
\begin{align*}
\Delta_{\hom^{\prime}}(\sigma \cdot_{\mono} \tau) &=  \Delta_{\hom^{\prime}}(\phi^{-1}(\phi(\sigma)\cdot_{\mono^{\prime}} \phi(\tau)))\\
&=  \phi^{-1} \otimes \phi^{-1} \circ \Delta_{\hom} \circ \cdot_{\mono^{\prime}} \circ (\phi \otimes \phi) (\sigma \otimes \tau)
\\&=  \phi^{-1} \otimes \phi^{-1} \circ \cdot_{\mono^{\prime}} \circ (\Delta_{\hom} \otimes \Delta_{\hom})  \circ (\phi \otimes \phi) (\sigma \otimes \tau)
\\&=  \phi^{-1} \otimes \phi^{-1} \circ \cdot_{\mono^{\prime}} \circ  (\phi \otimes \phi) \otimes (\phi \otimes \phi) \circ (\Delta_{\hom^{\prime}} \otimes \Delta_{\hom^{\prime}}) (\sigma \otimes \tau)
\\&= \Delta_{\hom^{\prime}}(\sigma) \cdot_{\mono} \Delta_{\hom^{\prime}}(\tau)
\end{align*}
  For  $(\LFinGraph, \cdot_{\regmono^{\prime}}, \Delta_{\hom^{\prime}})$ the proof is analogous. It then follows dually that $(\LFinGraph, \cdot_{\hom^{\prime}}, \Delta_{\mono})$  and $(\LFinGraph, \cdot_{\hom^{\prime}}, \Delta_{\regmono})$ are also bialgebras.

The rest of the proof consists in counterexamples, that can be easily checked.
\end{proof}
We now determine which of the bialgebras are in fact Hopf algebras. Recall that all of them are commutative and cocommutative. The underlying set is always $\LFinGraph$.
\begin{proposition}
If we consider the bialgebras of \cref{lemma:bialg_prop},
    \label{thm.HopfalgebrasList}
\begin{itemize}
\item $(\cdot_{\hom},\Delta_{\hom^{\prime}})$ and $(\cdot_{\hom^{\prime}},\Delta_{\hom})$ are Hopf algebras, connected and graded by \eqref{eq:vertexCountGrading},
\item $(\cdot_{\mono}, \Delta_{\hom^{\prime}}), (\cdot_{\mono^{\prime}}, \Delta_{\hom}), (\cdot_{\regmono}, \Delta_{\hom^{\prime}}),(\cdot_{\regmono^{\prime}}, \Delta_{\hom})$ are connected, filtered (with filtration induced by \eqref{eq:vertexCountGrading}) Hopf algebras,
\item $(\cdot_{\hom^{\prime}}, \Delta_{\mono})$, $(\cdot_{\hom}, \Delta_{\mono^{\prime}})$, $(\cdot_{\hom^{\prime}}, \Delta_{\regmono})$ and $(\cdot_{\hom}, \Delta_{\regmono^{\prime}})$ fail to be Hopf algebras.
\end{itemize}
\end{proposition}

\begin{remark}
  Some of these Hopf algebras appear in the literature already.
  $(\cdot_{\regmono^{\prime}},\Delta_{\hom})$ appears in \cite{penaguiao2020pattern}. So does the Hopf algebra $(\LFinGraphNo, \qs,\Delta_{\hom})$  in \cite{bib:BNH2015} which is a sub-Hopf algebra of $(\LFinGraph, \cdot_{\mono^{\prime}},\Delta_{\hom})$, see \Cref{rem:quasishuffle}, \cref{rem:quasishuffle:borie}.
\end{remark}

\begin{proof}~
     For connected, filtered bialgebras we can use \cref{thm.inv_conv_alg}. 
     We now show for $(\LFinGraph, \cdot_{\hom}, \Delta_{\mono^{\prime}})$ that the antipode does not exist. Analogous arguments work for the other cases. Suppose the antipode exists. A direct calculation of \eqref{eq:antipode_inverse} in the Appendix on the single vertex graph, $\vertex\:$, yields:
    \begin{align*}
      S(\vertex\:) = \sum_{n \ge 0} (-1)^n \vertex^{\:\cdot_{\hom} n} = - \vertex + \vertex\vertex - \vertex\vertex\vertex  + \cdots,
    \end{align*}
    which is not an element of $\LFinGraph$.
    \footnote{One could work with the dual coalgebra of \cite[Section 2.5]{bib:radford2011hopf} 
    to dualize $(\LFinGraph, \cdot_{\mono^{\prime}}, \Delta_{\hom})$ to a Hopf algebra. But we shall not need this here.}
\end{proof}

\begin{table}[H]
\centering
\begin{tabular}{l|l|l|l|l|l|l|}
\cline{2-7}
     Hopf algebra                 & $\cdot_{\hom} $  & $\cdot_{\hom^{\prime}}$ &  $\cdot_{\mono}$& $\cdot_{\mono^{\prime}}$ & $\cdot_{\regmono}$  & $\cdot_{\regmono^{\prime}}$  \\ \hline
\multicolumn{1}{|l|}{$\Delta_{\hom} $}  & \textcolor{brown}{no} & \textcolor{blue}{yes} & \textcolor{brown}{no} & \textcolor{blue}{yes} & \textcolor{brown}{no} &  \textcolor{blue}{yes}\\ \hline
\multicolumn{1}{|l|}{$\Delta_{\hom^{\prime}} $} & \textcolor{blue}{yes} & \textcolor{brown}{no} & \textcolor{blue}{yes} & \textcolor{brown}{no} & \textcolor{blue}{yes} &  \textcolor{brown}{no} \\ \hline
\multicolumn{1}{|l|}{$\Delta_{\mono} $} & \textcolor{brown}{no} & \textcolor{brown}{no} & \textcolor{brown}{no} & \textcolor{brown}{no} & \textcolor{brown}{no} &  \textcolor{brown}{no}  \\ \hline
\multicolumn{1}{|l|}{$\Delta_{\mono^{\prime}} $} & \textcolor{brown}{no} & \textcolor{brown}{no} & \textcolor{brown}{no} & \textcolor{brown}{no} & \textcolor{brown}{no} &  \textcolor{brown}{no} \\ \hline
\multicolumn{1}{|l|}{$\Delta_{\regmono} $} & \textcolor{brown}{no} & \textcolor{brown}{no} & \textcolor{brown}{no} & \textcolor{brown}{no} & \textcolor{brown}{no} &  \textcolor{brown}{no} \\ \hline
\multicolumn{1}{|l|}{$\Delta_{\regmono^{\prime}} $} & \textcolor{brown}{no} & \textcolor{brown}{no} & \textcolor{brown}{no} & \textcolor{brown}{no} & \textcolor{brown}{no} &  \textcolor{brown}{no}  \\ \hline
\end{tabular}
\end{table}

\section{Isomorphisms to the polynomial Hopf algebra on connected graphs}
\label{sec:Isom_poly_HA}

In this section, we investigate some of the properties of the Hopf algebras introduced in~\cref{sec:HA_graphs}. We will show that all of them are isomorphic to the polynomial Hopf algebra on connected graphs. For some of the Hopf algebras that we consider, this result is already known. The product $\cdot_{\regmono^{\prime}}$ appeared as a working example in reference \cite{penaguiao2020pattern}, where it was interpreted in the context of the theory of species. There it is shown that $(\cdot_{\regmono^{\prime}},\Delta_{\hom})$ is isomorphic to the polynomial Hopf algebra. Schmitt used in \cite[Thm.~10.2]{bib:schmitt1994incidence} the theory of incidence Hopf algebras to prove this result for $(\cdot_{\hom},\Delta_{\mathsf{Schmitt}})$, which is similar to $(\cdot_{\hom},\Delta_{\regmono^{\prime}})$, see \cref{remark:schmittinc}. This section relies on a classical result from the literature to Hopf algebras, based on the universal enveloping algebra on primitive elements, i. e. \cite[Thms.~4.3.1 and 4.3.3]{cartier2021classical}.

\subsection{Bicommutative unipotent Hopf algebras}
\label{ssec:bicommutative_unipotent}

In our setting, we can restrict our analysis to the family of bicommutative unipotent Hopf algebras.

\begin{definition}
   A bialgebra $\curlyH$ is said to be \textbf{unipotent} if for all $a\in \curlyH$, there exists $N(a)\geq 0$ such that for all $n \ge N(a)$:
   \begin{align*}
     (\id - u\circ \varepsilon)^{* n }(a) &= 0.
   \end{align*}
   
   If a bialgebra is 
   commutative and cocommutative we say that it is \textbf{bicommutative}.
\end{definition}

Note that any filtered connected bialgebra is unipotent and henceforth a Hopf algebra (see e.g. \cref{thm.inv_conv_alg}). 
We present a technical definition for constructing algebra homomorphisms.
\begin{definition}
\label{def:natural_extension}
    Suppose that $(A_1,m_1)$ and $(A_2,m_2)$ are unital commutative associative algebras and $(A_1,m_1)$ is freely generated by $\lrb{x_i}\subset A_1$.
    For any linear function $\varphi:\Q\lra{\lrb{x_i}}\to A_2$, we define its \textbf{natural extension} $\reallywidehat{\varphi}:(A_1,m_1)\to(A_2,m_2)$ by the linear extension of
    \begin{align*}
      \reallywidehat{\varphi} \circ m_1^{n-1}(x_{i_1} \otimes \cdots \otimes x_{i_n} ) :&= 
        m_2^{n-1}( \varphi(x_{i_1}) \otimes \cdots \otimes \varphi(x_{i_n} )).
    \end{align*}
\end{definition}
 Since $(A_1,m_1)$ is free, $\reallywidehat{\varphi}$ is well defined and an algebra homomorphism.

These natural extensions appear in a theorem by Samuelson and Leray as found in \cite[Thm.~4.3.3]{cartier2021classical}. This is a particular case of Cartier's theorem \cite[Thm.~4.3.1]{cartier2021classical} for bicommutative unipotent Hopf algebras.

\begin{theorem}[Samuelson--Leray--Cartier]
\label{thm:cartier}
 Suppose that $\curlyH$ is a bicommutative unipotent Hopf algebra over $\Q$.
Let $\mathtt{Prim}(\curlyH)$ be the set of primitive elements in $\curlyH$. The inclusion $\iota$ of $\mathtt{Prim}~(\curlyH)$ in $\curlyH$ naturally extends
to an isomorphism of Hopf algebras
\begin{align*}
   \hat\iota: (\langle\mathtt{Prim}(\curlyH)\rangle_{\Q} , \cdot , \Delta )  \to \curlyH,
\end{align*}
where $(\langle\mathtt{Prim}(\curlyH)\rangle_{\Q} , \cdot , \Delta )$ is the free commutative Hopf algebra on the primitive elements.
\end{theorem}

This theorem is paramount for the current text. We extend this result into a practical tool for verifying that a map is a Hopf algebra isomorphism.

\begin{lemma}
    \label{thm:freePrimitive}
    Suppose that $(A_1,m_1,\Delta_1)$ is a bicommutative unipotent Hopf algebra, and $(A_2,m_2,\Delta_2)$ is a bialgebra.
    If $\varphi:(A_1,m_1)\to(A_2,m_2)$ is an algebra homomorphism and \begin{align*}
        \varphi\lrp{ \mathtt{Prim}(A_1,\Delta_1)} \subset \mathtt{Prim}(A_2,\Delta_2),
    \end{align*} then $\varphi$ is a coalgebra homomorphism, i.e. $\Delta_2 \circ \varphi = (\varphi\otimes\varphi)\circ \Delta_1$.
\end{lemma}
\begin{proof}
    From \cref{thm:cartier}, we know that $(A_1,m_1)$ is free with generators $\lrb{x_i}\subset\mathtt{Prim}(A_1,\Delta_1)$. FIrst note that when $\varphi(1_{A_1})=1_{A_2}$, then $(\varphi \otimes \varphi)\circ \Delta_{1}(1_{A_1}) = (\varphi \otimes \varphi) (1_{A_1} \otimes 1_{A_1}) = \varphi(1_{A_1}) \otimes \varphi(1_{A_1}) = 1_{A_{2}} \otimes 1_{A_{2}} = \Delta_{2}(1_{A_{2}}) =  \Delta_{2}(\varphi(1_{A_{1}}))$. Let us then prove the comultiplicative property of $\varphi$ using induction on the size $n$ of monomials $x_{i_1} \cdots x_{i_n}$. We now use induction with base case $n=1$. From the hypothesis, we know that any primitive element $a\in\mathtt{Prim}(A_1,\Delta_1)$ has a primitive image. Hence
    \begin{align*}
      \Delta_2 \circ \varphi(a) &= 
        \varphi(a) \otimes 1_{A_2} +
        1_{A_2} \otimes \varphi(a) \\
                                &= 
        (\varphi\otimes \varphi) (a \otimes 1_{A_1} +
        1_{A_1} \otimes a) \\
                                &= 
        (\varphi \otimes \varphi) \circ \Delta_1 (a).
    \end{align*}
    
    Suppose now that for $b=x_{i_1}\cdots x_{i_n}\in A_1$, it holds that $$ \Delta_2\circ\varphi(b) = (\varphi\otimes\varphi)\circ\Delta_1(b). $$
    
    For $a\in \mathtt{Prim}(A_1)$,  we use the multiplicative property of  $\varphi$, and that $(A_2,m_2,\Delta_2)$ is a bialgebra.
    \begin{align*}
      \Delta_2 \circ \varphi \circ m_1 (a\otimes b) &= 
        \Delta_2 \circ m_2 \circ (\varphi \otimes \varphi)(a\otimes b) \\
                                                    &=
        m_{A_2\otimes A_2} \circ (\Delta_2 \otimes \Delta_2) \circ (\varphi\otimes\varphi) (a\otimes b).
    \end{align*}
    We use now that $\Delta_2 \circ \varphi(a) = (\varphi\otimes\varphi)\circ\Delta_1(a)$ and also the induction hypothesis. Then, the last expression equals
    \begin{align*}
        m_{A_2\otimes A_2} \circ \lrp{ (\varphi\otimes\varphi) \otimes (\varphi\otimes\varphi) } \circ (\Delta_1(a) \otimes \Delta_1(b)).
    \end{align*}
    Recall now that $\varphi\otimes\varphi: (A_1\otimes A_1,m_{A_1 \otimes A_1}) \to (A_2\otimes A_2,m_{A_2 \otimes A_2})$ is an algebra homomorphism. The former expression reduces to the following by the bialgebra compatibility of $\Delta_1$ and $m_1$.
    \begin{align*}  
        (\varphi\otimes\varphi) \circ m_{A_1\otimes A_1} \circ  (\Delta_1(a) \otimes \Delta_1(b)) 
        = (\varphi\otimes\varphi)\circ \Delta_1 \circ m_1 (a\otimes b).
    \end{align*}
\end{proof}

An important consequence of these results is that we only need to define a function on the primitive elements to obtain Hopf algebra homomorphisms. Even more, linear isomorphisms on primitive elements extend to Hopf algebra isomorphisms. The main idea is that linear combination of generators can be generators, e.g. $\Q[x+y,x-y]=\Q[x,y]$.

\begin{proposition}
    \label{thm:PrimitiveIsEnough}
    Suppose that $(\curlyH_1,m_1,\Delta_1)$ and $(\curlyH_2,m_2,\Delta_2)$ are bicommutative unipotent Hopf algebras.
    If $\varphi:\mathtt{Prim}(\curlyH_1,\Delta_1)\to \mathtt{Prim}(\curlyH_2,\Delta_2)$ is an invertible linear map, then its natural extension is a Hopf algebra isomorphism $\reallywidehat{\varphi}:(\curlyH_1,m_1,\Delta_1)\to(\curlyH_2,m_2,\Delta_2)$.
\end{proposition}
\begin{proof}
    \cref{thm:cartier} implies that $(\curlyH_1,m_1), (\curlyH_2,m_2)$ are freely  generated by $\lrb{x_i}\subset\mathtt{Prim}(\curlyH_1,\Delta_1)$ and $\lrb{y_j}\subset\mathtt{Prim}(\curlyH_2,\Delta_2)$, respectively. Since $\varphi$ is a linear isomorphism $\varphi(\lrb{x_i}) \subset \mathtt{Prim}(\curlyH_2,\Delta_2)$  is also a basis for $\mathtt{Prim}(\curlyH_2,\Delta_2)$ and freely generates $(\curlyH_2,m_2)$. Therefore, $\widehat{\varphi}$ is injective and surjective by construction. \cref{thm:freePrimitive} implies that $\widehat{\varphi}$ is a bijective bialgebra homomorphism. By \cref{thm.invertibleHomomorphisms} and~\cref{thm:homBiHopf}, it is a Hopf algebra isomorphism.
\end{proof}

We note that the reach of \cref{thm:PrimitiveIsEnough} can be extended to cocommutative unipotent bialgebras in the more general version of~\cref{thm:cartier} by Cartier~\cite[Theorem 4.3.1]{cartier2021classical}.

\subsection{The polynomial Hopf algebra of graphs}
\label{ssec:polynomial Hopf algebra}

One common technique for understanding Hopf algebras is comparing them to well-known objects.
The preferred candidate for us is $(\LFinGraph,\cdot_{\hom},\Delta_{\hom^{\prime}})$, commonly called the symmetric, polynomial, or free commutative Hopf algebra~\cite{bib:hazewinkel2010algebras,cartier2021classical,grinberg2020hopf}. This preference is persistent in the literature, in~\cite[Section 9]{bib:schmitt1994incidence}, Schmitt calls $\Delta_{\hom^{\prime}}$ ``the usual coproduct''. Some also call it the \textit{shuffle} Hopf algebra in the commutative alphabet of graphs~\cite{diehl2020generalized,cartier2021classical}. 
As we use \cref{thm:cartier} for this comparison, it is necessary to determine which elements are primitive for $\Delta_{\hom}$.
Unsurprisingly, these are the connected graphs.

\begin{lemma}
\label{lemma:primitive_elements}
Let us denote the connected graphs by $\FinGraph^0 := \lrb{ \tau \in\FinGraph |\, \tau \text{ is connected} }$. The primitive elements of both $\Delta_{\hom}$, and $\Delta_{\hom^{\prime}}$ are given by the span of the connected graphs, i.e.
\begin{align*}
  \mathtt{Prim}\lrp{\LFinGraph,\Delta_{\hom}} &= \langle\FinGraph^0\rangle_{\Q}=\mathtt{Prim}\lrp{\LFinGraph,\Delta_{\hom^{\prime}}}.
\end{align*}
\end{lemma}

\begin{proof}
    For $\Delta_{\hom^{\prime}}$, it follows from its definition that $ \langle\FinGraph^0\rangle_{\Q}\subseteq\mathtt{Prim}\lrp{\LFinGraph,\Delta_{\hom^{\prime}}}.$
    It is immediate to see that also $\langle \FinGraph^0 \rangle_{\Q} \subseteq \mathtt{Prim}\lrp{\LFinGraph,\Delta_{\hom}}$. Now, consider a linear combination of basis elements, such that at least one of the element is not a connected graph, then
    \begin{align*}
        \Delta_{\hom^{\prime}}\left(\sum c_{i}\tau_{i}\right) = \sum c_{i}\left(\e \otimes \tau_{i} + \tau_{i} \otimes \e\right) + \sum c_{i}(*)
    \end{align*}
    and now for $\sum c_{i}(*)=0$ implies that all $c_{i}=0$ because the terms in $(*)$ are (also) linearly independent. An analogous argument shows that $\lra{ \FinGraph^0 }_{\Q} =\mathtt{Prim}\lrp{\LFinGraph,\Delta_{\hom}}$.
\end{proof}

Aided by \cref{thm:cartier}, we construct an isomorphism from the polynomial Hopf algebra to the other Hopf algebras of \cref{thm.HopfalgebrasList}.

\begin{theorem}
    \label{thm:AllIsomorphisms}
    The inclusion $ \psi: \langle\FinGraph^{0}\rangle_{\Q} \to \LFinGraph$
    extends naturally --in the sense of \cref{def:natural_extension}-- to the Hopf algebra isomorphisms,
    \begin{align*}
        \widehat{\psi_{\dagger}} : (\langle\FinGraph\rangle_{\Q},\cdot_{\hom},\Delta_{\hom^{\prime}})\to \mathcal{H_{\dagger}},
    \end{align*}
    where $\mathcal{H_{\dagger}}$, denotes any of the Hopf algebra of \cref{thm.HopfalgebrasList}, $$\dagger = \{(\cdot_{\hom},\Delta_{\hom^{\prime}}), 
 (\cdot_{\hom^{\prime}},\Delta_{\hom}), (\cdot_{\mono},\Delta_{\hom^{\prime}}), (\cdot_{\mono^{\prime}},\Delta_{\hom}), (\cdot_{\regmono},\Delta_{\hom^{\prime}}), (\cdot_{\regmono^{\prime}},\Delta_{\hom})\}.$$ In particular, this implies that, all of them, as algebras, are free commutative on the set of connected graphs.
\end{theorem}

\begin{remark}
    The isomorphism $\widehat{\psi}_{(\cdot_{\hom^{\prime}},\Delta_{\hom})}$ is the well-known self-duality of the polynomial Hopf algebra, while $\widehat{\psi}_{(\cdot_{\regmono^{\prime}},\Delta_{\hom})}$ was given by \cite{penaguiao2020pattern}. In \cite{bib:schmitt1994incidence}, Schmitt showed that the cocommutative incidence Hopf algebras (of graphs), i.e. $(\LFinGraph,\cdot_{\hom},\Delta_{\mathsf{Schmitt}})$ (see \cref{remark:schmittinc}), is also isomorphic to $(\LFinGraph,\cdot_{\hom},\Delta_{\hom^{\prime}})$.
\end{remark}

\begin{proof}
  All these Hopf algebras are bicommutative and unipotent, so we apply~\cref{thm:cartier}. 
  
  Using \cref{lemma:primitive_elements}, we know that $\texttt{Prim}(\curlyH,\Delta_{\hom})= \texttt{Prim}(\curlyH,\Delta_{\hom^{\prime}})= \langle \FinGraph^{0} \rangle_{\Q}$. Therefore the inclusion $ \psi : \langle\FinGraph^{0} \rangle_{\Q} \to \LFinGraph $
  extends to an isomorphism between Hopf algebras
\begin{align*}
        \widehat{\psi_{\dagger}} : (\LFinGraph,\cdot_{\hom},\Delta_{\hom^{\prime}})\to \mathcal{H_{\dagger}},
    \end{align*}
    where $\dagger = \{(\cdot_{\hom},\Delta_{\hom^{\prime}}),(\cdot_{\hom^{\prime}},\Delta_{\hom}),(\cdot_{\mono},\Delta_{\hom^{\prime}}),(\cdot_{\mono^{\prime}},\Delta_{\hom}),(\cdot_{\regmono},\Delta_{\hom^{\prime}}),(\cdot_{\regmono^{\prime}},\Delta_{\hom})\}$.
\end{proof}

The isomorphy between Hopf algebras of \cref{thm.HopfalgebrasList} will be shown again in \cref{sec:Signatures_translation}, using different Hopf algebra maps.

\begin{remark}
  As a consequence of \cref{thm:AllIsomorphisms}, every graph can be uniquely written as a polynomial on connected graphs. As an example:

    \begin{align*}
        \edge\edge &= -\frac{1}{2}\edge -\frac{1}{2}\cherry  + \frac{1}{2}\edge\cdot_{\mono^{\prime}}\edge.
    \end{align*}

\end{remark}

\section{Signatures and translation between counting}
\label{sec:Signatures_translation}

One can find various classes of definitions for subgraph counting, see for instance \cite{curticapean2017homomorphisms,biggs1978cluster,borgs2006counting,bravo2021principled,penaguiao2020pattern,bib:BNH2015}. For the most part, such definitions of counting are based on enumerative combinatorics, either based on some for of set-restrictions or based on graph homomorphisms.

\subsection{Counting functions}
\label{ssec:CountingDefinitions}

The following definitions and theorems provide a formal setting to speak of graph counting functions. Due to the variety of them, we use the notation $c^{\dagger}_\alpha(\Gamma)$, where the super-index `$\dagger$' keeps a record of the sort of counting that we perform. As a cultural note, 
in a statistical setting, the graph $\alpha$ resp. $\Gamma$ are usually called a \textbf{pattern} resp. \textbf{sample}.
The following counting function was introduced in~\cite{lovasz1967operations}. 

\begin{definition}[Homomorphisms counting]
\label{def:homcount}
For $\sigma,\Upsilon \in \FinGraph$, define
\begin{align*}
c^\hom_\sigma(\Upsilon)   &:= \lrv{ \Hom(\sigma,\Upsilon) }
\end{align*} 
\end{definition}

Note that $c^\hom$ is \emph{not} counting sub-graphs. We also introduce the counting scaled by the cardinality of the automorphism group:

\begin{definition}[Scaled homorphisms counting]
    Let $\tau,\Lambda\in\FinGraph$. 
\begin{align*}
c_\sigma^{\hom^{\prime}}(\Upsilon) &:=
        \frac{\lrv{\Hom(\sigma,\Upsilon)}}{\lrv{\Aut(\sigma))}}.
    \end{align*}
\end{definition}
which is again \emph{not} counting subgraphs, as it does not have to be a non-negative integer.

We now introduce a counting function that is based on graph monomorphisms.
\begin{definition}[Monomorphisms counting]
    Let $\tau,\Lambda\in\FinGraph$. 
\begin{align*}
c_\sigma^\mono(\Upsilon) &:=
        \lrv{\Mono(\sigma,\Upsilon)}.
    \end{align*}
\end{definition}

which is an unnormalized counting of subgraphs.

\begin{definition}
    \label{def.CountingFunctionedge}
    
    Suppose that  $\tau,\Lambda\in\FinGraph$.
    Define the 
    \textbf{sub-graph counting} of $\tau$ inside of $\Lambda$
    as,
\begin{align*}
  c_{\tau}^{\mono^{\prime}}(\Lambda) :&= \lrv{\lrb{A \subseteq V(\Lambda), B\subseteq E(\Lambda) \mid \cup B \subseteq A,\; \Lambda\evaluatedAt{A,B} \cong \tau} }
\end{align*}
\end{definition}

These two counting operations are basically the same, as shown the following result.
\begin{proposition}
  \label{thm.CountingERenumerative}
  Suppose that $\Lambda$ is a fixed graph in $\FinGraph$. If $\tau \in \FinGraph$, then 
  \begin{align*}
      c_\tau^{\mono^{\prime}}(\Lambda)
      &=
      \frac{ \lrv{ \Mono(\tau,\Lambda) } }{|\Aut(\tau)|}.
    \end{align*}
  \end{proposition}
  \begin{proof}
  First we realize that
      \begin{align*}
        \Mono(\tau,\Lambda) &= \bigcup_{\substack{A\subset V(\Lambda)\\ B\subset E(\Lambda)\\ \cup B = A \\ \Lambda\evaluatedAt{A,B} \cong \tau}} \Iso\lrp{\tau,\,\Lambda\evaluatedAt{A,B} },
    \end{align*}
    and then it follows that
    \begin{align*}
        \Big|\bigcup_{\substack{A\subset V(\Lambda)\\ B\subset E(\Lambda)\\ \cup B = A \\ \Lambda\evaluatedAt{A,B} \cong \tau}} \Iso\lrp{\tau,\,\Lambda\evaluatedAt{A,B} }\Big| &=
        \sum_{\substack{A\subset V(\Lambda)\\ B\subset E(\Lambda)\\ \cup B = A \\ \Lambda\evaluatedAt{A,B} \cong \tau}} \Big|\Iso\lrp{\tau,\,\Lambda\evaluatedAt{A,B} }\Big|
       \\&=\sum_{\substack{A\subset V(\Lambda)\\ B\subset E(\Lambda)\\ \cup B = A \\ \Lambda\evaluatedAt{A,B} \cong \tau}} \Big|\Aut\lrp{\tau}\Big| = c^{\mono^{\prime}}_\tau(\Lambda) \Big|\Aut\lrp{\tau}\Big|
    \end{align*}
  \end{proof}

\begin{proposition}
    \label{thm.CountingFunctionedge}
    If $\Lambda \in \FinGraphNo$ and $\tau \in \FinGraphNo$, i.e. if they both have no isolated vertex, subgraph counting based only on edge subsets can be obtained from convolutions, using the restricted and corestricted version of the coproduct, $\Delta_{\underbar{\mono^{\prime}}}:\LFinGraphNo \to \LFinGraphNo \otimes \LFinGraphNo$ (see \Cref{rem:quasishuffle}, \cref{rem:quasishuffle:borie})
    \begin{align*}
        c_\tau^{\mono^{\prime}}(\Lambda) &=  \frac{\lrp{\tau^{*} *_{\cdot_{\mono^{\prime}}^{\prime}} \zeta }}{2^{|E(\tau)|}}(\Lambda).
       \end{align*}
    \end{proposition}
    \begin{proof}
   \begin{align*}
    \lrp{\zeta *_{\underbar{\mono^{\prime}}} \tau^{*}}(\Lambda)
    &=
    m_{\Q} \circ (\zeta \otimes \tau^{*}) \circ \Delta_{\underbar{\mono^{\prime}}}(\Lambda)  \\
    &=
    \sum_{\substack{A \subseteq E(\Lambda)\\ E(\Lambda)\setminus A \subseteq B \subseteq E(\Lambda)}}
    m_{\Q} \circ (\zeta \otimes \tau^{*}) \Lambda\evaluatedAt{A} \otimes \Lambda\evaluatedAt{B} \\
    &=
    \sum_{\substack{A \subseteq E(\Lambda)\\ E(\Lambda)\setminus A \subseteq B \subseteq E(\Lambda)}}
    m_{\Q} \circ (\zeta \otimes \tau^{*}) \Lambda\evaluatedAt{B} \otimes \Lambda\evaluatedAt{A} \\
    &=
    \sum_{\substack{A \subseteq E(\Lambda)\\ E(\Lambda)\setminus A \subseteq B \subseteq E(\Lambda)}}\tau^{*}(\Lambda\evaluatedAt{A}) \\
    &=
    2^{|E(\tau)|}
    \sum_{A \subseteq E(\Lambda)}
    \tau^{*}( \Lambda\evaluatedAt{A} ) \\
    &=
    2^{|E(\tau)|} c^{\mono^{\prime}}_\tau(\Lambda).
  \end{align*}
    \end{proof}

We finally define two counting operations that are based on regular monomorphisms.
\begin{definition}
      \label{def.CountingFunctionmono}
      For $\sigma$ and $\Upsilon$ in $\FinGraph$, define the 
    \textbf{regular monorphisms counting} of $\sigma$ inside of $\Upsilon$
    as
      \begin{align*}
        c_{\sigma}^{\regmono}(\Upsilon) :&= |\RegMono(\sigma,\Upsilon)|.
  \end{align*}
\end{definition}

 Counting induced subgraphs is called \textbf{motif counting}, see e.g.~\cite{bressan2018motif,bib:MSOIKCA2002}.

\begin{definition}
      \label{def.CountingFunctionvertex}
      For $\sigma$ and $\Upsilon$ in $\FinGraph$, define the 
    \textbf{vertex-induced counting} of $\sigma$ inside of $\Upsilon$
    as
      \begin{align*}
        c^{\regmono^{\prime}}_\sigma(\Upsilon) :&= \lrv{\lrb{U \subseteq V(\Upsilon) \mid \Upsilon_U \cong \sigma} }.
  \end{align*}
\end{definition}
 
We have again a reformulation in terms of graph homomorphisms.

\begin{proposition}
    Suppose that $\Upsilon\in\FinGraph$ is fixed. If $\sigma\in\FinGraph$, then
    \begin{align*}
           c^{\regmono^{\prime}}_\sigma(\Upsilon) = \frac{\lrv{ \RegMono(\sigma,\Upsilon) }}{|\Aut(\sigma)|}.
    \end{align*}
\end{proposition}
  \begin{proof}
    This time, the equality is due to 
    \begin{align*}
      \RegMono(\sigma,\Upsilon) 
                             &=  
     \bigcup_{U\subset V(\Upsilon)} \Iso(\sigma,\Upsilon_U).
    \end{align*}
    Now proceed
    analogously to the proof of in~\cref{thm.CountingERenumerative}.
\end{proof}

Analogous to \cref{thm.CountingFunctionedge}, we obtain the following. The proof is very similar and omitted.
\begin{proposition}
    \label{thm.CountingFunctionvertex}
    Suppose that $\Upsilon\in\FinGraph$.
    If $\sigma\in\FinGraph$, then
    \begin{align*}
      c^{\regmono^{\prime}}_\sigma(\Upsilon)
     &=
     \frac{\lrp{\zeta *_{\cdot_{\regmono^{\prime}}} \sigma^{*}}(\Upsilon)}{2^{|V(\sigma)|}}.
    \end{align*}
\end{proposition}

\subsection{Graph counting signatures}
\label{ssec:Signatures}

In this section, we aim at extending the notion of signatures to the setting of graphs. 
A time-series signature is a linear functional over a combinatorial Hopf algebra with multiplicative and comultiplicative properties. Each evaluation of this functional provides information on iterated-sums/integrals of a time-dependent function. The interaction of these functionals with the Hopf algebra operations allows us to simplify computations.

We can find plenty of examples of products and coproducts of graphs and plenty of graph functionals or indices~\cite{bib:MSOIKCA2002,bib:ROBSB2019,bib:GBH2020}. Among those, some products seem to respect counting functions, e.g.~\cite{maugis2020testing,penaguiao2020pattern,lovasz1967operations}. Such  interactions motivates us to define a collection of \textit{signatures for graphs}.
In the following, we will prove the signature property and further show that all these functionals are in some sense equivalent. Moreover, the equivalence is given by Hopf algebra isomorphisms.

\begin{definition}[Graph counting signatures]

    Assume that $\Lambda\in\FinGraph$ is a fixed graph.
    We define the following \textbf{graph counting signatures}.
    \begin{enumerate}

 \item \subitem$\GChom:\LFinGraph^\circ\to\LFinGraph^*$ given by  $
            \GChom(\Lambda) := \sum_{\sigma\in\FinGraph} c_\sigma^{\hom}(\Lambda)\ \sigma.$

             \subitem  $\GChomprime:\LFinGraph^\circ\to\LFinGraph^*$ given by  $
            \GChomprime(\Lambda) := \sum_{\sigma\in\FinGraph} c_\sigma^{\hom^{\prime}}(\Lambda)\ \sigma.$
              \item \subitem$\GCmono:\LFinGraph^\circ \to \LFinGraph^\circ$ given by  
        $ \GCmono(\Lambda) := \sum_{\sigma\in\FinGraph} c^\mono_\sigma(\Lambda)\ \sigma.$
    
        \subitem $\GCmonoprime:\LFinGraph^\circ \to \LFinGraph^\circ$ given by  
        $ \GCmonoprime(\Lambda) := \sum_{\sigma\in\FinGraph} c_{\sigma}^{\mono^{\prime}}(\Lambda)\ \sigma.$

  \item \subitem$\GCregmono:\LFinGraph^\circ\to\LFinGraph^\circ$ given by  $
            \GCregmono(\Lambda) := \sum_{\sigma\in\FinGraph} c_\sigma^{\regmono}(\Lambda)\ \sigma.$ 
        
        \subitem $\GCregmonoprime:\LFinGraph^\circ\to\LFinGraph^\circ$ given by  $
            \GCregmonoprime(\Lambda) := \sum_{\sigma\in\FinGraph} c_\sigma^{\regmono^{\prime}}(\Lambda)\ \sigma.$
            \end{enumerate}
\end{definition}

  \begin{remark}
  These functionals were defined in a fashion that resembles the iterated-integrals signature, i.e.
  \begin{align*}
    \lra{ \mathtt{GC}^{\dagger}(\Gamma) , \alpha } =  c^{\dagger}_\alpha(\Gamma).
  \end{align*}
  \end{remark}

\begin{example}
\begin{align*}
    \GChom(\ \tria) &= \e + 3 \vertex + 9 \vertex \vertex +  6 \edge + 27 \vertex \vertex \vertex + 18 \vertex \edge +  12 \cherry + 6 \tria + \cdots\\
    \GChomprime(\ \tria) &= \e + 3 \vertex + \frac{9}{2} \vertex \vertex +  3 \edge + \frac{9}{2} \vertex \vertex \vertex + 9 \vertex \edge +  6 \cherry + \tria + \cdots\\
       \GCmono(\ \tria) &= \e + 3 \vertex + 6 \vertex \vertex +  6 \edge + 6 \vertex \vertex \vertex + 6 \vertex \edge +  6 \cherry + 6\tria\\
       \GCmonoprime(\ \tria) &= \e + 3 \vertex + 3 \vertex \vertex +  3 \edge + \vertex \vertex \vertex + 3 \vertex \edge +  3 \cherry + \tria\\
       \GCregmono(\ \tria) &= \e + 3 \vertex + 6 \edge + 6\ \tria\\
       \GCregmonoprime(\ \tria) &= \e + 3 \vertex + 3 \edge + \ \tria
\end{align*}
\end{example}

$\GCmono,\GCmonoprime,\GCregmono$, $\GCregmonoprime$ can all be considered as changes of basis.
\begin{lemma}
\label{lemma:GC_linend}
The linear functions $\GCmono,\GCmonoprime,\GCregmono,\GCregmonoprime$ are linear automorphism.
\end{lemma}
\begin{proof}
We show $\GCmonoprime$. There are analogus arguments for the others.
We assign a total order on $\FinGraph$, first by number of vertices and then by number of edges, and resolve ties arbitrarily. For such an order, the matrix of $[\GCmonoprime]_{\FinGraph}$ is an invertible lower-triangular.
\begin{enumerate}
    \item $\lra{\GCmonoprime(\Lambda),\sigma} \not=0$ and $|V(\sigma)| = |V(\Lambda)|$ together with $|E(\sigma)| = |E(\Lambda)|$ implies $\sigma = \Lambda$. And hence, any choice of resolving ties is consistent with a triangular structure.
    \item $\lra{\GCmonoprime(\Lambda),\sigma} \not=0 $ implies that $\lrv{V(\sigma)} \leq \lrv{V(\Lambda)}$ and $\lrv{E(\sigma)} \leq \lrv{E(\Lambda)}$, i. e. the elements above the main diagonal are zero.
    \item It is invertible since the diagonal satisfies  $$\lra{\GCmonoprime(\Lambda),\Lambda}=1 \not= 0.$$
\end{enumerate}
\end{proof}

As a consequence of this lemma, $\GCmonoprime(\Lambda) = \GCmonoprime(\Lambda^{\prime})$ implies $\Lambda = \Lambda^{\prime}$.  This is an analog for graph counting to Lov\'{a}sz' result in \cite{lovasz1967operations} for $\GChom$.

\subsection{Multiplicative and comultiplicative properties}
\label{ssec:mult_comult_prop}

    We will now show that, given a fixed graph, the functionals $\GChom( \Lambda),\GCmono(\Lambda),$ and $\GCregmono(\Lambda)$ are characters over the appropriate Hopf algebras and furthermore satisfy a version of Chen's identity.

\begin{theorem}[Character property]
\label{thm:character property}

Suppose that $\Lambda \in\FinGraph$ is a fixed graph.

\begin{enumerate}
  \item 
    The functional $\GChom(\Lambda)$ is a character over $(\LFinGraph,\cdot_{\hom})$, i.e.
    \begin{align*}
       \LA \GChom(\Lambda), \sigma_{1} \RA
       \cdot
       \LA \GChom(\Lambda), \sigma_{2} \RA
       =
       \LA \GChom(\Lambda), \sigma_{1} \cdot_{\hom} \sigma_{2} \RA
    \end{align*}

  \item
  The functional $\GCmono(\Lambda)$ is a character over $(\LFinGraph,\cdot_{\mono})$, i.e.
    \begin{align*}
       \LA \GCmono(\Lambda), \sigma_{1} \RA
       \cdot
       \LA \GCmono(\Lambda), \sigma_{2} \RA
       =
       \LA \GCmono(\Lambda), \sigma_{1} \cdot_{\mono}\sigma_{2} \RA
    \end{align*}
     \item
     The functional $ \mathtt{GC}^{\regmono}(\Lambda)$ is a character over $(\LFinGraph,\cdot_{\regmono})$, i.e.
    \begin{align*}
       \LA \GCregmono(\Lambda), \sigma_{1} \RA
       \cdot
       \LA \GCregmono(\Lambda), \sigma_{2} \RA
       =
       \LA \GCregmono(\Lambda), \sigma_{1} \cdot_{\regmono} \sigma_{2} \RA.
    \end{align*}
\end{enumerate}
\end{theorem}

\begin{proof}~
  \begin{enumerate}
    \item This is a well known result, see for example \cite[Thm.~ 3.6]{lovasz1967operations}. The proof follows immediately by observing that there is a bijection between 
\begin{align*}
\{(f,g)| f \in \Hom(\sigma_{1},\Lambda), g \in \Hom(\sigma_{2},\Lambda)\} \to \Hom(\sigma_{1}\cdot_{\hom}\sigma_{2},\Lambda)
\end{align*}
\begin{align*}
(f,g) \mapsto h,
\end{align*}
where $h(a) = f(a)$ if $a \in V(\sigma_{1})$ and $h(a) = g(a)$, otherwise.

\item Take the set
\begin{align*}
 \{(f,g)| f \in \Mono(\sigma_{1},\Lambda), g \in \Mono(\sigma_{2},\Lambda)\}.
\end{align*}

First,
\begin{align*}
| \{(f,g)| f \in \Mono(\sigma_{1},\Lambda), g \in \Mono(\sigma_{2},\Lambda)\}| &=
|\{h \in \Hom(\sigma_{1} \cdot_{\hom} \sigma_{2},\Lambda)| h\evaluatedAt{\sigma_{1}}, h\evaluatedAt{\sigma_{2}}\;\text{are monomorphisms}\}|.
\end{align*}

Indeed, there is a bijection $(f,g) \mapsto h$ where $h$ is defined as

\begin{align*}
        h(x) = 
    \begin{cases}
f(x), \;\text{if}\; x \in V(\sigma_1)  \\
g(x), \;\text{if}\; x \in V(\sigma_2).
\end{cases}
\end{align*}

We can then use $iv.$ of \cref{cor:factorization} to see that
\begin{align*}
&|\{h \in \Hom(\sigma_{1} \cdot_{\hom} \sigma_{2},\Lambda) \mid h\evaluatedAt{\sigma_{1}}, h\evaluatedAt{\sigma_{2}}\;\text{are monomorphisms}, \Lambda\evaluatedAt{h(V(\sigma_{1} \cdot_{\hom} \sigma_{2}),h(E(\sigma_{1} \cdot_{\hom} \sigma_{2}))}\cong \gamma\}| \\&= \frac{1}{\lrv{\Aut(\gamma)}}\
    \lrv{\Big\{ h \in \RegEpi(\sigma_{1} \cdot_{\hom} \sigma_{2},\gamma)\mid h\evaluatedAt{\sigma_{1}}, h\evaluatedAt{\sigma_{2}}\;\text{are monomorphisms}\Big\}}
    \lrv{ \Mono(\gamma,\Lambda) }.
  \end{align*}

\item Take the set
\begin{align*}
 \{(f,g)| f \in \RegMono(\sigma_{1},\Lambda), g \in \RegMono(\sigma_{2},\Lambda)\}.
\end{align*}

Let $h \in \mor(\sigma_{1} \cdot_{\hom} \sigma_{2},\Lambda)$ defined as

\begin{align*}
        h(x) = 
    \begin{cases}
f(x), \;\text{if}\; x \in V(\sigma_1)  \\
g(x), \;\text{if}\; x \in V(\sigma_2).
\end{cases}
\end{align*}

We can then use $iii.$ of \cref{cor:factorization}.
  \end{enumerate}
\end{proof}
\begin{remark}
From the definitions of $\GChomprime$, $\GCmonoprime$ and $\GCregmonoprime$ and \cref{remark:defacto3} it follows immediately that
\begin{align*}
       \LA \GChomprime(\Lambda), \tau_{1} \RA
       \cdot
       \LA \GChomprime(\Lambda), \tau_{2} \RA
       &=
       \LA \GChomprime(\Lambda), \tau_{1} \cdot_{\hom^{\prime}} \tau_{2} \RA\\
       \LA \GChomprime(\Lambda), \sigma_{1} \RA
       \cdot
       \LA \GCmonoprime(\Lambda), \sigma_{2} \RA
       &=
       \LA \GCmonoprime(\Lambda), \sigma_{1} \cdot_{\mono^{\prime}} \sigma_{2} \RA\\
       \LA \GCmonoprime(\Lambda), \sigma_{1} \RA
       \cdot
       \LA \GCregmonoprime(\Lambda), \sigma_{2} \RA
       &=
       \LA \GCregmonoprime(\Lambda), \sigma_{1} \cdot_{\regmono^{\prime}} \sigma_{2} \RA
   \end{align*}
\end{remark}
The result regarding $\GCmonoprime(\Lambda)$ is known in the computer science literature, see for example \cite[Lemma 2]{maugis2020testing}. It was stated also by \cite{bib:BNH2015} using a representation in symmetric functions instead of graph algebras. In \cite{penaguiao2020pattern}, we have also a description for the character property of $\GCregmonoprime$. 

\begin{remark}
\label{remark:connectedcounting}
  Let $\Lambda \in \FinGraph$ and $\tau\in \FinGraph$, then 
  \begin{align*}
    &\LA \GCmonoprime(\Lambda), \tau \RA = \sum_{i \in I} a_{i} \prod_{j \in J}\LA \GCmonoprime(\Lambda),\theta_{ij} \RA.
    \end{align*}
  where $a_{i} \in \Q$, $\theta_{ij} \in \FinGraph^{0}$, i.e., they are \underbar{connected}. This is folklore in the computer science community, and it is a consequence of \cref{thm:AllIsomorphisms} and \cref{thm:character property}.

    The respective result for $\GChom$ is immediate given that
  \begin{align*}
    &\LA \GChom(\Upsilon), \sigma \RA = \prod_{i=1}^{n}\LA \GChom(\Upsilon),\sigma_{i} \RA,
\end{align*}
where $\sigma = \sigma_{1} \cdot_{\hom} \cdots \cdot_{\hom}  \sigma_{n}$.
  \end{remark}

Another property common for signatures is Chen's identity. This is the compatibility of the signature with respect to a product in the parameter, not in the argument. We first prove the following lemma. Chen's identity will then simply follow.

\begin{lemma}
    \label{thm.AbstractChen}
    Let $\lrp{\curlyH,m,\Delta}$ be a bicommutative unipotent Hopf algebra.
    If $\varphi,\varphi_1,\varphi_2\in\curlyH^*$ are characters, and 
    \begin{align*}
        \lra{\varphi,x} = \lra{\varphi_1,x} + \lra{\varphi_2,x}
    \end{align*}
    for all $x\in\mathtt{Prim}(\curlyH,\Delta)$, then for all $a\in \curlyH$
    \begin{align*}
        \lra{\varphi,a} &= 
        \lra{\varphi_1 \otimes \varphi_2 , \Delta(a) }.
    \end{align*}
\end{lemma}

\begin{proof}
    From the hypothesis, for $x\in\texttt{Prim}(\curlyH,\Delta)$,
    \begin{align*}
        \lra{\varphi,x} &= \lra{\varphi_1,x} + \lra{\varphi_2,x} 
        \\ &=  \lra{\varphi_1,x}\lra{\varphi_2,1_\curlyH} + \lra{\varphi_1,1_\curlyH}\lra{\varphi_2,x} 
        \\ &= \lra{\varphi_1\otimes \varphi_2 , x\otimes 1_\curlyH + 1_\curlyH \otimes x}
        \\&= \lra{\varphi_1\otimes \varphi_2 , \Delta(x) }.
    \end{align*}
    For the $m$-monomial $a= \prod_{i=1}^n x_i $,
    we use the character property and the previous computation.
    \begin{align*}
        \lra{\varphi,a} &= \prod_i \lra{\varphi,x_i} 
        \\&= 
        \prod_i \lra{\varphi_1\otimes \varphi_2 , \Delta(x_i) }.
    \end{align*}
    Now, since $\varphi_1\otimes \varphi_2$ is a character for $m_{\curlyH\otimes\curlyH}$, this implies that
    \begin{align*}
        \lra{\varphi,x} &= \prod_i \lra{\varphi_1\otimes \varphi_2 , \Delta(x_i) }
        \\&= 
        \lra{\varphi_1\otimes \varphi_2 , \prod_i \Delta(x_i) }
        \\&= 
        \lra{\varphi_1\otimes \varphi_2 , \Delta\Big(\prod_i x_i\Big) }.
    \end{align*}
    The result follows from $\curlyH$ being freely generated by $\lrb{x_i}\subset\texttt{Prim}(\curlyH,\Delta)$.
\end{proof}

\begin{theorem}[Chen's identity]
  \label{thm:chen}
  Suppose that $\sigma,\Lambda,\Psi\in\FinGraph$ are fixed. If the notation $\dagger$ represents any of the symbols in $\lrb{\hom,\mono,\regmono}$, and $\ddagger$ represents $\lrb{\hom^{\prime},\mono^{\prime},\regmono^{\prime}}$ then it holds 
  \begin{align*}
        \lra{\mathtt{GC}^{\dagger}(\Lambda\cdot_{\hom}\Psi),\sigma} &= \lra{\mathtt{GC}^{\dagger}(\Lambda)\otimes\mathtt{GC}^{\dagger}(\Psi),\Delta_{\hom^{\prime}}(\sigma)}\\
     \lra{\mathtt{GC}^{\ddagger}(\Lambda\cdot_{\hom}\Psi),\sigma} &= \lra{\mathtt{GC}^{\ddagger}(\Lambda)\otimes\mathtt{GC}^{\ddagger}(\Psi),\Delta_{\hom}(\sigma)},
  \end{align*}

\end{theorem}
\begin{example}
\begin{align*}
\langle\GChom(\Lambda \cdot_{\hom} \Psi), \edge \edge\rangle&=\langle\GChom(\Lambda) \otimes \GChom(\Psi), \Delta_{\hom^{\prime}}(\edge \edge\:)\rangle\\
 \end{align*}
\end{example}

\begin{remark}
  Chen's identity for iterated sums (or integrals)
  yields a procedure of efficient calculation.
  \Cref{thm:chen} is very weak in this regard,
  since it does not allow to simplify the calculation of $\GChom(\Lambda)$ if
  $\Lambda$ is connected.
\end{remark}
\begin{proof}
    For any connected graph $\tau$, i.e. $\tau \in \mathtt{Prim}(\LFinGraph,\Delta_{\hom^{\prime}})$, the following equation holds true for $\dagger$ in $\lrb{\hom,\mono,\regmono}$.
    \begin{align*}
        \lra{\mathtt{GC}^{\dagger}(\Lambda\cdot_{\hom}\Psi),\sigma} &= \lra{\mathtt{GC}^{\dagger}(\Lambda),\sigma} +
        \lra{\mathtt{GC}^{\dagger}(\Psi),\sigma} .
    \end{align*}
    The proof follows from \cref{thm.AbstractChen}, since $\mathtt{GC}^{\dagger}(\Lambda\Psi),\mathtt{GC}^{\dagger}(\Lambda)$ and $\mathtt{GC}^{\dagger}(\Psi)$ are all characters for $\lrp{\LFinGraph,\dagger,\Delta_{\hom^{\prime}}}$. For the other three, the proof is analogous.
\end{proof}

\begin{remark}
Note that we are showing the well known fact that counting graph homomorphisms can be ``reduced'' to counting connected graphs on connected graphs. Using \cref{thm:character property} and \cref{thm:chen}, consider $\Lambda = \Lambda_{1}\cdots\Lambda_{n}$, and $\sigma = \sigma_{1}\cdots \sigma_{m}$ where the $\Lambda_{i}$'s and the $\sigma_{i}$'s are all connected, then, as an example, it holds:
\begin{align*}
    \lra{\GChom\left(\Lambda_{1}\cdots\Lambda_{n}\right), \sigma_{1}\cdots \sigma_{m}}&= \prod_{j=1}^{m}\lra{\GChom\left(\Lambda_{1}\cdots\Lambda_{n}\right),\sigma_{j}}\\
    &= \prod_{j=1}^{m}\lra{\bigotimes_{i=1}^{n}\left(\GChom(\Lambda_{i})\right),\Delta_{\hom^{\prime}}^{n}\sigma_{j}}\\
      &= \prod_{j=1}^{m}\Big(\sum_{i=1}^{n}\lra{\GChom\left(\Lambda_{i}\right),\sigma_{j}}\Big).
\end{align*}
\end{remark}

\begin{table}[H]
\centering
\begin{tabular}{|c|c|c|}
\hline
& Character property        & Chen's identity         \\ \hline $\GChom(\Lambda)$ & $\cdot_{\hom}$  &  $\Delta_{\hom^{\prime}}$                   \\ \hline
                   $\GChomprime(\Upsilon)$ & $\cdot_{\hom^{\prime}}$                        &  $\Delta_{\hom}$                                            \\ \hline
$\GCmono(\Lambda)$ & $\cdot_{\mono}$ & $\Delta_{\hom^{\prime}}$ \\ \hline
$\GCmonoprime(\Lambda)$ & $\cdot_{\mono^{\prime}}$                           & $\Delta_{\hom}$             \\ \hline
{}$\GCregmono(\Upsilon)$ & $\cdot_{\regmono}$                         & $\Delta_{\hom^{\prime}}$                                                 \\ \hline
$\GCregmonoprime(\Upsilon)$ & $\cdot_{\regmono^{\prime}}$                        & $\Delta_{\hom}$                                                \\ \hline
\end{tabular}
\caption*{Multiplicative (signature or character property) and comultiplicative properties (Chen) .}
\end{table}

\subsection{Translating between the counting characters}
\label{ssec:translating}

Here we show the equivalence of the different counting procedures. This equivalence is used in practice due to the computational efficiency of counting graph homomorphisms, see for instance \cite[Section 1.4]{curticapean2017homomorphisms}, and it is well known. Later on, we will show that these maps are Hopf algebra isomorphisms.

\begin{definition}
Define the linear maps $\LFinGraph \to \LFinGraph$:
\begin{align}
\label{eq:trans_regmono_mono}
  \Phi_{\regmono \leftarrow \mono}(\tau) 
    :&= \sum_\sigma\ \frac{\lrv{ \Mono(\tau,\sigma) \cap \Epi(\tau,\sigma) }}{|\Aut(\sigma)|} \ \sigma, \\
  \label{eq:trans_mono_hom}
  \Phi_{\mono \leftarrow \hom}(\tau)
  :&= \sum_\sigma \frac{\lrv{ \RegEpi(\tau,\sigma) }}{|\Aut(\sigma)|}\ \sigma, \\
  \Phi_{\regmono \leftarrow \hom}(\tau) :&= \sum_\sigma \frac{\lrv{ \Epi(\tau,\sigma) }}{|\Aut(\sigma)|}\ \sigma, \\
   \Phi_{\hom^{\prime} \leftarrow \hom}(\tau) &=  \Phi_{\mono^{\prime} \leftarrow \mono}(\tau) =  \Phi_{\regmono^{\prime} \leftarrow \regmono}(\tau) := |\Aut{(\tau)}| \tau.
  \end{align}
\end{definition}

\begin{example}
\label{exampleTranscounting}
\begin{align*}
  \Phi_{\regmono \leftarrow \mono}(\cherry) &= \cherry + \tria\\ 
    \Phi_{\hom \leftarrow \mono}(\cherry) &= \edge + \cherry\\ 
      \Phi_{\regmono \leftarrow \hom}(\cherry) &= \edge + \cherry + \tria\\ 
    \end{align*}
\end{example}

\begin{remark}
\label{thm:rewriteViEr}
For all these maps it holds that connected graphs are mapped to connected graphs. 
 \end{remark}

These maps allow to translate between the different counting operations.

\begin{theorem}
\label{thm:trans_counting}
For $\tau, \Lambda \in \FinGraph$, it holds that
\begin{align}
\label{eq:viei}
\Big\langle  \GCmono( \Lambda ), \tau \Big\rangle &= \Big\langle  \GCregmono( \Lambda ), \Phi_{\regmono \leftarrow \mono}(\tau) \Big\rangle, \\
  \label{eq:eihomdp}
  \Big\langle  \GChom( \Lambda ), \tau \Big\rangle &= \Big\langle  \GCmono( \Lambda ), \Phi_{\mono \leftarrow \hom}(\tau) \Big\rangle, \\ 
    \label{eq:vihomdp}
  \Big\langle  \GChom( \Lambda ), \tau \Big\rangle &= \Big\langle  \GCregmono( \Lambda ), \Phi_{\regmono \leftarrow \mono}\circ \Phi_{\mono \leftarrow \hom}(\tau) \Big\rangle,
  \\ \label{eq:hom_homdp}
  \Big\langle  \GChom( \Lambda ), \tau \Big\rangle &= \Big\langle  \GChomprime( \Lambda ), \Phi_{\hom^{\prime}\leftarrow \hom}(\tau) \Big\rangle,
    \\ \label{eq:monei}
  \Big\langle  \GCmono( \Lambda ), \tau \Big\rangle &= \Big\langle  \GCmonoprime( \Lambda ), \Phi_{\mono^{\prime}\leftarrow \mono }(\tau) \Big\rangle,
    \\ \label{eq:regmovi}\Big\langle  \GCregmono( \Lambda ), \tau \Big\rangle &= \Big\langle  \GCregmonoprime( \Lambda ), \Phi_{\regmono^{\prime}\leftarrow \regmono }(\tau) \Big\rangle.
\end{align}
\end{theorem}
\begin{proof}
  Using \Cref{cor:factorization}.i we have
  \begin{align*}
    \Big\langle  \GCmono( \Lambda ), \tau \Big\rangle
    &=\lrv{ \Mono(\tau,\Lambda) } 
     \\
    &=  \sum_\sigma \lrv{\lrb{\phi \in \Mono(\tau,\Lambda) \mid \sigma \cong \Lambda_{ \phi(V(\tau)) }}} \\
    &=\sum_\sigma 
    \lrv{\Epi(\tau,\sigma) \cap \Mono(\tau,\sigma)}
    \frac
    {\lrv{\RegMono(\sigma,\Lambda)}}
    { \lrv{\Aut(\sigma)} } \\
    &=
    \Big\langle  \GCregmono( \Lambda ), \Phi_{\regmono \leftarrow \mono}(\tau) \Big\rangle.
  \end{align*}
  This proves \eqref{eq:viei}.

  \bigskip
  Using \Cref{cor:factorization}.iv we have
  \begin{align*}
    \Big\langle  \GChom( \Lambda ), \tau \Big\rangle
    &=\lrv{ \Hom(\tau,\Lambda) } \\
    &=\lrv{\lrb{\phi \in \Hom(\tau,\Lambda) \mid \sigma \cong \Lambda\evaluatedAt{\phi(V(\tau)),\phi(E(\tau))}} } \\
    &= \sum_\sigma \frac{1}{\lrv{\Aut(\sigma)}} \lrv{\RegEpi(\tau,\sigma)}\ \lrv{\Mono(\sigma,\Lambda)} \\
    &= \sum_\sigma  \lrv{\RegEpi(\tau,\sigma)}\ \frac{\lrv{\Mono(\sigma,\Lambda)}}{\lrv{\Aut(\sigma)}} \\
    &=
    \Big\langle  \GCmono( \Lambda ), \Phi_{\mono\leftarrow \hom} (\tau) \Big\rangle.
  \end{align*}
  This proves \eqref{eq:eihomdp}, and  \eqref{eq:vihomdp} follows by composition.  \eqref{eq:hom_homdp}, \eqref{eq:monei} and \eqref{eq:regmovi} are obvious.
\end{proof}

\begin{remark}
Using \Cref{cor:factorization}.ii
\begin{align*}
\Phi_{\regmono \leftarrow \mono}\circ \Phi_{\mono \leftarrow \hom}&= \sum_{\sigma,\gamma}\frac{1}{|\Aut(\sigma)||\Aut(\gamma)|}|\RegEpi(\tau,\sigma)||\Mono(\sigma,\gamma)\cap\Epi(\sigma,\gamma)|\gamma\\
&= \sum_{\gamma}\frac{1}{|\Aut(\gamma)|}|\Epi(\tau,\gamma)|\gamma = \Phi_{\regmono \leftarrow \hom}
\end{align*}
\end{remark}

\begin{remark}~
  These results are already known in the literature for enumerative combinatorics. 
  \begin{enumerate}
    \item A result traditionally attributed to Lov\'asz, which can be read in \cite[Equation (8)]{curticapean2017homomorphisms}, is the same as 
    \begin{align*}
     \Big\langle  \GChom( \Lambda ), \tau \Big\rangle = \Big\langle  \GCmono( \Lambda ), \Phi_{\mono \leftarrow \hom}(\tau) \Big\rangle.
    \end{align*} In \cite{curticapean2017homomorphisms}, $\operatorname{Spasm}(\tau)$ denotes the set of graphs $\sigma \in \FinGraph$ such that
    $\RegEpi(\tau,\sigma) \neq \emptyset$.
    \item By \eqref{eq:viei} one has
      \begin{align*}
        |\Mono(\tau,\Lambda)|
        &= \langle \GCmono(\Lambda), \tau \rangle \\
        &= \langle \GCregmono(\Lambda), \sum_\sigma \lrv{ \Mono(\tau,\sigma) \cap \Epi(\tau,\sigma) }\ \sigma \rangle \\
        &= \sum_\sigma \lrv{ \Mono(\tau,\sigma) \cap \Epi(\tau,\sigma) } \frac{\lrv{ \RegMono(\sigma,\Lambda)}}{|\Aut(\sigma)|} ,
      \end{align*}
      which can be shown to be equal to \cite[Equation (5.15)]{lovasz2012large}, but the formulation is slightly different. In \cite[Equation (5.15)]{lovasz2012large}, it is stated that $\forall \sigma,\Upsilon \in \FinGraph$, and $|V(\sigma)|=n$,
      \begin{align*}
          |\Mono(\sigma,\Upsilon)| = \sum_{
E(\sigma) \subseteq X \subseteq E(K)
          }|\RegMono(K_{V(\sigma),X},\Upsilon)|
      \end{align*}
      where $K$ denotes the complete graph on the vertices of $\sigma$, and $K_{V(\sigma),X}:=(V(\sigma),X)$. To obtain the counting of the regular monomorphisms in terms of the monomorpshims, M\"{o}bius inversion can be used: see  \cite[Equation (5.17)]{lovasz2012large}.
      
      \item Recreating \cite[Eq. 5.16]{lovasz2012large}, we obtain
      \begin{align}
      \label{eq:hom-inj-lov}
      \langle\GChom(\Lambda), \tau\rangle &= \sum_{\pi \in P_{V(\tau)}}|\Mono(\tau/\pi,\Lambda)|
      \end{align}
      where $\pi$ runs over the set partitions of $V(\tau)$ and $\tau/\pi$ is the simple graph obtained by identifying the vertices in each block. Note that for $\tau/\pi \cong \sigma$, it holds that
      \begin{align*}
       \frac{|\RegEpi(\tau,\sigma)|}{|\Aut(\sigma)|} &= |\{\pi \in P_{V(\tau)}: \tau/\pi \cong \sigma\}|
      \end{align*}
      and therefore
       \begin{align*}
      |\Hom(\tau, \Lambda)| &= \sum_{\pi \in P_{V(\tau)}}|\Mono(\tau/\pi,\Lambda)|\\ &= \sum_{\sigma} |\{\pi:\, \tau/\pi \cong \sigma\}||\Mono(\sigma,\Lambda)| \\&= \sum_{\sigma} \frac{|\RegEpi(\tau,\sigma)|}{|\Aut(\sigma)|}|\Mono(\sigma,\Lambda)|\\
      &=  \Big\langle  \GCmono( \Lambda ), \Phi_{\mono \leftarrow \hom} \tau \Big\rangle.
      \end{align*}

M\"{o}bius inversion also plays a role since \cref{eq:hom-inj-lov} can be written more explicitly in the language of incidence algebra as follows:
\begin{align*}
g(\hat{0},\hat{1}) &= (\zeta * f )(\hat{0},\hat{1})= \sum_{\hat{0} \le \pi \le \hat{1}}\zeta(\hat{0},\pi)f(\pi,\hat{1})\\
\end{align*}
where $\pi$ ranges over the lattice of partitions of $V(\tau)$ and $\hat{0}$ and $\hat{1}$ are the finest and the coarsest partitions respectively. Let $g(\pi,\hat{1}):=|\Hom(\tau/\pi,\Lambda)|$ and $f(\pi,\hat{1}):=|\Mono(\tau/\pi,\Lambda)|$. Moreover $g(\cdot,\pi^{\prime})=0$ and $f(\cdot,\pi^{\prime})=0$ as well if $\pi^{\prime} \neq \hat{1}$ . We can then recover $f$ using the M\"{o}bius function:

\begin{align*}
f(\hat{0},\hat{1}) &= (\mu * g )(\hat{0},\hat{1})= \sum_{\hat{0} \le \pi \le \hat{1}}\mu(\hat{0},\pi)g(\pi,\hat{1})
\end{align*}

which is

\begin{align*}
|\Mono(\tau,\Lambda)|   &= \sum_{\pi \in P_{V(\tau)}}\mu(\hat{0},\pi)\langle\GChom(\Lambda), \tau/\pi\rangle
\end{align*}

which is \cite[Eq. 5.18]{lovasz2012large}. An explicit formula for the M\"{o}bius function is given in  \cite[Eq. A.2]{lovasz2012large}:
\begin{align*}
    \mu(\hat{0},\pi) = (-1)^{|V(\tau)|-|\pi|}\prod_{B \in \pi}(|B|-1)!
\end{align*}

\end{enumerate}

\end{remark}

These translating maps behave well with respect to the underlying Hopf structure. 

\begin{theorem}
\label{thm:combiIso}

The maps of \cref{thm:trans_counting}

\begin{enumerate}
\item $\Phi_{\regmono \leftarrow \mono}: (\cdot_{\mono},\Delta_{\hom^{\prime}}) \to (\cdot_{\regmono},\Delta_{\hom^{\prime}})$
\item $\Phi_{\mono \leftarrow \hom}: (\cdot_{\hom},\Delta_{\hom^{\prime}}) \to (\cdot_{\mono},\Delta_{\hom^{\prime}})$
\item $\Phi_{\regmono \leftarrow \hom}: (\cdot_{\hom},\Delta_{\hom^{\prime}}) \to (\cdot_{\regmono},\Delta_{\hom^{\prime}})$
\item $\Phi_{\hom^{\prime} \leftarrow \hom}: (\cdot_{\hom},\Delta_{\hom^{\prime}}) \to (\cdot_{\hom^{\prime}},\Delta_{\hom}),$
\item $\Phi_{\mono^{\prime} \leftarrow \mono}: (\cdot_{\mono},\Delta_{\hom^{\prime}}) \to (\cdot_{\mono^{\prime}},\Delta_{\hom})$
\item $\Phi_{\regmono^{\prime} \leftarrow \regmono}: (\cdot_{\regmono},\Delta_{\hom^{\prime}}) \to (\cdot_{\regmono^{\prime}},\Delta_{\hom})$
\end{enumerate}
are Hopf algebra isomorphisms.
\end{theorem}
\begin{remark}
We can summarize the isomorphisms in the following picture: 
\begin{center}
\begin{tikzcd}
(\cdot_{\hom},\Delta_{\hom^{\prime}})  \arrow[dddd,"\Phi_{\hom \leftarrow \hom^{\prime}}",shift left=1.5ex,leftarrow]\arrow[dddd,"\Phi_{\hom^{\prime} \leftarrow \hom}"',shift right=1.5ex]   \arrow[rrrr,"\Phi_{\mono \leftarrow \hom}",shift left =1.5ex] \arrow[rrrr,"\Phi_{\hom \leftarrow \mono}"', shift right =1.5ex,leftarrow]
&&&& (\cdot_{\mono},\Delta_{\hom^{\prime}})  \arrow[dddd,"\Phi_{\mono \leftarrow \mono^{\prime}}",shift left=1.5ex,leftarrow]\arrow[dddd,"\Phi_{\mono^{\prime}\leftarrow \mono}"',shift right=1.5ex]  \arrow[rrrr,"\Phi_{\regmono \leftarrow \mono}",shift left =1.5ex] \arrow[rrrr,"\Phi_{\mono \leftarrow \regmono}"', shift right =1.5ex,leftarrow]    &&&& (\cdot_{\regmono},\Delta_{\hom^{\prime}}) \arrow[dddd,"\Phi_{\regmono \leftarrow \regmono^{\prime}}",shift left=1.5ex,leftarrow]\arrow[dddd,"\Phi_{\regmono^{\prime} \leftarrow \regmono}"',shift right=1.5ex]    \\&&&&\\&&&&\\&&&&\\
(\cdot_{\hom^{\prime}},\Delta_{\hom})  \arrow[rrrr,"\Phi_{\hom^{\prime} \leftarrow \mono^{\prime}}",shift left =1.5ex,leftarrow] \arrow[rrrr,"\Phi_{\mono^{\prime} \leftarrow \hom^{\prime}}"', shift right =1.5ex] 
&&&& (\cdot_{\mono^{\prime}},\Delta_{\hom})   \arrow[rrrr,"\Phi_{\mono^{\prime} \leftarrow \regmono^{\prime}}",shift left =1.5ex,leftarrow] \arrow[rrrr,"\Phi_{\regmono^{\prime} \leftarrow \mono^{\prime}}"', shift right =1.5ex] &&&&(\cdot_{\regmono^{\prime}},\Delta_{\hom}) 
\end{tikzcd}
\end{center}
where we add all the inverse maps and
where $\Phi_{\mono^{\prime} \leftarrow \hom^{\prime}}:=\Phi_{\mono^{\prime} \leftarrow \mono}  \circ \Phi_{\mono \leftarrow \hom} \circ\Phi_{\hom \leftarrow \hom^{\prime}}$ and $\Phi_{\regmono^{\prime} \leftarrow \mono^{\prime}}:=\Phi_{\regmono^{\prime} \leftarrow \regmono}\circ \Phi_{\regmono \leftarrow \mono}\circ\Phi_{\mono \leftarrow \mono^{\prime}}$. Compare:
\begin{align}
 \label{eq:trans_regmonoprimehomprime}
 \Phi_{\regmono^{\prime} \leftarrow \mono^{\prime}}(\tau) 
&= \frac{1}{|\Aut(\tau)|}\sum_\sigma\ \lrv{ \Mono(\tau,\sigma) \cap \Epi(\tau,\sigma) }  \sigma,\\
\label{eq:trans_monoprimehomprime}
      \Phi_{\mono^{\prime} \leftarrow \hom^{\prime}}(\tau)
  &= \frac{1}{|\Aut(\tau)|}\sum_\sigma \lrv{ \RegEpi(\tau,\sigma) }\ \sigma, 
 \end{align}
with \cref{eq:trans_regmono_mono} and \cref{eq:trans_mono_hom} respectively. The vertical arrows (pointing in the direction) coincide as linear maps.

\end{remark}

\begin{remark}
    Note that $\Phi_{\hom^{\prime} \leftarrow \hom},\Phi_{\mono \leftarrow \hom}, \Phi_{\mono^{\prime} \leftarrow \hom}:=\Phi_{\mono^{\prime} \leftarrow \mono} \circ \Phi_{\mono \leftarrow \hom}, \Phi_{\regmono\leftarrow \hom}$ and $\Phi_{\regmono^{\prime} \leftarrow \hom}:=\Phi_{\regmono^{\prime} \leftarrow \regmono} \circ \Phi_{\regmono\leftarrow \hom}$ are \emph{not} the isomorphisms listed in \cref{thm:AllIsomorphisms}. As an example:
    \begin{align*}
        \widehat{\psi}_{(\cdot_{\hom^{\prime}},\Delta_{\hom})}(\sigma \cdot_{\hom} \tau) = \sigma \cdot_{\hom^{\prime}} \tau = \frac{|\Aut(\sigma \cdot_{\hom} \tau)|}{|\Aut(\sigma)||\Aut(\tau)|}\; \sigma \cdot_{\hom} \tau  \neq |\Aut(\sigma \cdot_{\hom} \tau)| \sigma \cdot_{\hom} \tau = \Phi_{\hom^{\prime} \leftarrow \hom}(\sigma \cdot_{\hom} \tau)
    \end{align*}
\end{remark}

\begin{proof}
We want to use \cref{thm:PrimitiveIsEnough}, and divide the proof into three parts.  
\begin{enumerate}
    \item  The functions are proved to be algebra homomorphisms. 
    
    \item We verify that they are linear bijections when restricted and corestricted to $\langle\FinGraph^0\rangle_{\Q}$.
    
    \item 
    From \cref{thm:AllIsomorphisms}, we know that the algebras are commutative free on the primitive elements. Hence all algebra homomorphisms are natural extensions~(\cref{def:natural_extension}) of linear functions defined on the generators (the primitive elements with respect to either $\Delta_{\hom}$ or $\Delta_{\hom^{\prime}}$, i.e. the connected graphs).
    By \cref{thm:PrimitiveIsEnough}, these functions are then Hopf algebra isomorphisms.
\end{enumerate}

To show i., let $\tau_{1},\tau_{2} \in \LFinGraph$, and write
\begin{align*}
  \langle {\GCregmono}(\Lambda),\Phi_{\regmono \leftarrow \mono}(\tau_{1} \cdot_{\mono} \tau_{2})\rangle &=
  \langle \GCmono(\Lambda),\tau_{1} \cdot_{\mono} \tau_{2}\rangle \\ &=
  \langle \GCmono(\Lambda),\tau_{1}\rangle\langle \GCmono(\Lambda),\tau_{2}\rangle  \\ &=\langle{\GCregmono}(\Lambda),\Phi_{\regmono \leftarrow \mono}(\tau_{1})\rangle\langle{\GCregmono}(\Lambda),\Phi_{\regmono \leftarrow \mono}(\tau_{2})\rangle  \\ &=
  \langle{\GCregmono}(\Lambda),\Phi_{\regmono \leftarrow \mono}(\tau_{1}) \cdot_{\regmono^{\prime}} \Phi_{\regmono \leftarrow \mono}(\tau_{2})\rangle.
\end{align*}

Since $\GCregmono:\LFinGraph^{\circ} \to \LFinGraph^{\circ}$ is a linear automorphism, see \cref{lemma:GC_linend}, we can conclude that $\forall \tau_{1},\tau_{2}\in \LFinGraph$
\begin{align*}
    \Phi_{\regmono \leftarrow \mono}(\tau_{1} \cdot_{\mono} \tau_{2}) &=  \Phi_{\regmono \leftarrow \mono}(\tau_{1}) \cdot_{\regmono^{\prime}}\Phi_{\regmono\leftarrow \mono}(\tau_{2})
\end{align*}
For $\Phi_{\regmono \leftarrow \hom}$, the argument is analogous. For $\Phi_{\mono \leftarrow \hom}$ we use that $\GCmono$ is also a linear automorphism, see \cref{lemma:GC_linend}. For $\Phi_{\hom^{\prime} \leftarrow \hom}$, $\Phi_{\mono^{\prime} \leftarrow \mono}$, $\Phi_{\regmono^{\prime} \leftarrow \regmono}$ the proof is obvious.

Recall that $$\mathtt{Prim}(\LFinGraph,\Delta_{\hom^{\prime}})=\langle\FinGraph^0\rangle_{\Q}=\mathtt{Prim}(\LFinGraph,\Delta_{\hom})$$ as proved in~\cref{lemma:primitive_elements}. Clearly
\begin{align*}
  \Phi_{\regmono \leftarrow \mono}(\tau) 
    :&= \sum_\sigma\ \frac{\lrv{ \Mono(\tau,\sigma) \cap \Epi(\tau,\sigma) }}{|\Aut(\sigma)|} \ \sigma, \\
  \Phi_{\mono \leftarrow \hom}(\tau)
  :&= \sum_\sigma \frac{\lrv{ \RegEpi(\tau,\sigma) }}{|\Aut(\sigma)|}\ \sigma, \\
 \end{align*}
map connected graphs to connected graphs, and the same holds for $\Phi_{\regmono \leftarrow \hom}$. If we order the graphs by the number of vertices, and then solve ties arbitrarily, the representation matrix associated to $\Phi_{\mono \leftarrow \hom}$ is upper triangular, with diagonal entries all different from zero. If we order the graphs by the number of vertices, and then for each fixed positive integer $n$ order the graphs with $n$ vertices by the number of edges and then solve ties arbitrarily, the representation matrix associated to $\Phi_{\regmono \leftarrow \mono}$ is lower triangular, with diagonal entries all different from zero. We then use \cref{thm:PrimitiveIsEnough} and we are done.  It follows that by composition that $\Phi_{\regmono \leftarrow \hom}$ is also a Hopf algebra isomorphisms. For the other maps one can verify immediately that they are Hopf isomorphisms.
\end{proof}

\section{Conclusions and outlook}
\label{sec:conclusion}

This work explored different definitions of subgraph counting
(related to edge-restriction subgraphs, to vertex-induction, and to homomorphisms)
from the point of view of combinatorial Hopf algebra. These perspectives are equivalent in the sense that the corresponding Hopf algebras are isomorphic, where the isomorphy
can be chosen to respect the counting operations.
The counting functions, parameterized by the counted-in graphs, are interpreted as 
signature-type objects (a concept from rough path theory).

As observed in \cref{remark:connectedcounting}, counting occurrences can be expressed as polynomials of the counting occurrences on connected graphs alone. This result can be traced back to the classical articles of Whitney \cite{whitney1932coloring} and Biggs \cite{biggs1978cluster}. However, we remark that these two articles express the counting of \textit{decomposable} graphs in terms of \textit{indecomposable} graphs. For instance, the following relation
$$
    c_{\!\!\tria\hspace{-0.2cm} \edge}(\Lambda) = c_{\!\!\tria}\!\!(\Lambda)( c_{\edge}(\Lambda) -3).
$$
appears in \cite{biggs1978cluster}, which cannot be obtained from the product treated in the current work. While the concepts of indecomposable and connected graphs are tightly related, the precise algebraic relations remain to be explored in further work.

We also describe Chen's identity -- the comultiplicative property of graph counting. It permits the simplification of computations for disconnected graphs in terms of their connected components. However, we remark that it does \emph{not} decompose connected graphs into smaller objects. This is deemed important in practice since graph counting problems usually appear on large connected sample graphs. It is an open problem to find a more practical type of Chen's identity.

\printbibliography

\appendix

\section{Appendix}
\label{sec:appendix}

\subsection{Background on Hopf algebras}
\label{ssec:back_Hopf}

\subsubsection{Hopf algebras}
\label{sssec:Hopf}

This section is based on \cite{grinberg2020hopf} and \cite{bib:manchon2008hopf}. Here we give the definition of Hopf algebra. We start by recalling the basic notions necessary for their definition.

\subsubsection{Algebras and coalgebras}
\label{ssec:AlgCoalg}

Let $\ka$ be a field (of characteristic zero). All vector spaces considered here are $\ka$-vector spaces and all linear maps are $\ka$-linear maps. We now give the definition of (associative) algebra.

\begin{definition}[Algebra]
A $\ka$-\textbf{algebra} $A$ is a vector space equipped with a  \textbf{product} $m~:~A~ \otimes~A~\to~A$ and a linear functional $u:\ka \to A$ called \textbf{unit} for the which two conditions hold.
\allowdisplaybreaks
\begin{align*}
   &m \circ (m \otimes \id) =    m \circ (\id \otimes m)\\
   & m \circ (\id \otimes u) \circ l = \id =  m \circ (u \otimes \id) \circ l^{\prime} .
\end{align*}
The first condition says that the product is \textbf{associative}, while the second means that there is a unit element in $A$. Here $l:A \to A \otimes \ka$ and $l^{\prime}: A \to \ka \otimes A,$ are isomorphisms sending $a \mapsto a \otimes 1_{\ka}$ and $a \mapsto 1_{\ka} \otimes a$, respectively. Note that in the following we will write $ab=a \cdot_{\!\!\scriptscriptstyle{A}} b :=m(a \otimes b)$. The unit map is simply $u(1_{\ka}) = 1_{A}$.
\end{definition}

The definition of coalgebra can be seen as the dual notion of algebra. 

\begin{definition}[Coalgebra]
A $\ka$-\textbf{coalgebra} $C$ is a vector space equipped with two linear maps: a \textbf{coproduct} $\Delta:C \to C \otimes C$ and a \textbf{counit} $\varepsilon: C \to \ka$ such that the following holds:
\allowdisplaybreaks
\begin{align}
   & (\Delta \otimes \id) \circ \Delta = (\id \otimes \Delta) \circ \Delta \label{coalg1}\\
   & t \circ (\id \otimes \varepsilon) \circ \Delta = \id = t^{\prime} \circ (\varepsilon \otimes \id) \circ \Delta \label{coalg2}.
\end{align}
The first condition is called \textbf{coassociativity}. Here $t:C \otimes \ka \to C$, $t^{\prime}: \ka \otimes C \to C,$ are isomorphisms mapping $c \otimes 1_{\ka} \mapsto c$ respectively $1_{\ka} \otimes c \mapsto c$. Note that for a given coproduct $\Delta$ there is at most one map $\varepsilon$ such that the previous condition holds.
\end{definition}

\begin{definition}[Commutative, Cocommutative]
We say that an algebra $A$ is commutative if $m_A = m_A \circ \tau$, where $\tau: A \otimes A \to A \otimes A,\tau(a_{1} \otimes a_{2})=a_{2} \otimes a_{1}$. A coalgebra $C$ is cocommutative if $\Delta_C = \tau \circ \Delta_C$.
\end{definition}

For the coproduct, we will use \textbf{Sweedler's notation}.
\begin{align*}
    \Delta(c) = \sum_{(c)}c_{(1)} \otimes c_{(2)} = \sum c_{(1)} \otimes c_{(2)}.
\end{align*}
We now introduce homomorphisms for algebras and coalgebras. 

\begin{definition} A linear map $\varphi:(A,m_A)\to (B,m_B)$ is an \textbf{algebra homomorphism} or an \textbf{algebra map} whenever
\begin{align*}
    \varphi \circ m_A = m_B \circ ( \varphi \otimes \varphi).
\end{align*}
Additionally, $\varphi\circ u_A = u_B $ for unital associative algebras. A \textbf{coalgebra map} or \textbf{coalgebra homomorphism} is a linear function $\psi:(C,\Delta_C) \to (D,\Delta_D)$ which preserves the coalgebra structure.
\begin{align*}
    (\psi\otimes\psi)\circ \Delta_C = \Delta_D \circ \psi. 
\end{align*}
If the coalgebra is counital, then $ \varepsilon_D \circ \psi = \varepsilon_C.$ 
\end{definition}

Isomorphisms of algebras and coalgebras are bijective homomorphisms such that the inverses are also morphisms. We have the following result.

\begin{lemma}
\label{thm.invertibleHomomorphisms}
Let $(A_{1},m_{1}),(A_{2},m_{2})$ be algebras and let $\phi:A_{1} \to A_{2}$ be a homomorphism. In case $\phi$ admits a compositional inverse, $\phi$ is an isomorphism. Moreover, let $(C_{1},\Delta_{1})$, $(C_{2},\Delta_{2})$ be coalgebras and let $\psi:C_{1} \to C_{2}$ be a homomorphism. In case $\psi$ admits a compositional inverse, $\psi$ is an isomorphism. 
\end{lemma}

\begin{proof}
    The definition of algebra homomorphism implies that $\phi \circ m_1 = m_2 \circ (\phi\otimes\phi).$
    Composing by $\phi^{-1}$ on the left yields
        $m_1 = \phi^{-1}\circ \phi \circ m_1 = \phi^{-1} \circ m_2 \circ (\phi\otimes\phi).$
    The result follows from composing the last equation with $\phi^{-1}\otimes\phi^{-1}$ on the right 
    \begin{align*}
        m_1 \circ (\phi^{-1}\otimes\phi^{-1}) = \phi^{-1}\circ m_2.
    \end{align*}
    For coalgebras,
    \allowdisplaybreaks
    \begin{align*}
        \Delta_1 = (\psi^{-1}\otimes\psi^{-1}) \circ (\psi\otimes \psi)\circ 
        \Delta_1 = (\psi^{-1}\otimes\psi^{-1}) \circ \Delta_2 \circ \psi,
    \end{align*}
    and the result follows once we compose on the right with $\psi^{-1}$.
\end{proof}

\subsubsection{Bialgebras}
\label{ssec:bialg}

\begin{definition}[Tensor product of two algebras] Given two coalgebras $A$, $B$, the tensor product $A \otimes B$ becomes a $\ka$-algebra with multiplication defined as follows:
\allowdisplaybreaks
\begin{align*}
    &m_{A \otimes B}: (A \otimes B) \otimes (A \otimes B) \to A\otimes B\\
    &m((a \otimes b)\otimes (a^{\prime} \otimes b^{\prime})) := a a^{\prime} \otimes b b^{\prime}
\end{align*}
and unit map:
\allowdisplaybreaks
\begin{align*}
    & u_{A \otimes B}: \ka \to A \otimes B\\
    &u_{A \otimes B}(1_{k}) = 1_{A} \otimes 1_{B}. 
\end{align*}
\end{definition}
We are now ready to define bialgebras.

\begin{definition}[Bialgebra] Let $A$ be both a $\ka$-algebra and a $\ka$-coalgebra. We say that $A$ is a bialgebra if for all $x,y \in A$:
\begin{align*}
    & \Delta(xy) = \Delta(x)\Delta(y), \\
\end{align*}
i.e.~the coproduct $\Delta:A \to A \otimes A$ is an algebra morphism, and 
\begin{align*}
    & \varepsilon(xy) = \varepsilon(x)\varepsilon(y),
\end{align*}
i.e.~$\varepsilon:A \to \ka$ is an algebra morphism.
\end{definition}

\begin{definition}[Bialgebra morphism] Let $A, B$ be $\ka$-bialgebras. Let $\varphi:A \to B$ be a $\ka$-linear map that is an algebra as well as a coalgebra homomorphism. Then we say that $\varphi$ is a bialgebra homomorphism.
\end{definition}

\begin{definition} Let $x$ be an element in the coalgebra $C$. We say that $x$ is group-like if $\Delta(x)= x \otimes x$ and $\varepsilon(x)=1$. 
\end{definition}

\begin{definition}
Let $A$ be a bialgebra, we say that $x \in A$ is primitive if $\Delta(x)= 1_{A}\otimes x + x \otimes 1_{A}$.
\end{definition}

\begin{remark}
    Note that zero is not group-like. In fact, zero is primitive, $\Delta(0) = 0 =  1_{A} \otimes 0 + 0 \otimes 1_{A}$, as the set of primitive elements forms a subspace.
\end{remark}

\begin{definition}[Ideal, Coideal] 
i) $J$ is a two-sided \textbf{ideal} of an algebra $A$, if it is a subspace of $A$ such that $m(J \otimes A) \subset A$ and $m(A \otimes J) \subset A$. The quotient $A/J$ is again an algebra.\\
ii) $J$ is a two-sided \textbf{coideal} of a coalgebra $C$, if it is a subspace of $C$ such that $\Delta(J) \subset J\otimes A + A \otimes J$ and $\varepsilon(J)=0$. The quotient $C/J$ is again a coalgebra.\\
iii) In case $J$ is both a two-sided ideal as well as a coideal of a bialgebra $A$, then the quotient $A/J$ is a bialgebra.
\end{definition}

\subsubsection{Hopf algebra}

A Hopf algebra is a bialgebra with a particular linear map called antipode. The latter is defined to be the inverse of $\id$ for the convolution product. Hence, to define a Hopf algebra we need to introduce the concept of convolution algebra.
Let $C$ be a coalgebra and $A$ an algebra. Consider the vector space $\text{Hom}(C,A)$ of linear maps. It becomes a unital algebra by defining the convolution product:
\begin{align*}
    f * g := m_{A} \circ (f \otimes g) \circ \Delta_{C}
\end{align*}
for $f,g \in \text{Hom}(C,A)$. The unit is given by $u_A \circ \varepsilon_C$. Indeed, using \eqref{coalg2}, we find for every $f \in \text{Hom}(C,A)$:
\allowdisplaybreaks
\begin{align*}
    \sum f({c_{(1)}})\varepsilon(c_{(2)}) = f(c) =  \sum\varepsilon(c_{(1)})f(c_{(2)}).
\end{align*}
In particular, we can endow $\text{End}(A)$ with a convolution product when $A$ is a bialgebra.

\begin{definition}[Hopf Algebra]
A \textbf{Hopf algebra} is a bialgebra $(\curlyH,m,u,\Delta,\epsilon)$ with an antipode, i.e., there exists a linear map $S:\curlyH\to\curlyH$ such that
\allowdisplaybreaks
\begin{align}
\label{eq.antipode}
    S * \id =  u \circ \epsilon = \id * S.
\end{align}
\end{definition}

The following statement tells us that if the antipode exists, it is unique. Moreover, it is an algebra anti-homomorphism.

\begin{theorem}
    Let $(\curlyH,m,u,\Delta,\epsilon)$ be a bialgebra.  If $S$ and $S^{\prime}$ are both antipodes, then $S=S^{\prime}$. We have $S(1_\curlyH)=1_\curlyH$. Moreover, the antipode is an anti-homomorphism.
    $$ S\circ m( a\otimes b) = m\circ( S(b) \otimes S(a)).$$
\end{theorem}

\begin{proof}
Uniqueness is guaranteed by the fact that inverse elements in a ring are unique. Let $S$ and $S^{\prime}$ be inverses of $\id$ under the convolution product then: $(S -S^{\prime}) * \id = u \circ \epsilon - u \circ \epsilon = 0$. Since $\Delta(1_\curlyH) = 1_\curlyH \otimes 1_\curlyH$,, $S(1_\curlyH)1_\curlyH=1_\curlyH$ means that $S(1_\curlyH)=1_\curlyH$. For the anti-homomorphism property check the proof of \cite[Proposition 1.4.10]{grinberg2020hopf}.
\end{proof}

It is known, as noted in \cite[p. 43]{cartier2021classical}, that given two Hopf Algebras, $\curlyH_{1}$ and $\curlyH_{2}$, any bialgebra homomorphism is automatically a Hopf algebra homomorphism.

\begin{proposition}
    \label{thm:homBiHopf}
    Suppose that $(\curlyH_1,m_1,\Delta_1,S_1)$ and $(\curlyH_2,m_2,\Delta_2,S_2)$ are Hopf algebras. For all bialgebra homomorphisms $\varphi~:~\curlyH_1~\to~\curlyH_2$ , it holds that
    \begin{align}
      \varphi \circ S_{1} &= S_{2} \circ \varphi.
    \end{align}
\end{proposition}

\begin{proof}
   Notice that both functions $\varphi \circ S_1$ and $S_2 \circ \varphi$ are convolutional inverses of $\phi$. We can first use the multiplicative property of $\varphi$, in order to see it.
   \begin{align*}
     m_2 \circ (\varphi \otimes \varphi \circ S_1) \Delta_1 &= 
       m_2 \circ (\varphi\otimes\varphi) \circ (\id_1 \otimes S) \circ \Delta_2 \\
                                                            &= 
       \varphi \circ m_1 \circ (\id_1 \otimes S_1 ) \circ \Delta_1 \\
                                                            &= \varphi \circ (u_1 \circ \varepsilon_1 ) \\
                                                            &= u_2 \circ \varepsilon_1
       \end{align*}
    This is the same function we obtain by using the comultiplicative property.
       \begin{align*}
       u_2 \circ \varepsilon_1 
       &= (u_2 \circ \varepsilon_2) \circ \varphi \\
       &= m_2 \circ (S_2 \otimes \id_2) \circ \Delta_2 \circ \varphi \\
       &= m_2 \circ (S_2\otimes \id_2) \circ (\varphi\otimes\varphi) \circ \Delta_1.
    \end{align*}
    The conclusion follows from the uniqueness of the inverse in associative unital algebras.
\end{proof}

In addition to the abstract definitions, there are some natural examples of what is a Hopf algebra. Perhaps the most important of them, according to~\cref{thm:cartier}, is the polynomial Hopf algebra. 

\begin{example}
    The \textbf{polynomial} Hopf algebra in one generator $x$ is defined in the space $$
    \textbf{k}[x] := \lrp{  \textbf{k}\lra{1,x,x^2,x^3,\ldots}, m, \Delta}
$$ 
where product and coproduct are    
\allowdisplaybreaks
 \begin{align*}
     m(x^{m} \otimes x^{n}):= x^{m+n}
     \quad \text{resp.} \quad
     \Delta(x^n) :=
    \sum_{k=0}^n \binom{n}{k} x^{n-k} \otimes x^k.
\end{align*}
The counit and antipode are given by
\begin{align*}
\varepsilon(a) = 
    \begin{cases}
1, \;\text{if}\; a =  1\\
0, \;\text{else}\\
\end{cases}
\quad \text{resp.} \quad
    S(x^{n}) = (-1)^{n}x^n.
\end{align*}

    Notice that if $e_a: \textbf{k}[x] \to \textbf{k}$ is the linear extension of the evaluation in $a \in \textbf{k}$: $x^n \mapsto a^n$, then
    \begin{align*}
      m \circ (e_a \otimes e_b) \circ \Delta &= e_{a+b}.
    \end{align*}
    This is due to the binomial theorem.
\end{example}

\subsubsection{Grading, filtration, and connectedness}
\label{ssec:gradedfiltered}

We give the definition of a graded vector space, from which one can define graded algebras, coalgebras, and bialgebras. We will see that in some cases the weaker notion of filtered algebras and coalgebras can be used when these objects are not graded.

\begin{definition}[Graded vector space]
A graded $\ka$-vector space $V$ is a vector space that admits a decomposition as direct sum $V:= \bigoplus_{n \ge 0} V_{n}$ of vector spaces $\{V_n\}_{n \ge 0}$. Note that for $ V \otimes V$, one considers the induced grading given by $V$
\begin{align*}
    (V \otimes V)_{n} := \bigoplus_{i+j = n} V_{i} \otimes V_{j}.
\end{align*}
\end{definition}

It is natural to consider graded linear maps.
\begin{definition}[Graded maps]
   Let $V,W$ be graded vector spaces. Then $\phi:V \to W$ is graded if $\phi(V_{n}) \subset W_{n}$, for all $n \ge 0$.
\end{definition}

Let $V$ be a graded vector space. In case $V_{0} \cong \ka$ we say that $V$ is connected. This is a technical definition: its meaning will become clearer when introducing filtered connected bialgebras. 

\begin{definition}[Graded algebras and coalgebras]
   A graded algebra and coalgebra, $A = \bigoplus_{n \ge 0} A_{n}$ respectively $C= \bigoplus_{n \ge 0} C_{n}$, are graded vector spaces such that for all $n,m \ge 0$
\begin{align*}
    A_{n}A_{m} \subset A_{n+m} 
    \quad \text{resp.} \quad
    \Delta(C_{n}) \subset \bigoplus_{i+j=n} C_{i} \otimes C_{j}.
\end{align*}
In other words, products and coproducts are graded maps. 
 \end{definition}

A bialgebra is graded if it is graded as an algebra as well as a coalgebra.

\begin{definition}[Filtered algebras and coalgebras] 
We consider now an algebra $A = \bigcup_{n \ge 0}A^{n}$ such that $A^{0} \subset A^{1} \subset \cdots \subset A^{n} \subset \cdots$. We say that $A$ is a filtered algebra if $A^{n}A^{m} \subset A^{n+m}.$ 

Similarly, a coalgebra $C = \bigcup_{n \ge 0}C^{n}$ is filtered if
\begin{align*}
   \Delta(C^{n}) \subset \sum_{i+j=n} C^{i} \otimes C^{j}.
\end{align*}
A bialgebra is filtered if is filtered as an algebra as well as a coalgebra. 
\end{definition}

Note that a graded bialgebra is always filtered since we can define the filtration in terms of
\begin{align*}
    A^{n}:= \bigoplus_{i=0}^{n}A_{i}.
\end{align*}

We can define filtered maps (this is a very weak notion) between filtered vector spaces $V,W$ as linear maps $\phi:V \to W$ such that $\phi(V^{n}) \subset W^{n}$, for all $n\ge 0$.

\subsubsection{Connected filtered Hopf algebra}
\label{ssec:confilHA}

Here we show that in the particular case where the bialgebra is connected and filtered, the antipode always exists. We start by introducing some notation. Let $A$ be an algebra, for $k \ge 0$ we define inductively the iterated product $m_{k}: A^{\otimes(k+1)} \to A$ 
\begin{align*}
    &m^{0}:= \id\\
    &m^{k} := m \circ (\id \otimes m^{k-1}),\quad \text{if}\;k \ge 1.
\end{align*}
Similarly, for a coalgebra $C$ we define the iterated coproduct $\Delta^{k}: C \to C^{\otimes(k+1)}$
\begin{align*}
    &\Delta^{0}:= \id\\
    &\Delta^{k} := (\id \otimes \Delta^{k-1}) \circ \Delta, \quad\text{if}\;k \ge 1.
\end{align*}
We can now define for any $f \in \text{Hom}(C,A)$ and $n \ge 1$, the $n$-th convolution power
\begin{align*}
    f^{*0} := u_{A} \circ \varepsilon_{C}
    \quad \text{and} \quad 
    f^{*n+1}:= m^{n} (\underbrace{f \otimes \cdots \otimes f}_{\text{n+1 times}}) \Delta^{n}.
\end{align*}

\begin{lemma}\cite[Section 4.2]{bib:manchon2008hopf}
\label{thm.coproduct_filtered_connected}
 Let $\curlyH$ be a connected filtered bialgebra, then for any $x \in \curlyH^{n}$ with  $n \ge 1$, we can write:
 \begin{align*}
     \Delta(x) = x \otimes 1_{\curlyH} + 1_{\curlyH} \otimes x + \widetilde{\Delta}(x),
 \end{align*}
 where the reduced coproduct (the non primitive part)
 \begin{align*}
    \widetilde{\Delta}(x) \in \bigoplus_{\substack{i+j=n \\ i \neq 0, j \neq 0}}\curlyH^{i} \otimes \curlyH^{j}.
 \end{align*}
 In particular, if $|x|\:=1$, where $|x|:= \textup{min}\{n \in \N: x \in \curlyH^{n}\}$ then
  \begin{align*}
    \Delta(x) = x \otimes 1_{\curlyH} + 1_{\curlyH} \otimes x .
 \end{align*}
Moreover,  $\widetilde{\Delta}$ is coassociative on $\textup{ker}(\varepsilon)$ and $\widetilde{\Delta}_{k} := (\id^{\otimes k - 1} \otimes \widetilde{\Delta}) \circ (\id^{\otimes k - 2} \otimes \widetilde{\Delta}) \circ \cdots \circ \widetilde{\Delta}$ sends $\curlyH^{n}$ to $(\curlyH^{n-k})^{\otimes k + 1}$.
\end{lemma}

Note that $1_{\curlyH}$ is group-like, i.e., $\Delta(1_{\curlyH})=1_{\curlyH} \otimes 1_{\curlyH}$. We will use the following variation to Sweedler's notation for the reduced coproduct $\widetilde{\Delta}$:
\begin{align*}
    \widetilde{\Delta}(x) := \sum_{(x)}x^{\prime} \otimes x^{\prime\prime}.
\end{align*}

From the next proposition, it follows that the antipode always exists in the connected filtered case.

\begin{proposition}\cite[Section 4.3]{bib:manchon2008hopf}
\label{thm.inv_conv_alg}
 Let $\curlyH$ be a connected filtered bialgebra. Suppose $A$ is an algebra. Consider the convolution algebra $(\mathrm{Hom}(\curlyH,A),*)$ with unit $u_{A}\circ\varepsilon_{\curlyH}$. The convolution product defines a group law on the set
 \begin{align*}
     G(A):= \{\varphi \in \mathrm{Hom}(\curlyH,A), \varphi(1_{\curlyH}) = 1_{A} \}.
 \end{align*}
\end{proposition}

\begin{proof}
 $G(A)$ is closed under the convolution product. Indeed
 \begin{align*}
     \varphi * \varphi^{\prime}(1_{\curlyH}) =m_{A} \circ (\varphi \otimes \varphi^{\prime}) \circ \Delta_{\curlyH}(1_{\curlyH}) = 1_{A}
 \end{align*}
 
 We need to show the existence of inverses. For this we consider
 \begin{align*}
      \varphi ^{*-1} := (u_{A}\circ\varepsilon_{\curlyH} - (u_{A}\circ\varepsilon_{\curlyH} - \varphi))^{*-1} = \sum_{n \ge 0}(u_{A}\circ\varepsilon_{\curlyH} -\varphi)^{*n}.
\end{align*}
The series always terminates. Indeed, first consider the elements of $\curlyH^{0}$. We have $(u_{A}\circ\varepsilon_{\curlyH})(1_{\curlyH})=1_{A}$ and $(u_{A}\circ\varepsilon_{\curlyH} - \varphi)^{*n}(1_{\curlyH})=0$ for $n \ge 1$, since $\Delta^{n}(1_{\curlyH}) = 1_{\curlyH}^{\otimes (n+1)} $. Due to the filtered bialgebra being connected, for all $x\in\curlyH^0$, it is true that $\varphi^{*k}(x)=0$ for $k>0$.

Now for $n\geq 0$ and  $x\in\curlyH^n$, suppose that  $(u_{A}\circ\varepsilon_{\curlyH} - \varphi)^{*k}(x)=0$ for $k> n$. If $y\in\curlyH^{n+1}$, then it holds true that
\begin{align*}
  \Delta(y) &= y\otimes 1_{\curlyH}+ 1_{\curlyH}\otimes y + \sum_{(y)} y^\prime \otimes y^{\prime\prime},
\end{align*}
where $y^\prime,y^{\prime\prime} \in \curlyH^{n}$ by \cref{thm.coproduct_filtered_connected}. Using coassociativity, one can observe that for $k>n$,
\begin{align*}
    (u_{A}\circ\varepsilon_{\curlyH} - \varphi)^{*k+1}(y)
    &= 
    m^{k} \circ (u_{A}\circ\varepsilon_{\curlyH} - \varphi)^{\otimes k+1} \Delta^{k}(y) \\
    &= 
    m^{k} \circ (\varphi^{\otimes k}\otimes \varphi)  \lrp{ \Delta^{k-1}(y) \otimes 1_{\curlyH}+ \Delta^{k-1}(1_{\curlyH}) \otimes y + \sum_{(y)} \Delta^{k-1}(y^\prime) \otimes y^{\prime\prime}  } \\
    &=  {\varphi}^{*k}(y){\varphi}(1_{\curlyH}) +
    \varphi^{*k}(1_{\curlyH})\varphi(y) + \sum_{(y)} {\varphi}^{*k}(y^\prime)\varphi(y^{\prime\prime}).
\end{align*}
Notice that, by induction hypothesis $y^\prime\in\curlyH^n$ and hence
\begin{align*}
  \varphi(1_\curlyH) = \varphi^{*k}(1_{\curlyH}) &=
    \varphi^{*k}(y^\prime) = 0.
\end{align*}
\end{proof}

In the special case in which we consider the convolution algebra on $\text{End}(\curlyH)$, recall that the antipode is the convolutional inverse of the identity:
\begin{align}
\label{eq:antipode_inverse}
      S = \sum_{n \ge 0}(u \circ \varepsilon -\id)^{*n}.
 \end{align}

\begin{remark}
As pointed out in \cite{bib:manchon2008hopf}, when having a connected filtered Hopf algebra, we can compute the antipode recursively. Recall that $S(1_{\curlyH})=1_{\curlyH}$. Let $x \in \curlyH^{n}$, $n \ge 1$ then
\begin{align*}
    m \circ ( S \otimes \id )\Delta(x) = u \circ \varepsilon (x) = 0
\end{align*}
and
\begin{align*}
    0=m \circ ( S \otimes \id )\Delta(x) 
    &= m \circ ( S \otimes \id )(1_{\curlyH} \otimes x + x \otimes 1_{\curlyH} + \widetilde{\Delta}(x))\\
    &=x + S(x) + m \circ ( S \otimes \id )\widetilde{\Delta}(x)\\
    &= x + S(x) + \sum_{(x)}S(x^{\prime})x^{\prime\prime}. 
\end{align*}
This yields the recursion
\begin{align}
\label{eq:antipode_recursion}
    S(x) = -x - \sum_{(x)}S(x^{\prime})x^{\prime\prime}, 
\end{align}
where in the sum on the righthand side, $S$ is evaluated on strictly smaller elements.
\end{remark}

\subsubsection{Dual vector spaces, adjoint maps}
\label{sec.DualBialgebras}

We introduce notation for the dual of a vector space and for the graded dual. The main reference for this section is \cite{grinberg2020hopf}. The vector spaces we will be working on are of \textbf{finite type}, i.e. they are graded $V = \bigoplus_{n \ge 0}V_{n}$ and each subspace $V_{n}$ has finite dimension.

\begin{definition}[Graded dual]
\label{def.gradedDual}
Let $V$ be a $\mathbf{k}$-vector space. The dual of $V$ is the $\mathbf{k}$-vector space $V^{*}:=\text{Hom}(V,\mathbf{k})$. 
When $V=\bigoplus_{n \ge 0}V_{n}$ is a graded vector space, the \textbf{graded dual} is the subspace $V^{\circ}:= \bigoplus_{n \ge 0}V^{*}_{n}\subset V^*$.
\end{definition}

Note that $V^*=\prod_{n\geq 0} V_n^*$ is a larger space than $V^{\circ}$. Functions that are non-zero on infinitely many $V_{n}$ are also part of $V^{*}$. For instance, let $\zeta \in V^*$ be the functional sending each basis element of $V$ to $1$. Then, $\zeta\in V^*$, but $\zeta\not\in V^\circ$. 

For a vector space $V$, we denote the \textbf{bilinear pairing} as follows:
\begin{align*}
    \langle \cdot, \cdot \rangle_{V}:& V^{*} \otimes V \to \mathbf{k}\\
    & f\otimes v \mapsto f(v).
\end{align*}
Consider now a $\mathbf{k}$-linear map $\phi:V \to W$. It induces an \textbf{adjoint map} $\phi_{adj}:W^{*} \to V^{*}$ defined as follows:
\begin{align*}
    \langle g, \phi(v) \rangle_{W} = \langle \phi_{adj}(g), v \rangle_{V}
\end{align*}
In case $V,W$ are finite-dimensional, then $\text{Hom}(V,W)$ and $\text{Hom}(W^*,V^*)$ are isomorphic as vector spaces via the adjoint map. This connection is carried on by the matrix transpose on this finite case. However, there is not an equivalent notion of matrix transpose for the infinite-dimensional case~\cite{simon1997there}.

Fortunately, it is known \cite{grinberg2020hopf,bib:hazewinkel2010algebras,cartier2007primer} that this notion is well defined for the graded dual of finite type vector spaces and homomorphisms respecting that grading.

For a graded coalgebra $(C,\Delta_{C})$, the \textbf{dual product} on $C^{\circ} \otimes C^{\circ}$ is uniquely defined as follows.
\begin{align}
\label{eq:dual_coalg}
    \langle m_{C^{\circ} \otimes C^{\circ}}(a \otimes b) , c \rangle_{C} :=    \langle a \otimes b , \Delta_{C}(c) \rangle_{C \otimes C}.
\end{align}
The linear transformation $m_{C^{\circ} \otimes C^{\circ}}$ is then the adjoint map of $\Delta_{C}$ accordingly restricted to the graded dual. Note that this is only a \underbar{subalgebra} of the convolution algebra on $C^{*}$.

Moreover, for a graded algebra $(A,m_A)$ on a vector space of finite type, one can uniquely define its \textbf{dual coproduct}.
\begin{align}
\label{eq:dual_alg}
    \langle \Delta_{A^{\circ}}(c) , a \otimes b \rangle_{A} :=    \langle  c, m_{A}(a \otimes b) \rangle_{A \otimes A}.
\end{align}

These definitions allow us to dualize immediately graded bialgebras. 
In the following, we will consider bialgebras that are not graded and give sufficient conditions so that \eqref{eq:dual_coalg} and \eqref{eq:dual_alg} can still be applied. 

\begin{lemma}
\label{lemma:dualnongraded}
Suppose that $(C,\Delta)$ is a coalgebra where $C=\bigoplus_{n \ge 0}C_{n}$ is of finite type. If
\begin{align*}
    \Delta(C_{n}) \subset \bigoplus_{i+j \ge n} C_{i} \otimes C_{j},
\end{align*}
then its dual product can be defined as in~\cref{eq:dual_coalg}.
Assume that $(A,m_A)$ is an algebra where $A=\bigoplus_n A_n$ is of finite type. If
\begin{align*}
    m(A_{n} \otimes A_{k}) \subset \bigoplus_{\ell \ge \textup{max}(n,k)}A_{\ell},
\end{align*}
then we can define its dual coalgebra $\Delta_{A^{\circ}}(c)$as in \eqref{eq:dual_alg}.
\end{lemma}

\begin{proof}
Fix $a\otimes b \in C_{i}^\circ \otimes C_{j}^\circ$. Then $\langle  a \otimes b, \Delta(c) \rangle_{C \otimes C}$ can be non-zero only for finitely many basis elements of $C$. Indeed, for $c\in C_{\ell}$ where $\ell > i+j $ it holds that $\Delta(c) \in \bigoplus_{x+y \geq l} C_{x}\otimes C_{y}$. Since $i+j < \ell $ and $x+y \geq \ell $, 
\begin{align*}
\langle  a \otimes b, \Delta(c) \rangle_{C \otimes C}=0
 \end{align*}
Fix $c \in A^{\circ}_{t}$. $\langle  c, m(a \otimes b) \rangle_{A \otimes A}$ can be non-zero only for a finite number of basis elements of $A \otimes A$. Indeed: \begin{align*}
      a \otimes b \in \bigoplus_{i> t\;\text{or}\;j > t
      } A_{i} \otimes A_{j} \implies \langle  c, m(a \otimes b) \rangle_{A \otimes A}=0.
  \end{align*}
\end{proof}

\subsubsection{Characters}
 
 Assume that $\curlyH$ is a Hopf Algebra. The characters form a group.
 
\begin{proposition}
    \label{thm.convolutionOfCharacters}
   If $(A,m_{A},u_A)$ is a commutative unital associative algebra. If $f,g:(\curlyH,m_{\curlyH}) \to (A,m_A)$ are unital algebra homomorphisms, then $f*(g \circ S)$ is also a unital algebra morphism.
   
\end{proposition}

\begin{proof}
    Recall that $h\circ m_{\mathcal{H}} = m_{A} \circ (h\otimes h)$ if $h$ is an algebra morphisms.
    \begin{align*}
        m_{A} \circ (f \otimes g) \circ \Delta_{\mathcal{H}} \circ m_{\mathcal{H}}
        &= 
        m_{A} \circ (f \otimes g) \circ
        (m_{\mathcal{H}} \otimes m_{\mathcal{H}}) \circ \tau_{(2,3)} \circ 
        (\Delta_{\mathcal{H}} \otimes \Delta_{\mathcal{H}}) \\
        &= 
        m_{A} \circ
        (m_{A} \otimes m_{A}) \circ (f\otimes f \otimes g \otimes g) \circ \tau_{(2,3)} \circ 
        (\Delta_{\mathcal{H}} \otimes \Delta_{\mathcal{H}}) \\
        &= 
        m_{A} \circ
        (m_{A} \otimes m_{A}) \circ\tau_{(2,3)} \circ (f\otimes g \otimes f \otimes g) \circ 
        (\Delta_{\mathcal{H}} \otimes \Delta_{\mathcal{H}}).
    \end{align*}
    Here, $\sigma_{(2,3)}:= \id \otimes \tau \otimes \id$. Thanks to commutativity of $A$, 
    we have that 
    $$
        m_{A} \circ
        (m_{A} \otimes m_{A}) \circ\tau_{(2,3)} = m_{A} \circ
        (m_{A} \otimes m_{A}).
    $$
    Therefore,
    \begin{align*}
        m_{A} \circ (f \otimes g) \circ \Delta_{\mathcal{H}} \circ m_{\mathcal{H}}
        &= 
        m_{A} \circ
        (m_{A}\circ (f\otimes g) \circ  \Delta_{\mathcal{H}} )  \otimes 
        (m_{A}\circ (f\otimes g) \circ  \Delta_{\mathcal{H}} ).
    \end{align*}
    This shows that the set of algebra morphisms is closed under the convolution product. Now, regarding inverses, we see that
    \begin{align*}
      m_{A} \circ (g \otimes (g\circ S) ) \circ \Delta_{\mathcal{H}} &= 
        m_{A} \circ (g\otimes g) \circ (\id \otimes S) \circ \Delta_{\mathcal{H}} \\
                                                                     &= 
        g \circ m_{\mathcal{H}} 
        \circ (\id \otimes S) \circ \Delta_{\mathcal{H}} \\
                                                                     &= g \circ u_{\mathcal{H}}\circ \varepsilon_{\mathcal{H}} \\
                                                                     &= 
        u_{A}\circ\varepsilon_{\mathcal{H}}.
    \end{align*}
\end{proof}

\subsection{Coassociativity, cocommutativity, associativity and commutativity}
\label{ssec:Associativity.Proof}

\begin{theorem}[as in \cref{thm:coassociative}]
  \label{thm:coassociative_app}
  The coproducts in \cref{sssec:symgraph}, \cref{sssec:coalgebra_sub}, and  \cref{sssec:coalgebra induced} are counital, coassociative and cocommutative. Dually, the corresponding products are unital, associative and commutative.
\end{theorem}

\begin{proof}

It is immediate to see that the coproducts are counital and the products are unital. Cocommutativity is also immediate. Associativity of $\cdot_{\hom}$ and $\cdot_{\hom^{\prime}}$ is immediate. Coassociativity of $\Delta_{\hom}$ and $\Delta_{\hom^{\prime}}$ then follows since $\cdot_{\hom}$ and $\cdot_{\hom^{\prime}}$ are graded and the underlying vector space is of finite type.

We now show the coassociativity of $\Delta_{\mono^{\prime}}$. Define for $\tau \in \FinGraph$, the following sets,
\begin{align*}
    S&:=\{((A_{1},A_{2}),(B_{1},B_{2}),(C_{1},C_{2}))|A_{1} \cup B_{1} \cup C_{1} = V(\tau), A_{2} \cup B_{2} \cup C_{2} = E(\tau),\\
    &\qquad\qquad\qquad\qquad \bigcup A_{2} \subseteq A_{1}, \bigcup B_{2} \subseteq B_{1}, \bigcup C_{2} \subseteq C_{1}\}\\
     T&:=\{((A_{1},A_{2}),(B_{1},B_{2})|A_{1} \cup B_{1} = V(\tau), A_{2} \cup B_{2} = E(\tau), \bigcup A_{2} \subseteq A_{1}, \bigcup B_{2} \subseteq B_{1}\}
\end{align*}
and the maps,
\begin{align*}
 \phi&: S \to T,\;((A_{1},A_{2}),(B_{1},B_{2}),(C_{1},C_{2})) \mapsto ((A_{1}\cup B_{1},A_{2}\cup B_{2}),(C_{1},C_{2}))\\
  \psi&: S \to T,\;((A_{1},A_{2}),(B_{1},B_{2}),(C_{1},C_{2})) \mapsto ((A_{1},A_{2}),(B_{1}\cup C_{1}, B_{2} \cup C_{2})).
\end{align*}
$\phi$ and $\psi$ are clearly surjective. We denote with $s:=((A_{1},A_{2}),(B_{1},B_{2}),(C_{1},C_{2}))$ an element of $S$ and also write $f(s):=\tau\evaluatedAt{A_{1},A_{2}}\otimes\tau\evaluatedAt{B_{1},B_{2}}\otimes\tau\evaluatedAt{C_{1},C_{2}}$.
We then have
\begin{align*}
    \sum_{t \in T}\sum_{s \in \psi^{-1}(t)}f(s) = \sum_{s \in S} f(s) =\sum_{t \in T}\sum_{s \in \phi^{-1}(t)}f(s)
    \end{align*}
    and 
    \begin{align*}
        \sum_{t \in T}\sum_{s \in \psi^{-1}(t)}f(s) &=\sum_{\substack{X_{1}\cup Y_{1}=V(\tau)\\X_{2}\cup Y_{2}=E(\tau)\\ \cup X_{2}\subseteq X_{1},  \cup Y_{2}\subseteq Y_{1}}}\sum_{\substack{A_{1}\cup B_{1}=X_{1}\\A_{2}\cup B_{2}=X_{2}\\\cup A_{2}\subseteq A_{1},  \cup B_{2}\subseteq B_{1}}}\tau\evaluatedAt{A_{1},A_{2}}\otimes\tau\evaluatedAt{B_{1},B_{2}}\otimes\tau\evaluatedAt{Y_{1},Y_{2}}\\
        \sum_{t \in T}\sum_{s \in \phi^{-1}(t)}f(s) &=\sum_{\substack{X_{1}\cup Y_{1}=V(\tau)\\X_{2}\cup Y_{2}=E(\tau)\\ \cup X_{2}\subseteq X_{1},  \cup Y_{2}\subseteq Y_{1}}}\sum_{\substack{A_{1}\cup B_{1}=Y_{1}\\A_{2}\cup B_{2}=Y_{2}\\\cup A_{2}\subseteq A_{1},  \cup B_{2}\subseteq B_{1}}}\tau\evaluatedAt{X_{1},X_{2}}\otimes\tau\evaluatedAt{A_{1},A_{2}}\otimes\tau\evaluatedAt{B_{1},B_{2}}.
    \end{align*}
This shows that $\forall \tau \in \FinGraph$:
\[(\Delta_{\mono^{\prime}} \otimes \id) \circ \Delta_{\mono^{\prime}}(\tau) =  (\id \otimes \Delta_{\mono^{\prime}}) \circ \Delta_{\mono^{\prime}}(\tau).\] The proof for $\Delta_{\regmono^{\prime}}$ is very similar. It then follows, using \cref{lemma:dualnongraded} that $\cdot_{\mono^{\prime}}$ and $\cdot_{\regmono^{\prime}}$ are associative products. We can then easily see, thanks to \cref{remark:defacto3}, that $\cdot_{\mono}$ and $\cdot_{\regmono}$ are also associative and using once again \cref{lemma:dualnongraded} we obtain the coassociativity of $\Delta_{\mono}$ and $\Delta_{\regmono}$. 
\end{proof}

\subsection{Categorical background}
\label{ssec:category}

We recall basic notions from category theory.

\subsubsection{Factorization of morphisms}

For the moment we work in an
arbitrary category $\Cat$.
Let $A,B$ be objects in $\Cat$.
A morphism $f \in \mor(A, B)$ is
\begin{itemize}

  \item
    a \textbf{monomorphism}
    if for all objects $C$ and all $g,h \in \mor(C, A)$
    \begin{align*}
      f \circ g  = f \circ h \Rightarrow g = h.
    \end{align*}

  \item
    a \textbf{regular monomorphism}
    if it is the equalizer of some parallel pair of morphisms,
    i.e. if there is a limit diagram of the form
    \begin{align*}
      A \arr{f} B \rightrightarrows D.
    \end{align*}
    It is well-known that a regular monomorphism is a monomorphism.

  \item 
    an \textbf{epimorphism},
    if for all objects $C$ and for all $g,h \in \mor(B, C)$
    \begin{align*}
      g \circ f = h \circ f \Rightarrow g = h.
    \end{align*}

  \item
    a \textbf{regular epimorphism}
    if it is the coequalizer of some parallel pair of morphisms,
    i.e. if there is a colimit diagram of the form
    \begin{align*}
      D \rightrightarrows A \arr{f} B.
    \end{align*}
    It is well-known that a regular epimorphism is an epimorphism.

  \item
    a \textbf{section}
    if $f$ as a left inverse, i.e.
    if there is an $f_{L}\in\mor(B, A)$ such that
    \begin{align*}
     f_{L} \circ f = \id_{A}.
    \end{align*}

     \item
    a \textbf{retraction}
    if $f$ as a right inverse, i.e. if there is an $f_{R}\in\mor(B,A)$ such that
    \begin{align*}
     f \circ f_{R} = \id_{A}.
    \end{align*}

\item  an \textbf{isomorphism} if it has a two sided inverse: there is $f^{-1} \in \mor(B,A)$
    with $f^{-1} \circ f = \id_A, f \circ f^{-1} = \id_{B}$.

 \item
    an \textbf{extremal monomorphism}
    if whenever $f = c \circ e$ with $e$ an epimorphism, then $e$ is an isomorphism. 

    \item
    an \textbf{extremal epimorphism}
    if whenever $f = m \circ c$ with $m$ a monomorphism, then $m$ is an isomorphism.
    
\end{itemize}

Let $E,M$ be classes of morphisms in $\Cat$.
We say that $\Cat$ has \textbf{$(E,M)$-factorization}
if every morphism $f$
in $\Cat$ can be written as $f=e \circ m$
for some $e \in E, m \in M$.

From \cite{addmek1990abstract} , we recall the following definition.
\begin{definition}
\label{def:EMfact}
A category $\Cat$ is \textbf{$(E,M)$-structured} if
\begin{enumerate}
 \item $E$ and $M$ are both closed under composition of isomorphisms,
 \item  $\Cat$ has \textbf{$(E,M)$-factorization},
 \item $\Cat$ has the \textbf{unique} \textbf{$(E,M)$-diagonalization property}, i. e., for each commutative diagram:
 \begin{center}
\begin{tikzcd}
A \arrow[r, "e"] \arrow[d,"f"']
& B \arrow[d, "g"] \\
C \arrow[r, "m"']
&  D
\end{tikzcd}
\end{center}
with $e \in E$ and $m \in M$ there exists a unique morphism $d$ such that the following diagram commutes:
\begin{center}
\begin{tikzcd}
A \arrow[r, "e"] \arrow[d,"f"']
& B \arrow[d, "g"] \arrow[dl,"d"'] \\
C \arrow[r, "m"']
&  D
\end{tikzcd}
\end{center}
 \end{enumerate}
 \end{definition}

The proof of the following proposition can be found in \cite[Proposition 14.4]{addmek1990abstract}.
 
\begin{proposition}
\label{proposition:unique_fact}
If $\mathcal{C}$ is $(E,M)$-structured, then $(E,M)$-factorizations are essentially unique, i.e.,
\begin{enumerate}

\item if $A \arr{e_{i}} C_{i}\arr{m_{i}} B$ are $(E,M)$-factorizations of $A \arr{f}B$ for $i = 1, 2$, then there
exists a (unique) isomorphism $h$, such that the following diagram commutes:
\begin{center}
\begin{tikzcd}
A \arrow[r, "e_{1}"] \arrow[d,"e_{2}"']
& C_{1} \arrow[d, "m_{1}"] \arrow[dl,"h"'] \\
C_{2} \arrow[r, "m_{2}"']
&  B
\end{tikzcd}
\end{center}

\item if $A \arr{f}B = A \arr{e} C\arr{m} B$ is a factorization and $C \arr{h} D$
is an isomorphism, then $A \arr{f}B = A \arr{h \circ e} D\arr{m \circ h^{-1}} B$ is also an $(E,M)$-factorization of $f$.
\end{enumerate}

\end{proposition}

We also recall part of \cite[Proposition 14.14]{addmek1990abstract}.
\begin{proposition}
\label{prop:rmer}
If a category $\Cat$ admits $(\RegEpi,\Mono)$ ($(\Epi,\RegMono)$) factorization, then the category is $(\RegEpi,\Mono)$($(\Epi,\RegMono)$)-structured. This factorization is unique in the sense of \cref{proposition:unique_fact}.
\end{proposition}
\begin{proof}
As an example, $\RegMono$ is closed under precomposition with isomorphisms.
Indeed, if $f \in \mor(A,B)$ is a regular monomorphism,
it is the equalizer of some parallel pair $\ell,m \in \mor(B,C)$.
If $\phi: \mor(X,A)$ is an isomorphism,
then $f \circ \phi$ is the equalizer of the same parallel pair $\ell,m$.

\end{proof}

\subsubsection{Category of finite simple graph}
\label{sssec: categoryfinitesimplegraphs}

The category of  finite simple graphs consists of
finite simple graphs $\sigma = (V(\sigma),E(\sigma))$ as objects
and graph homomorphisms as morphisms,
i.e. $f: \sigma \to \tau$ is a morphism if $\{v_{i},v_{j}\} \in E(\sigma)$ implies
$\{f(v_{i}),f(v_{j})\} \in E(\tau)$. We now give a concrete description of some of thes above classes of morphisms.

\begin{proposition}
  In the category of finite simple graphs
  $f \in \mor(\sigma,\tau)$ is
  \begin{enumerate}
    \item a monomorphism if and only if the function $f: V(\sigma)\to V(\tau)$ is injective (\textit{injective on vertices}\footnote{.. and then automatically on edges ..}).
    
    \item an epimorphism if and only if the function $f: V(\sigma)\to V(\tau)$ is surjective (\textit{surjective on vertices}).
    
  \end{enumerate}
\end{proposition}
\begin{proof}
\begin{enumerate}

\item 
  Let $f \in \mor(\sigma,\tau)$, $f: V(\sigma) \to V(\tau)$ be injective on vertices. Now take any $g,g^{\prime} \in \mor(\tau,\gamma)$, $g,g^{\prime}: V(\tau) \to V(\gamma)$. If $(f \circ g) (v) = (f \circ g^{\prime}) (v)$, then  it follows that $g(v) = g^{\prime}(v)$ for all $v \in V(\sigma)$, i.e. $f$ is a monomorphism. Now let $f \in \mor(\sigma,\tau)$ be a monomorphism. Assume that $f$ is not injective. Then there exists two different vertices $v,v^{\prime} \in V(\sigma)$ such that $f(v)=f(v^{\prime})$. Take the graph with a unique vertex, $H:=(*,\emptyset)$ and consider $g,g^{\prime} \in \mor(H,\sigma)$ defined as $g(*)=v$ and $g^{\prime}(*)=v^{\prime}$. Then we have $f \circ g = f \circ g^{\prime}$ and $g \neq g^{\prime}$, i.e. $f$ is not a monomorphism, which is a contradiction.
   \item It is easy to see that if $f \in \mor(\sigma, \tau)$, $f: V(\sigma) \to V(\tau)$ is surjective on vertices, then it is an epimorphism. Now assume that $f$ is an epimorphism but is not surjective. Then, define the graph $\tau^{\prime}$ as follows
  \begin{align*}
    V(\tau^{\prime}) &:= \im f\ \disj \ ( V(\tau) \setminus \im f)\ \disj ( ( V(\tau) \setminus \im f ) \times \{ * \}) \\
    E(\tau^{\prime}) &:=
    \{ \{x,y\}
      \mid
      \{x,y\} \in E(\tau)
      \text{ \underbar{or} }
      x = (x',*), y = (y',*) \text{ with } \{x',y'\} \in E(\tau) \\
      &\qquad \qquad \qquad
      \text{ \underbar{or} }
      x \in V(\tau), y = (y',*) \text{ with } \{x,y'\} \in E(\tau)\}.
  \end{align*}
  Define $g,g^{\prime}: \tau \to \tau^{\prime}$ as $g(v) = v, \forall v \in V(\tau)$
  and
  \begin{align*}
    g'(v) :=
    \begin{cases}
      v &\qquad v \in \im f \\
      (v,*) &\qquad a \in V(\tau)\setminus \im f
    \end{cases}.
  \end{align*} then it follows that $ g \circ f = g^{\prime} \circ f$ and $g \neq g^{\prime}$, i.e. $f$ is not an epimorphism, which is a contradiction.
  \end{enumerate}
\end{proof}

We now show how equalizers and coequalizers behave in this category.

\begin{lemma}
\label{lemma:graph_equalizers}
  In the category of simple graphs, equalizers always exist.
\end{lemma}

\begin{proof}
Let $f, f^{\prime} \in \mor(\sigma,\tau)$. We define $S:=\{a \in V(\sigma)\;|\;f(a)=f^{\prime}(a)\}$ and consider the induced subgraph $\sigma_{S}$ and the inclusion map $\iota: \sigma_{S} \to \sigma$. Clearly $\iota \circ f = \iota \circ f^{\prime}$ and $\iota \in \mor(\sigma_{S},\sigma)$. Let $(\sigma^{\prime},g)$, with $g:\sigma^{\prime} \to \sigma$, be a cone: i. e. $f \circ g = f^{\prime} \circ g$. Clearly, $g(V(\sigma^{\prime})) \subset S$. Define the co-restriction of $g$ on $S \subset V(\sigma)$, $g_{S}:V(\sigma^{\prime}) \to S$, $g_{S}(a):=g(a), \forall a \in V(\sigma^{\prime})$. Clearly $g_{S} \in \mor(\sigma^{\prime},\sigma_{S})$ and $g = \iota \circ g_{S}$. Uniqueness of $g_{S}$ holds because $\iota$ is a monomorphism. To summarize, we have shown that the following diagram commutes:
\begin{center}
\begin{tikzcd}
\sigma_{s} \arrow[r,"\iota",hook]
&\sigma \arrow[r, "f", shift left=1.5ex] \arrow[r,"f^{\prime}"']
& \tau \\
\sigma^{\prime} \arrow[u,"g_{S}",dashed]  \arrow[ru,"g"'] &&. 
\end{tikzcd}
\end{center}
\end{proof}
On the contrary, 
\begin{example}
\label{example:graph_coequalizers}
  In the category of simple graphs, coequalizers do not always exist.
\end{example}
\begin{proof}
Let $\sigma=(\{1,2,3\},\{\{1,2\},\{1,3\}\})$ and $\tau(\{1,2,3\},\{\{1,2\},\{1,3\},\{2,3\})$. Define the two graph homomorphisms $f,f^{\prime}:\sigma \to \tau$ as $f(1)=1,f(2)=2,f(3)=3$, and $f^{\prime}(1)=1,f^{\prime}(2)=3,f^{\prime}(3)=2$. Clearly there does not exist a graph $\gamma$ and a graph homorphisms $g$ such that \begin{align*}
\sigma
\mathrel{\mathop{\rightrightarrows}^{f}_{f^{\prime}}} \tau \arr{g} \gamma.\end{align*}
\end{proof}
We now characterize the existence of coequalizers. 
\begin{lemma}
\label{lemma:graph_coequalizers_exist}
Consider \begin{align*}
\sigma
\mathrel{\mathop{\rightrightarrows}^{f}_{f^{\prime}}} \tau \end{align*}
any parallel pair in the category of finite simple graphs. Consider the partition on the vertex set of $\tau$, $P_{V(\tau)}$, induced by the smallest equivalence relation containing $\displaystyle \bigcup_{v \in V(\sigma)} \{(f(v),f^{\prime}(v))\}$.
Consider the graph, possibly with loops,
$$\tau^{\prime}:=(P_{V(\tau)},\{\{p,q\} \in [V]_{\textup{\textbf{multi}}}^{2}|\;\exists x \in p, \exists y \in q\;:\{x,y\} \in E(\tau)\}),$$ where $[V]_{\textup{\textbf{multi}}}^{2}$ denotes the set of multisets of cardinality $2$ with elements in $P_{V(\tau)}$. A coequalizer exists if and only if $\tau^{\prime}$ does not contain loops. The coequalizer is given by $\tau^{\prime}$ together with the map $g:\tau \to \tau^{\prime}$ that takes each vertex to the respective block. This map is surjective and $g(E(\tau)) = E(\tau^{\prime})$.
\end{lemma}
\begin{proof}
When $\tau^{\prime}$ has a loop, we have that a ``proper'' multiset $\{p,p\} \in [V]_{\textbf{multi}}^{2}$, i. e. $\exists\{ x,y\} \subseteq p \subseteq V(\tau)$ such that $\{x,y\} \in E(\tau)$. Now Let $f^{1}:=f$ and $f^{-1}:=f^{\prime}$. If there is a loop then $\exists v_{1},...,v_{n} \in V(\sigma)$ and $\exists e_{1}, ..., e_{n} \in \{-1,1\}$ such that $ x = f^{e_{1}}(v_{1})\sim f^{-e_{1}}(v_{1}) = f^{e_{2}}(v_{2}) \sim f^{-e_{2}}(v_{2}), \ldots, f^{e_{n-1}}(v_{n-1}) \sim f^{-e_{n-1}}(v_{n-1}) = f^{e_{n}}(v_{n}) \sim f^{-e_{n}}(v_{n}) = y$. Since $g\circ f^{e_{i}}(v_{i}) = g\circ f^{-e_{i}}(v_{i})$, for all $i=1,...,n$, then  $g(x)=g(y)$.
Hence, since $\{x,y\} \in E(\tau)$,
$g$ is \emph{not} a graph homomorphism
(i.e. in the category of simple, undirected graphs).

For the other direction, assume now that $\tau^{\prime}$ does not have a loop. We show that for all graphs $\gamma$ and graph homomorphism $h$ from $\tau$ to $\gamma$ such that $h \circ f = h \circ f^{\prime}$, the following diagram commutes:
\begin{center}
\begin{tikzcd}
\sigma\arrow[r, "f", shift left=1.5ex] \arrow[r,"f^{\prime}"']
&\tau \arrow[r,"g"]
\arrow[rd,"h"]
& \tau^{\prime} \arrow[d,"u",dashed] \\
&& \gamma
\end{tikzcd}
\end{center}
where $u$ is the unique map such that $h = u \circ g$. Recall that $g:\tau \to \tau^{\prime}$ is the map taking each vertex to the corresponding block. The equality $g \circ f = g \circ f^{\prime}$ clearly holds. Moreover, $g$ is a graph homomorphism, since 
\begin{align*}
    \{x,y\} \in E(\tau) \implies \{g(x),g(y)\} \in E(\tau^{\prime})
\end{align*}
by definition of $E(\tau^{\prime})$. We then argue as follows: in order for $h$ and $\gamma$ to be a cocone it must be that the partition induced on the vertex set of $\tau$ by the preimages of $h$ is at least as coarse as the one induced by the preimages of $g$. Then we define $u:\tau^{\prime} \to \gamma$ as:
\begin{align*}
u(c):=h\circ g^{-1}(c)
\end{align*}
which is well defined because $\forall c \in V(\tau^{\prime})$, $\exists v \in V(\gamma)$ such that $g^{-1}(c) \subseteq h^{-1}(v)$. It is also a graph homomorphism since
\begin{align*}
    \{p,q\} \in E(\tau^{\prime}) \implies \{u(p),u(q)\} \in E(\gamma)
\end{align*}
which holds because
\begin{align*}
    \{p,q\} \in E(\tau^{\prime}) \implies \exists x \in p, \exists y \in q\;:\{x,y\} \in E(\tau) \implies \{h(x) ,h(y)\} \in E(\gamma).
\end{align*}
and since $h(x) = u(p)$ and $h(y) = u(q)$
\begin{align*}
    \{h(x) ,h(y)\} \in E(\gamma) \implies \{u(p) ,u(q)\} \in E(\gamma).
\end{align*}
Finally $g: \tau \to \tau^{\prime}$ is obviously surjective and for each $\{p,q\}\in E(\tau^{\prime})$, recall that $\exists x \in p, \exists y \in q\;:\{x,y\} \in E(\tau)$ and $p = g(x)$ and $q = g(y)$. Uniqueness of $u$ follows from $g$ being an epimorphism.

\end{proof}
We now give a concrete description of regular monomorphisms and regular epimorphisms in the category of finite simple graphs.

\begin{proposition}
  In the category of finite simple graphs
  $f \in \mor(\sigma,\tau)$ is
  \begin{enumerate}
  
\item a regular monomorphism if and only if the function $f: V(\sigma)\to V(\tau)$ is injective and
      \begin{align*}
        \{i,j\} \in E(\sigma) \Leftrightarrow \{f(i),f(j)\} \in E(\tau),
      \end{align*}
      i.e. if and only if the 'corestricted' map $f: \sigma \to \tau_{f(V(\sigma))}$
      is an isomorphism.
\item a regular epimorphism if and only if the function $f: V(\sigma)\to V(\tau)$ is surjective
      \emph{and} the induced function $f: E(\sigma) \to E(\tau)$ is surjective (\textit{surjective on vertices and edges}).

      \end{enumerate}
\end{proposition}

\begin{remark}
  It is not clear to us whether there exists a useful alternative
  characterization of sections in the category of simple graphs.

  A section is automatically a regular monomorphism, but the converse is not true in general.
  For example, in the category of simple graphs, the
  map taking an edge to one of the edges of a triangle
  is a regular monomorphism but not a split monomorphism.
\end{remark}
\begin{proof}
We provide hands-on proofs.
For high-level proofs of these statements, see \cite{nlabSimpleGraphs}.
  
  \begin{enumerate}
      \item
   Let $f \in \mor(\sigma,\tau)$ be an equalizer of a parallel pair of morphisms $g,g^{\prime} \in \mor(\tau,\gamma)$. Define $S=\{ a \in V(\tau)|g(a)=g^{\prime}(a)\}$.
   Using \cref{lemma:graph_equalizers}, we know that $\tau_{S}$ together with the inclusion $\iota: \tau_{S} \to \tau$
   is also an equalizer. Since two equalizers must be isomorphic, we have that $f = \iota \circ h $, where $h \in \Iso(\sigma,\tau_S)$.
  We now show that $f(V(\sigma)) = S$. If $ v \in S$, then $v \in f(V(\sigma))$, since $\iota(v) = v = f(h^{-1}(v))$. If $v^{\prime} \in f(V(\sigma))$, then $v^{\prime} \in S$, since $\exists v\in V(\sigma)$ such that $v^{\prime} = f(v) = \iota \circ h(v)$. Therefore: $\sigma \cong \tau_{S} = \tau_{f(V(\sigma))}$.
  
  For the other direction, assume that $f: \sigma \to \tau$ is injective such that moreover $f: \sigma \to \tau_{\im f}$ is an isomorphism. We can also assume $f$ not to be surjective, otherwise we would have an isomorphism. It is easy to see that any $f \in \Iso(\sigma,\tau)$ is an equalizer simply by taking the parallel arrows to be both equal to $\id_{\tau}$. Define the graph $\tau^{\prime}$ as follows
  \begin{align*}
    V(\tau^{\prime}) &:= \im f\ \disj \ ( V(\tau) \setminus \im f)\ \disj ( ( V(\tau) \setminus \im f ) \times \{ * \}) \\
    E(\tau^{\prime}) &:=
    \{ \{x,y\}
      \mid
      \{x,y\} \in E(\tau)
      \text{ \underbar{or} }
      x = (x',*), y = (y',*) \text{ with } \{x',y'\} \in E(\tau) \\
      &\qquad \qquad \qquad
      \text{ \underbar{or} }
      x \in V(\tau), y = (y',*) \text{ with } \{x,y'\} \in E(\tau)\}.
  \end{align*}
  Define $g,g^{\prime}: \tau \to \tau^{\prime}$ as $g(v) = v, \forall v \in V(\tau)$
  and
  \begin{align*}
    g'(v) :=
    \begin{cases}
      v &\qquad v \in \im f \\
      (v,*) &\qquad a \in V(\tau)\setminus \im f
    \end{cases}.
    \end{align*}
 From the proof of \cref{lemma:graph_equalizers} we know that $\tau_{S}$ is an equalizer with the inclusion map and $S=\{a \in V(\tau)|g(a)=g(a^{\prime})\}$. Then by construction $S=f(V(\sigma))$ and $\sigma \cong \tau_{f(V(\sigma))} = \tau_{S}$. Let $f_{S}:\sigma \to \tau_{S}$ be the unique map defined as $f_{S}(x):=f(x)$ such that $f = \iota \circ f_{S}$. In particular $f_{S}$ is an isomorphism. Then given any graph $\gamma$ and every $\ell \in \mor(\gamma,\tau)$ such that $g \circ \ell = g^{\prime} \circ \ell$, we can define $u:=f_{S}^{-1} \circ \ell_{S}$, where $\ell_{S}:\gamma \to \tau_{S}$ is defined as $\ell_{S}(v):=\ell(v)$, which is the only map which satisfies $\ell = f \circ u$.

  \item let $f:\tau \to \gamma$ be the coequalizer of
  \begin{align*}
\sigma
\mathrel{\mathop{\rightrightarrows}^{h}_{h^{\prime}}} \tau \end{align*}
  Then we know that $\gamma \cong \tau^{\prime}$ where $(\tau^{\prime},g)$, with $g \in \mor(\tau,\tau^{\prime})$, is the coequalizer defined as in \cref{lemma:graph_coequalizers_exist}. Since $f = u \circ g$ where $u \in \Iso(\tau^{\prime},\gamma)$, both $u$ and $g$ are surjective maps on vertices and edges. It then follows that $f(E(\gamma))=E(\tau^{\prime})$.
  
 For the other direction, let $f:\tau \to \gamma$ be a surjective map, such that the induced map on the edges is also surjective. Consider the partition of the vertices of $\tau$ where each block is given by $f^{-1}(v)$ for all $v \in V(\gamma)$. Let $\sigma:= (V(\tau),\emptyset)$. Now let $h:\sigma \to \tau$ be defined as $h(v):=v$, for all $v \in V(\sigma)$ and $h^{\prime}:\sigma \to \tau$ defined on each block of the partition of the vertices of $\sigma$, $b = \{v_{1},...,v_{n}\}$ as $h^{\prime}(v_{i}):=v_{i}$ if the cardinality of the block $b$ is one, and $h^{\prime}(v_{i}):= v_{i+1}$ for $i=1,...,n-1$ and $h^{\prime}(v_{n}):= v_{1}$, otherwise. Consider the partition on the vertices of $\tau$ induced by the smallest equivalence relation which contains
 \begin{align}
 \label{eq:eq_rel}
 \bigcup_{i \in V(\sigma)}(h(i),h^{\prime}(i)).
 \end{align}
 This partition coincides with the one induced by the preimages of $f$. Indeed if $|f^{-1}(v)|=1$, we have $h(f^{-1}(v)) = f^{-1}(v) = h^{\prime}(f^{-1}(v))$ and $f^{-1}(v) \sim f^{-1}(v)$, and there is no other vertex in $\tau$ which is in relation with $f^{-1}(v)$. Now consider a block $\{b_{1},...,b_{n}\}$ where $|f^{-1}(v^{\prime})|>1$. By construction, since $h$ is the identity on the block and $h^{\prime}$ is a cycle on $\{b_{1},...,b_{n}\}$, it holds that $b_{1}\sim\cdots\sim b_{n}$, and again there is no other vertex in $\tau$ which is in relation with the elements of $f^{-1}(v^{\prime})$. Now a coequalizer exists for
 
  \begin{align*}
\sigma
\mathrel{\mathop{\rightrightarrows}^{h}_{h^{\prime}}} \tau. \end{align*}

using \cref{lemma:graph_coequalizers_exist}: the graph $\tau^{\prime}$ whose vertices are given by the partition induced by \eqref{eq:eq_rel} has no loops since by assumption there are no edges between vertices in the same block. We now show that $\gamma \cong \tau^{\prime}$. $\tau^{\prime}$ together with $g:\tau \to \tau^{\prime}$ (the map that is sending the vertex of $\sigma$ to its block with partition induced by  \eqref{eq:eq_rel}) is a coequalizer. The unique map $u:\tau \to \gamma$ such that $f = u \circ g$ is given by $f \circ g^{-1}$. We now show that $g \circ f^{-1}$ is well defined and a graph homomorphism as well, which means that $f \circ g^{-1}$ is an isomorphism. Well-definiteness comes from the fact the partitions induced by the preimages of $f$ and $g$ are equal: for all $x \in V(\gamma)$ there exists a unique $a \in V(\tau^{\prime})$ such that $g \circ f^{-1}(x) = a$. It is also a graph homomorphism, because given $\{x,x^{\prime}\} \in E(\gamma)$, there exists $\{y,y^{\prime}\} \in E(\tau)$ such that $y = f(x)$ and $y^{\prime} = f(x^{\prime})$ ($f$ is surjective on the edges of $\gamma$) and then we use that $g$ is a graph homomorphism, so that $\{g(y),g(y^{\prime})\} \in E(\tau^{\prime})$. Now let $\theta$ be a graph and let $\ell \in \mor(\tau,\theta)$ be such that  $\ell\circ h = \ell \circ h^{\prime}$. Now $s:=\ell \circ g^{-1}$ is the unique map which satisfies $\ell = s \circ g$. Then $v \in \mor(\gamma, \theta)$ defined as  $v = (\ell \circ g^{-1}) \circ (g \circ f^{-1}) = \ell \circ f^{-1}$, is the unique map which satisfies $\ell = v \circ f$.


  \end{enumerate}
\end{proof}

\begin{proposition}
The category of finite simple graphs is $(\RegEpi,\Mono)$-structured and $(\Epi,\RegMono)$-structured. 
\end{proposition}
\begin{proof}
We show that any graph homomorphism, $f:\sigma \to \tau$ can be written as $g\circ f^{\prime}$, where $f^{\prime}$ is a regular epimorphism and $g$ a monomorphism. Simply take $f^{\prime}: \sigma \to \tau\evaluatedAt{f(V(\sigma)),f(E(\sigma))}$ defined as $f^{\prime}(a):=f(a)$ and $g:\tau\evaluatedAt{f(V(\sigma)),f(E(\sigma))} \to \tau$ defined as $g(a):=a$. Moreover any homomorphism can be written as $h \circ f^{\prime}$ where $f^{\prime}$ is an epimorphism and $h$ a regular monomorphism. Let $f^{\prime}: \sigma \to \tau_{f(V(\sigma))}$ where $f^{\prime}(a):=f(a)$ and $h: \tau_{f(V(\sigma))}\to \tau_{f(V(\sigma))}$, with $h(a):=a$. We then use \cref{prop:rmer}.
\end{proof}

These factorizations are unique in the sense of \cref{proposition:unique_fact}. 
We can now deduce the following statement, which is used in the main text.
\begin{corollary}
  \label{cor:factorization}
  In the category of finite simple graphs
  \begin{enumerate}
    \item

  \begin{align*}
    \lrv{ \lrb{ \phi\in\Mono(\tau,\Lambda) \mid \Lambda_{\im \phi} \cong \sigma }  }
    =
    \frac1{\lrv{\Aut(\sigma)}}\
    \lrv{ \Mono(\tau,\sigma) \cap \Epi(\tau,\sigma) }\
    \lrv{ \RegMono(\sigma,\Lambda) },
  \end{align*}

  \item
  \begin{align*}
    \lrv{ \lrb{ \phi\in\Epi(\tau,\rho) \mid \rho\evaluatedAt{\phi(V(\tau)),\phi(E(\tau))} \cong \sigma }  }
    =
    \frac1{|\Aut(\sigma)|}\
    \lrv{ \RegEpi(\tau,\sigma) }\
    \lrv{ \Mono(\sigma,\rho) \cap \Epi(\sigma,\rho) }
  \end{align*}

    \item 
  \begin{align*}
  \lrv{ \lrb{ \phi\in\Hom(\tau,\Lambda) \mid \Lambda_{\im \phi} \cong \sigma }  }
  =
    \frac1{\lrv{\Aut(\sigma)}}\
    \lrv{ \Epi(\tau,\sigma) }\
    \lrv{ \RegMono(\sigma,\Lambda) }
  \end{align*}

    \item 
  \begin{align*}
    \lrv{ \lrb{ \phi\in\Hom(\tau,\Lambda) \mid \Lambda\evaluatedAt{\phi(V(\tau)),\phi(E(\tau))} \cong \sigma }  }
    =
    \frac1{\lrv{\Aut(\sigma)}}\
    \lrv{ \RegEpi(\tau,\sigma) }\
    \lrv{ \Mono(\sigma,\Lambda) }
  \end{align*}


  \end{enumerate}
\end{corollary}

\begin{proof}
We only show iii. Let $f \in \lrb{ \phi\in\Hom(\tau,\Lambda) \mid \Lambda_{\im \phi} \cong \sigma } $. Each $f$ can be factorized as $ f = \iota \circ f^{\prime}$, where $f^{\prime}:\tau \to \Lambda_{\im\phi}$, is defined as $f^{\prime}(v) := f(v)$ and $\iota: \Lambda_{\im\phi} \to \Lambda$ as $\iota(v):=v$. Let $f = m \circ e$ with $e \in \Epi(\tau,\sigma)$ and $m \in \RegMono(\sigma,\Lambda)$ be another factorization of $f$. Then, from \cref{proposition:unique_fact,} we know that there exists a unique $h \in \Iso(\Lambda_{\im\phi},\sigma)$ such that:

\begin{center}
\begin{tikzcd}
\tau \arrow[r, "f^{\prime}"] \arrow[d,"e"']
& \Lambda_{\im\phi} \arrow[d, "\iota"] \arrow[dl,"h"'] \\
\sigma \arrow[r, "m"']
&  \Lambda
\end{tikzcd}
\end{center}
the diagram commutes, which means that $e = h \circ f^{\prime}$ and $m = \iota \circ h^{-1}$. For any $e \in \Epi(\tau,\sigma)$ and $m \in \RegMono(\sigma,\Lambda)$ we have that $m \circ e = f \in \lrb{ \phi\in\Hom(\tau,\Lambda) \mid \Lambda_{\im \phi} \cong \sigma }$. Define the map, $\psi$:
\begin{align*}
   \psi: \Epi(\tau,\sigma)  \times \RegMono(\sigma,\Lambda) &\to \lrb{ \phi\in\Hom(\tau,\Lambda) \mid \Lambda_{\im \phi} \cong \sigma }\\
   \psi((e,m)) &:= m\circ e
\end{align*}
This map is clearly surjective and for each $f \in \lrb{ \phi\in\Hom(\tau,\Lambda) \mid \Lambda_{\im \phi} \cong \sigma }$, $|\psi^{-1}(f)| = |\Aut(\sigma)|$.
\end{proof}
\end{document}